\documentclass[12pt]{amsart}

\usepackage{amsmath,amsfonts,amssymb,amsthm}
\usepackage[mathscr]{eucal}
\usepackage{enumitem,kantlipsum}
\usepackage{color}
\usepackage{datetime2}
\usepackage[dvipsnames]{xcolor}
\usepackage{graphicx}
\usepackage[abs]{overpic}
\usepackage{caption}
\usepackage{float}
\usepackage{marginnote}
\usepackage[normalem]{ulem}
\usepackage[numbers]{natbib}
\renewcommand{\baselinestretch}{1.1}

\makeatletter
\@namedef{subjclassname@2020}{\textup{2020} Mathematics Subject Classification}
\makeatother

\makeatletter
\def\l@section{\@tocline{1}{0pt}{0pc}{}{}}
\def\l@subsection{\@tocline{2}{0pt}{0pc}{0em}{}}
\renewcommand{\tocsubsection}[3]{%
  \indentlabel{\@ifnotempty{#2}{\hspace*{2em}\makebox[2.2em][l]{%
    \ignorespaces#1 #2.\hfill}}}#3}
\makeatother

\newcommand{\no}{\nonumber}

\definecolor{blue}{RGB}{3, 8, 135}
\definecolor{red}{RGB}{188, 55, 84}

\RequirePackage[colorlinks=true, linkcolor = red,citecolor=blue,urlcolor=blue,backref=page]{hyperref}

\numberwithin{equation}{section}
\theoremstyle{definition}
\newtheorem{definition}{Definition}[section]
\theoremstyle{definition}

\newtheorem*{ntn}{Notation}

\newtheorem{remark}[definition]{Remark}

\theoremstyle{plain}
\newtheorem{theorem}[definition]{Theorem}

\newtheorem{lemma}[definition]{Lemma}
\newtheorem{cor}[definition]{Corollary}
\newtheorem{Prop}[definition]{Proposition}
\newtheorem{conjecture}[definition]{Conjecture}

\newcommand{\nc}{\newcommand}
\nc{\I}{{\mathbf 1}}
\newcommand{\remove}[1]{}
\nc{\bN}{{\mathbf N}}
\nc{\cB}{{\mathcal B}}
\nc{\cF}{{\mathcal F}}
\nc{\cX}{{\mathcal X}}
\nc{\cS}{{\mathcal S}}
\nc{\cC}{{\mathcal C}}
\nc{\cP}{{\mathcal P}}
\nc{\cN}{{\mathcal N}}
\nc{\R}{{\mathbb R}}

\nc{\E}{{\mathbb E}}
\nc{\C}{{\mathbb C}}
\newcommand{\Cd}{\mathbb{C}^d}
\nc{\N}{{\mathbb N}}
\nc{\Z}{{\mathbb Z}}
\nc{\BX}{{\mathbb X}}
\nc{\BM}{{\mathbb M}}
\nc{\BY}{{\mathbb Y}}
\nc{\bfW}{{\mathbf W}}
\nc{\bfY}{{\mathbf Y}}
\nc{\bfZ}{{\mathbf Z}}
\nc{\bx}{\mathbf{x}}
\nc{\bd}{\partial}
\def\1{\mathbf{1}}
\newcommand{\md}{\text{d}}
\nc{\BP}{\mathbb{P}}
\nc{\BE}{\mathbb{E}}
\nc{\BQ}{\mathbb{Q}}
\newcommand{\eps}{\varepsilon}
\newcommand{\beas}{\begin{eqnarray*}}
\newcommand{\eeas}{\end{eqnarray*}}
\newcommand{\bes} {\begin{equation*}}
\newcommand{\ees} {\end{equation*}}
\newcommand{\be} {\begin{equation}}
\newcommand{\ee} {\end{equation}}
\newcommand{\bea} {\begin{eqnarray}}
\newcommand{\eea} {\end{eqnarray}}
\newcommand{\bdy}{\partial}

\newcommand{\GL}{\operatorname{GL}}
\newcommand{\wt}{\widetilde}
\newcommand{\M}{\mathcal{M}}
\newcommand{\tr}{{\operatorname{T}}}
\newcommand{\rl}{\mathbb{R}}

\newcommand{\cont}{\mathcal{C}}

\newcommand{\de} {\delta}
\newcommand{\bde}{\boldsymbol\de}
\newcommand{\bbet}{{\boldsymbol\beta}}
\newcommand{\rea}{\operatorname{Re}}
\newcommand{\ima}{\operatorname{Im}}
\newcommand{\zbar}{\overline{z}}

\DeclareMathOperator{\BV}{{\mathbb Var}}
\DeclareMathOperator{\BC}{{\mathbb Cov}}

\newcommand\smpartl[2]{\frac{\partial{#1}}{\partial{#2}}}
\newcommand\partl[2]{\dfrac{\partial{#1}}{\partial{#2}}}
\newcommand\smsecpartl[3]{\frac{\partial^2{#1}}{\partial{#2}\text{ }\partial{#3}}}
\newcommand\secpartl[3]{\dfrac{\partial^2{#1}}{\partial{#2}\partial{#3}}}

\newcommand{\balp}{{\boldsymbol\alpha}}
\newcommand{\zt}{\zeta}

\newcommand{\gam}{\gamma}

\newcommand{\Om}{\Omega}

\emergencystretch=\maxdimen
\makeatletter
\g@addto@macro\bfseries{\boldmath}
\makeatother

\begin{document} 
\title{Volume Approximation of Strongly  $\C$-Convex Domains by Random Polyhedra} 
\author{Siva Athreya}
\address{Theoretical Statistics and Mathematics Unit,  Indian Statistical Institute,  Bangalore.}
\email{athreya@isibang.ac.in}
\author{Purvi Gupta}
\address{Department of Mathematics, Indian Institute of Science, Bangalore}
\email{purvigupta@iisc.ac.in}
\author{D. Yogeshwaran}
\address{Theoretical Statistics and Mathematics Unit,  Indian Statistical Institute,  Bangalore.}
\email{d.yogesh@isibang.ac.in}



\begin{abstract}
Polyhedral-type approximations of convex-like domains in $\C^d$ have been considered recently by the second author. In particular, the decay rate of the error in optimal volume approximation as a function of the number of facets has been obtained.  In this article,  we take these studies further by investigating polyhedra constructed using random points (Poisson or binomial process) on the boundary of a strongly $\C$-convex domain.  We determine the rate of error in volume approximation of the domain by random polyhedra, and conjecture the precise value of the minimal limiting constant.   Analogous to the real case,  the exponent appearing in the error rate of random volume approximation coincides with that of optimal volume approximation, and can be interpreted in terms of the Hausdorff dimension of a naturally-occurring metric space. Moreover, the limiting constant is conjectured to depend on the {\em M{\"o}bius-Fefferman measure}, which is a complex analogue of the Blaschke surface area measure. Finally, we also prove $L^1$-convergence, variance bounds, and normal approximation.  
\end{abstract}

\maketitle

\renewcommand{\baselinestretch}{0.8}\normalsize
         {\hypersetup{ linktocpage=true, linkcolor=blue}\small
               \tableofcontents}   \renewcommand{\baselinestretch}{1.1}\normalsize 

\section{Introduction}\label{S:intro}

Approximation of convex bodies by (convex) polyhedra is an important topic in convex geometry; see \citet{Gr93} and \citet{bronstein2008approximation}.  Motivated by both theoretical and practical considerations, this has led to numerous studies of approximation of convex bodies by random polyhedra; see \citet{reitzner2010random} and \citet{hug2013random}.  A fascinating feature of such studies is that for both deterministic and random approximations, the asymptotic rates and leading constants are rather well understood.  For instance, the least possible volume discrepancy between a convex body $K\subset\rl^d$ and the convex hull of $n$ points on its boundary decays like $v^{\sharp}_{\rl}(K)n^{-2/(d-1)}$, where the limiting constant $v^{\sharp}_{\rl}(K)$ is given in terms of the Blaschke affine surface area measure of $K$.  This was shown by \citet{mcclure1975polygonal} for $d =2$, and by \citet{GruasymptoticII} for all $d$.  On the other hand, the least possible volume discrepancy between $K$ and the convex hull of $n$ i.i.d.  random points chosen according to a continuous density on its boundary decays like $v^*_{\rl}(K)n^{-2/(d-1)}$, where $v^*_{\rl}(K)$ differs from $v^{\sharp}_\rl(K)$ only by a dimensional constant. Here, ``least" is to be understood as over all possible continuous densities on the bounday.  This was proven by \citet{schneider1988random} in the case of $d = 2$,  \citet{muller1990approximation} for balls in all dimensions, and by \citet{schutt2000random} and \citet{reitzner2002random} for general convex domains (under different regularity assumptions). Similar behaviour has also been observed for circumscribed polytopes in \cite{GruasymptoticII} and random circumscribed polytopes in \citet{boroczky2004approximation}.

While there are several notions of convexity in several complex variables,  the asymptotic analysis of polyhedral approximations in the complex setting is rather underexplored.  Deterministic volume approximations for two different classes of (convex-like) domains in $\Cd$ were investigated by the second author in \cite{Gu17} and \cite{Gu21}.  In both cases, it was either shown or conjectured that for a domain $D\subset\Cd$,  the error in volume approximation decays like $v^{\sharp}_\C(D)n^{-1/d}$, where $v^{\sharp}_\C(D)$ depends on complex analogues of the Blaschke surface area, i.e., the Fefferman hypersurface measure in the former case, and the M{\"o}bius--Fefferman hypersurface measure in the latter case. In this article, we initiate the study of random polyhedral approximations in the complex setting by exploring random analogues of the approximation considered in \cite{Gu21}. We determine the rate of approximation, and make a conjecture for the limiting constant; see Page \hyperlink{page.7}{7} for an an overview of the results. 

\subsubsection*{\bf Organization of the paper.} For the convenience of the reader, we summarize the central objects and the main results of this paper in Section~\ref{SS:Summary}.  We also provide some motivation and context, as well as describe the key proof ideas in this section.  Section~\ref{s:modelmain} expands on Section~\ref{SS:Summary} and contains the statements of the main results of this paper. The underlying domains and polyhedral sets are defined in detail in Section~\ref{SS:DomPoly}. The main results are stated in Section~\ref{s:poissonrandpoly}.  We end Section~\ref{s:modelmain} with some more remarks on the main theorems (Section~\ref{s:remltthms}). The proofs of the main results are spread over Sections~\ref{S:GeomEst} and ~\ref{s:proofsmain}. In the former, we collect some recurring notation and all the geometric (deterministic) ingredients needed for the proofs. Finally, the main (probabilistic) proofs are in Section~\ref{s:proofsmain}, where each subsection is devoted to the proof of each main result. In Appendix \ref{s:probtool}, we recall some well-known results on Poisson and binomial processes that are used in our proofs.  In Appendices \ref{SS:curvature} and \ref{SS:Models}, we collect all the background material on the geometry of $\C$-convex domains that plays a role in Section~\ref{S:GeomEst}. In an attempt to make the article accessible to both probabilists as well as complex analysts, we have provided more background details than is customary in either of the two fields. 

\subsubsection*{\bf Convention on constants.} Throughout the paper, $c$ and $C$  will denote generic positive constants whose values may change from line to line. The dependence of these constants on specific parameters of the problem will be indicated via subscripts as in $C_{g,D}, C_D$, etc.

\subsection{Context and Contribution}
\label{SS:Summary}

Convex geometry has had fruitful interactions separately with probability and several complex variables. We envisage that the joint interaction between the three subjects could be equally exciting.  Apart from mere mathematical curiosity, the motivation behind considering $\C$-convexity is to enlarge the class of convex-like domains for which approximation problems can be studied. In geometric terms, $\C$-convexity is the natural analogue of classical convexity in complex projective geometry. Although there is a developing literature on random polytopes in non-Euclidean geometry (see, for example, \citet{besau2020asymptotic,besau2020random}),  our work is fundamentally different in that the underlying notion of convexity itself is non-Euclidean (complex projective). 

\subsubsection*{\bf Strongly $\C$-convex domains.} Recall that for a $\cont^2$ domain $K\subset\rl^d$, convexity and strong convexity are characterized by the semi-positive definiteness and positive definiteness, respectively, of the second fundamental form of $bK$, the boundary of $K$ (viewed as a smooth manifold), at all points of $bK$. Analogous to this, a $\cont^2$ domain $D\subset\Cd$ is {\em $\C$-convex} ({\em strongly $\C$-convex}), if for all $w\in bD$, 
the restriction of the second fundamental form of $bD$ to  $\mathcal H_wbD$ is semi-positive definite (positive definite), where $\mathcal H_wbD$ is the maximal complex subspace of the real tangent space of $bD$ at $w$ when viewed as a subspace of $\Cd$. We refer to $\mathcal H_wbD$ as  the {\em complex tangent space} to $bD$ at $w$. In this paper, we work with bounded strongly $\C$-convex domains in $\Cd$. However, for the sake of convenience and illustration, the examples given in Figures \ref{F:saddle}, \ref{F:hyperbola} and \ref{F:parabola} are of unbounded strongly $\C$-convex domains. The planar case ($d=1$) appears to be simpler since complex tangent spaces reduce to points, so there is no complex geometry to speak of; compare Figure \ref{F:d=1} ($d=1$) and Figures \ref{F:hyperbola} and \ref{F:parabola} ($d=2$). Thus, the case $d = 1$ calls for a separate investigation, and we exclusively deal with the case $d>1$ in this paper. 

\begin{figure}
\begin{center} 
\begin{overpic}[grid=false,tics=10,scale=0.5]{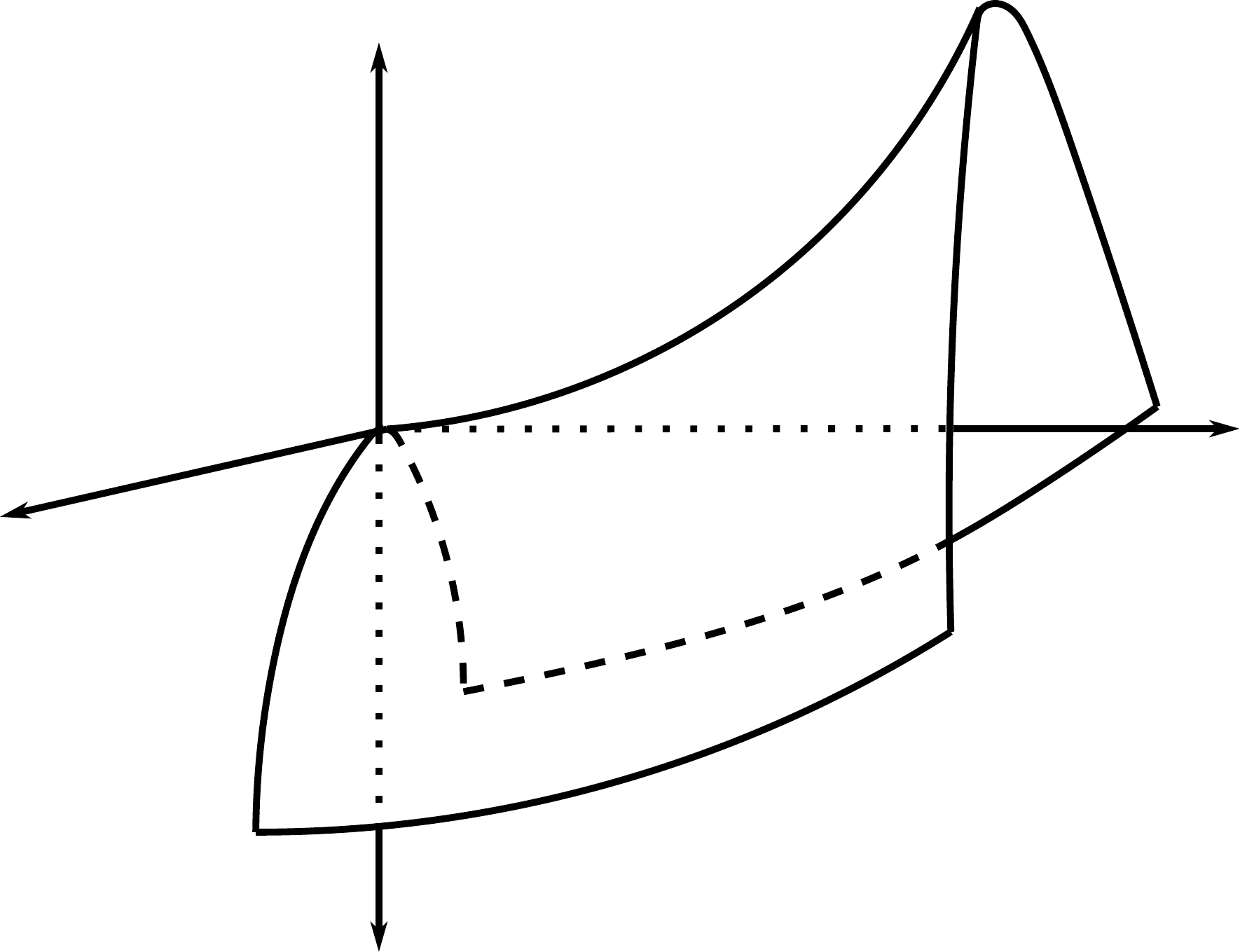} 
\put(240,90){ $|z_1|$}
\put(5,80){ $x_2$}
\put(83,172){$y_2$}
\put(200,58){$bD$}
\end{overpic}
\medskip
\caption{A non-convex strongly $\C$-convex (unbounded) domain in $\C^2$: $D=\{y_2> |z_1|^2-x_2^2\}$.}\label{F:saddle}
\end{center}

\end{figure}
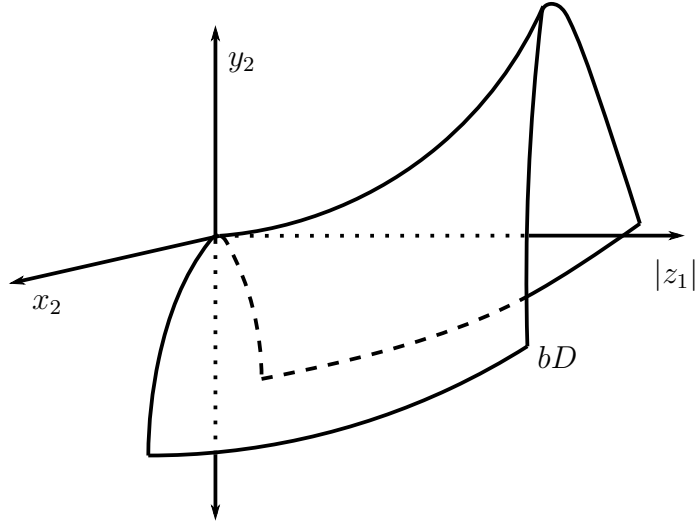

We postpone a broader discussion on $\C$-convexity to Section~\ref{SS:DomPoly}, but note that (real) convex sets in $\Cd\cong\rl^{2d}$ are always $\C$-convex but the converse is not true; see Figure \ref{F:saddle}. While real convexity (and the class of real convex polyhedra) is invariant under real affine transformations,  $\C$-convexity is invariant under the (larger) group of complex projective transformations, making it a suitable notion of convexity for complex  projective geometry. From this point of view, we shortly introduce a class of polyhedral sets that is invariant under the intermediate group of complex affine transformations. In complex analysis, $\C$-convexity plays a significant role in questions of duality, interpolation, integral representations of holomorphic functions and metric geometry; see \citet{APR}, \citet{Ba16}, \citet{LaSt} and \citet{Le1984}, for instance. 
  
\subsubsection*{\bf Induced polyhedra.} In the real setting, given a finite set $\varphi=\{w^1,...,w^n\}$ on the boundary of a real convex domain $K\subset\rl^d$, the simplest construction of a convex polyhedron is to take the convex hull $P_{\text{i}}(\varphi)$ of the point set $\varphi$, i.e., $P_{\text{i}}(\varphi)$ is the intersection of all the convex sets containing $\varphi$. In the complex setting, since the intersection of two $\C$-convex sets need not be $\C$-convex, there is no notion of a ``$\C$-convex hull"; see \cite[Section~2.2]{APR}. A more natural construction of polyhedra is suggested by the following construction of circumscribed polyhedra in the real setting. Given a finite set $\varphi=\{w^1,...,w^n\}$ on the boundary of a strongly convex domain $K\subset\rl^d$, let $P_{\text{c}}(\varphi)$ denote the the intersection of those $n$ half-spaces determined by the tangent planes to $bK$ at $w^1,...,w^n$ that contain $K$. Analogous to this construction, a class of ``complex" polyhedra are defined in \cite{Gu21} as follows.  Given a strongly $\C$-convex domain $D\subset\Cd$, a {\em source set} $\varphi=\{w^1,...,w^n\}\subset bD$ and a {\em size function} $\bde:bD\rightarrow (0,\infty)$, let $P(\varphi;\bde)$ be the union of those connected components of 
\be\label{E:poly}
		\bigcap_{j=1}^{n}\left\{z\in \Cd:\text{dist}(z,\mathcal H_{w^j}bD)>\bde(w^j)\right\},
\ee
that intersect $D$, where $\mathcal H_wbD$ is the complex tangent space to $bD$ at $w$; see Section~\ref{SS:DomPoly}. The set \eqref{E:poly} is {\em $\C$-linearly convex} and hence, $P(\varphi;\bde)$ is {\em weakly $\C$-linearly convex}; see \cite[Section 2.1]{APR}. We call $P(\varphi;\bde)$ an {\em induced polyhedron with at most $n$ facets}. Though ``$D$-induced polyhedra'' may be a more precise term,  we drop the domain-dependence for convenience. In comparison,  $P_{\text{i}}(\varphi)$ and $P_{\text{c}}(\varphi)$ have at most $n$ vertices and at most $n$ facets, respectively. Geometrically, the role of half-spaces is played by complements of `tubes' of radii $\bde(\boldsymbol\cdot)$ with central axes $\mathcal H_{\boldsymbol\cdot}bD$. See Figure \ref{F:d=1} for illustrations of induced polyhedra in dimension one. In higher dimensions, induced polyhedra are hard to draw, but see Figures~\ref{F:hyperbola} and \ref{F:parabola} for illustrations of these `tubes' that generate the polyhedra.
\begin{figure}[H]
\begin{center} 
\begin{overpic}[grid=false,tics=10,scale=0.45]{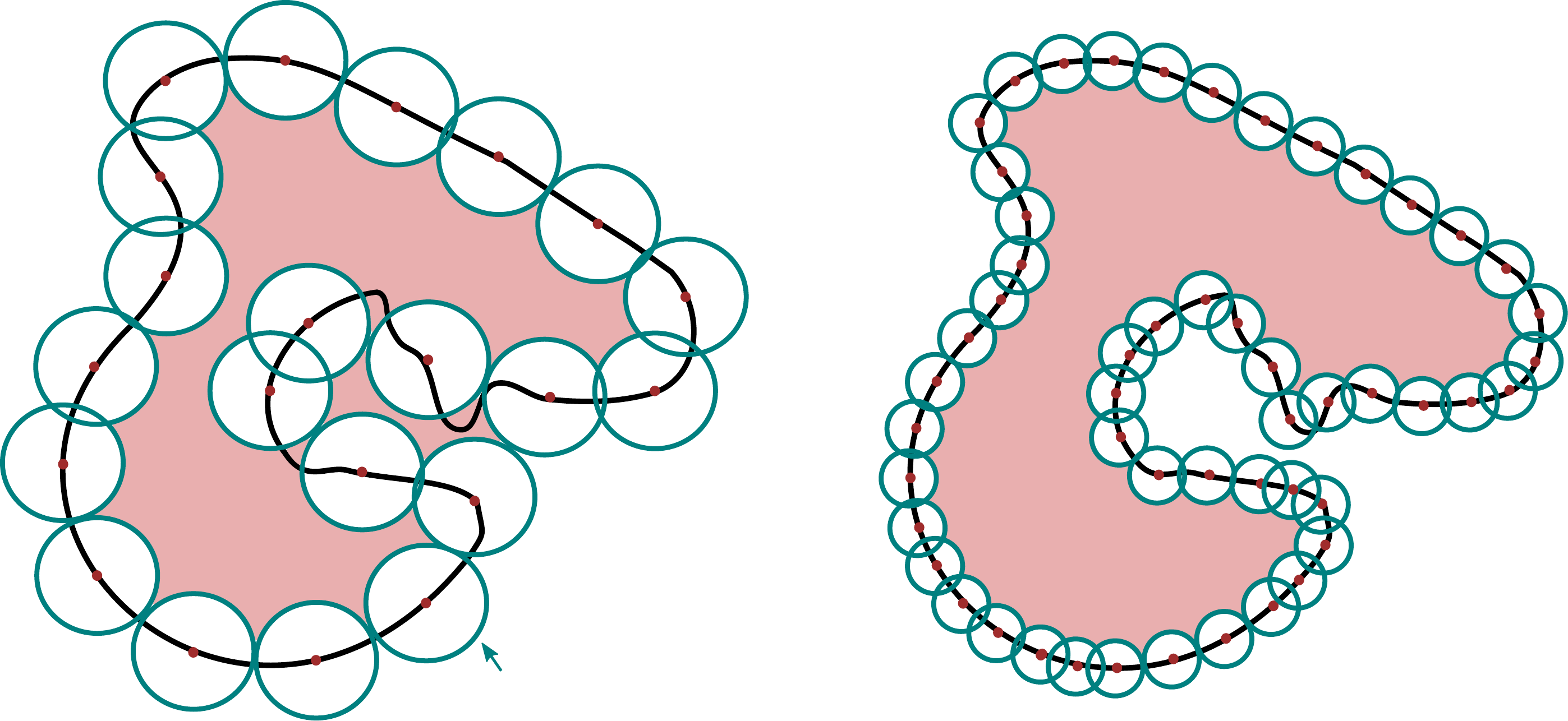} 
\put(65,-15){$(a)$}
\put(255,-15){$(b)$}
\put(100,20){\small\textcolor{Maroon}{$w$}}
\put(108,4){\textcolor{teal}{\footnotesize$\{z:\text{dist}(z,\mathcal H_w(bD))=\de\}$}}
\end{overpic}
\bigskip
\caption{Induced polyhedra in $\C$: the domain $D$ is the region bound by the black curve; the induced polyhedron is the union of regions shaded in pink.\\ $(a)$ A disconnected and non-contained induced polyhedron with $21$ facets.\\
$(b)$ A connected and contained induced polyhedron with $52$ facets.}\label{F:d=1}
\end{center}
\end{figure}

\begin{figure}[H]
\begin{center} 
\begin{overpic}[grid=false,tics=10,scale=0.5]{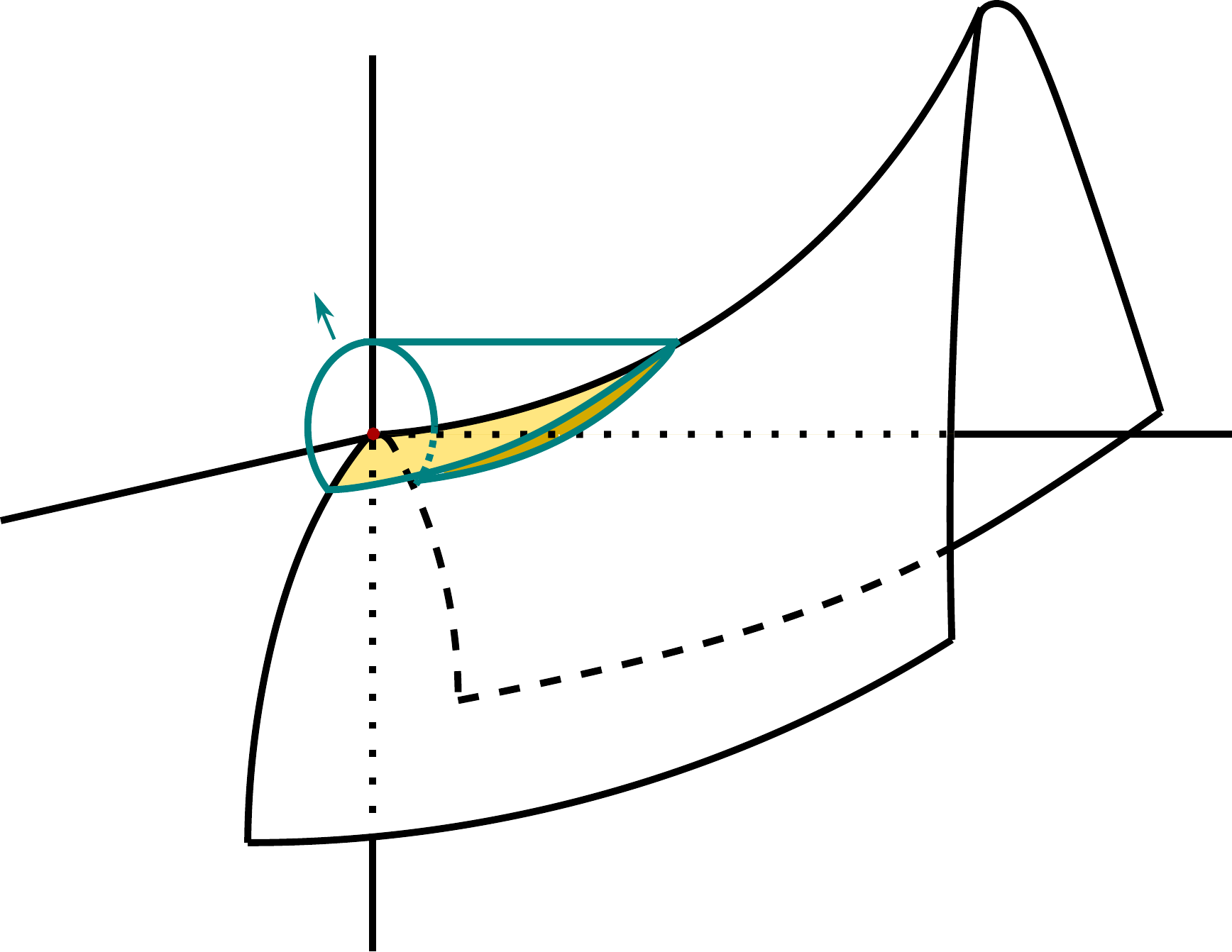} 
\put(240,95){ $|z_1|$}
\put(240,115){ \small$\mathcal H_w(bD)$}
\put(66,108){\small\textcolor{Maroon}{$w$}}
\put(5,80){ $x_2$}
\put(83,172){$y_2$}
\put(-60,140){\footnotesize\textcolor{teal}{$\{z\in\overline D:\text{dist}(z,\mathcal H_w(bD))=\de\}$}}
\put(200,58){$bD$}
\end{overpic}
\medskip
\caption{A ``cut" $C(w;\de)=\{z\in D:\text{dist}(z,\mathcal H_w(bD)<\de)\}$ (region bound by the teal tube and $bD$) and a ``cap" $S(w;\de)=\{z\in bD:\text{dist}(z,\mathcal H_w(bD)\leq\de)\}$ (in yellow) in the domain   $D=\{y_2>|z_1|^2-x_2^2\}$ at $w=(0,0)$. The interior of the complement of $C(w;\de)$ in $\C^2$ is an induced polyhedron with one facet.}\label{F:hyperbola}
\end{center}
\end{figure}

\begin{figure}[H]
\begin{center} 
\begin{overpic}[grid=false,tics=10,scale=0.45]{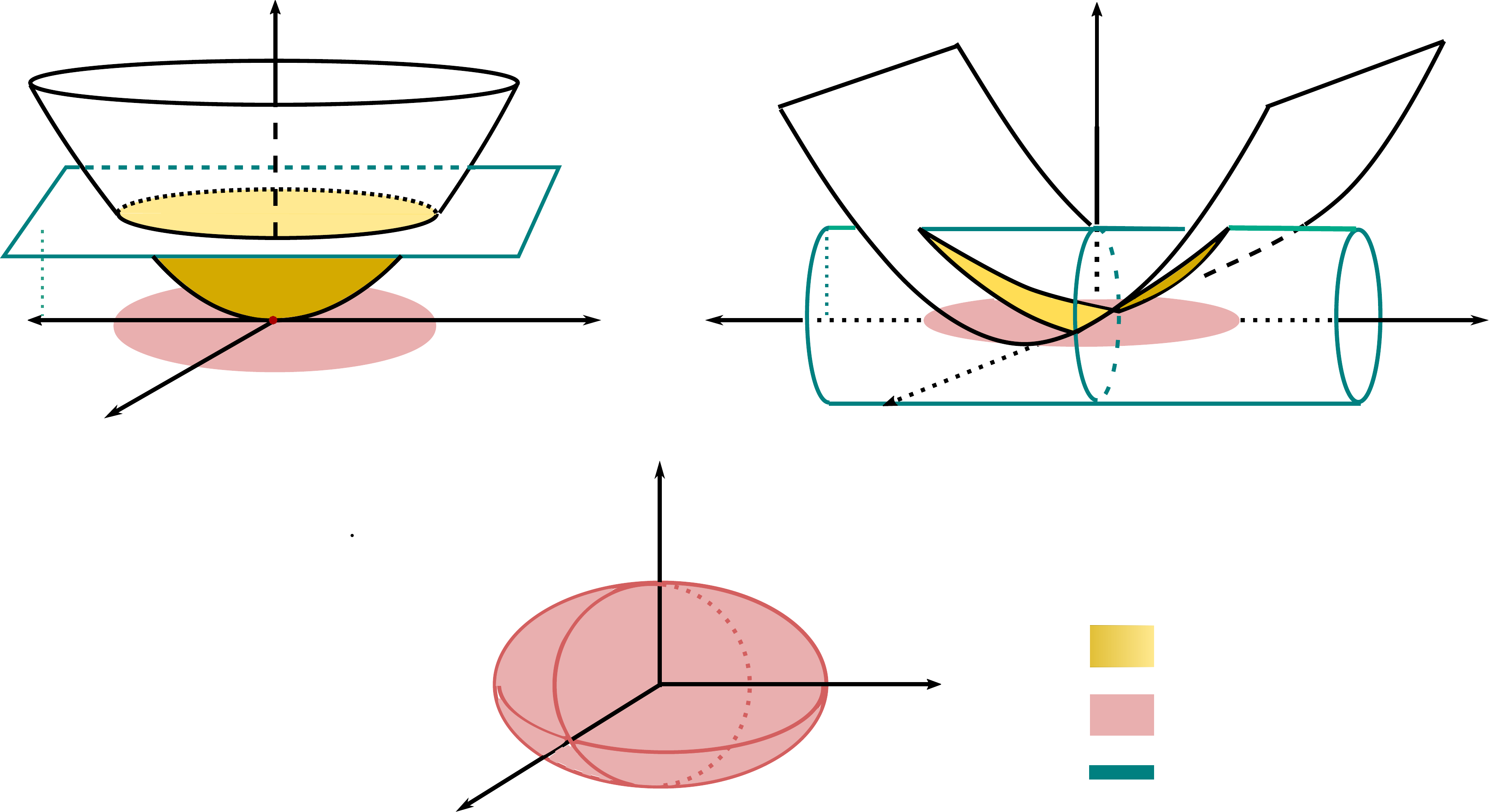} 
\put(70,240){$y_2$}
\put(130,240){$bD$}
\put(15,115){ $x_1$}
\put(180,140){ $y_1$}
\put(100,120){ $\mathcal H_w(bD)$}
\put(82,142){\textcolor{Maroon}{$w$}}
\put(70,100){ $(a)$}
\put(15,160){\footnotesize\textcolor{teal}{$\delta$}}
\put(323,240){$y_2$}
\put(400,240){$bD$}
\put(330,100){ $(b)$}
\put(260,115){ $x_2$}
\put(450,140){ $y_1$}
\put(425,160){ \small$\mathcal H_w(bD)$}
\put(257,160){\footnotesize\textcolor{teal}{$\delta$}}
\put(185,100){ $x_2$}
\put(202,50){ \small$\delta$}
\put(217,43){ \small$\sqrt{\delta}$}
\put(175,35){ \small$\sqrt{\delta}$}
\put(290,30){$y_1$}
\put(130,-5){$x_1$}
\put(195,-18){ $(c)$}
\put(335,65){\small{Key}}
\put(360,47){\small{$S(w;\de)$}}
\put(360,25){\small{$\wt S(w;\de)$}}
\put(360,9){\small{$\{\text{dist}(z,\mathcal H_w(bD))=\de\}$}}
\end{overpic}
\vspace{20pt}
\caption{$3$D slices of the domain $D=\{y_2>|z_1|^2\}$ near $w=(0,0)$; $\mathcal H_w(bD)=\{z_2=0\}$. Here, $S(w;\de)$ is the ``cap" $\{z\in bD:\text{dist}(z,\mathcal H_w(bD))\leq \de\}$ and $\wt S(w;\de)$ is the projection of $S(w;\de)$ onto $\mathcal T_w(bD)=\{y_2=0\}$. \\
$(a)$ $D$ in the $x_2=0$ plane,\\
$(b)$ $D$ in the $x_1=0$ plane,\\
$(c)$ $\wt S(w;\de)$ in the $y_2=0$ plane, i.e., in $\mathcal T_w(bD)$.}\label{F:parabola}
\end{center}

\end{figure}

  Induced polyhedra are special cases of what are known as {\em analytic polyhedra} in the complex analysis literature. However, as pointed out in \cite{Gu17,Gu21}, the full class of analytic polyhedra is not suitable for the kind of approximation problems we are interested in. The size function may seem as an artefact of the complex setting since it is absent in the construction of $P_\text{c}(\varphi)$ above. However, convex polyhedra with nontrivial size functions have appeared in the real setting; see \citet{ludwig1999asymptotic} and \citet[Section 5]{glasauer1996asymptotic}.

\subsubsection*{\bf Random polyhedra and the metric of approximation.} We wish to discuss a random analogue of the following question investigated in \cite{Gu21}: given a bounded strongly $\C$-convex domain $D\subset\Cd$, determine the asymptotics as $n\rightarrow\infty$ of 
	\be
	\label{e:VnD}
	v_n(D) := \inf \{\lambda(D\setminus P(\varphi;\bde)): \varphi=\{w^1,...,w^n\}\subset bD,\   \bde:bD\rightarrow (0,\infty),\ P(\varphi;\bde)\subset D \},
	\ee
where $\lambda$ is the Lebesgue measure on $\Cd$. We know from \cite{Gu21} that 
	\be\label{E:opt}
		v_n(D) \sim v^{\sharp}_\C(D) n^{-1/d}\qquad  \text{as}\ n \to \infty,
	\ee
for some domain-dependent constant $v^{\sharp}_\C(D)$. 

We randomize $P(\varphi;\bde)$ by taking $\varphi$ to be random point sets.  Let $\{W^j\}_{j \geq 1}$, be i.i.d.  points sampled according to a probability density $f(\cdot)$ with respect to the surface area measure $\sigma$ on $bD$.  We either take the random point set to be $\varphi = \eta_n := \{W^1,\ldots,W^{N_n}\}$ (Poisson process) where $N_n$ is an independent Poisson($n$) random variable, or $\varphi = \beta_n := \{W^1,\ldots,W^n\}$ (binomial process). We shortly comment on the choice of $\bde$.  

As with \eqref{e:VnD}, we use the Lebesgue volume $\lambda$ to measure the error of approximation of $D$ by $P(\varphi;\bde)$. Since we are interested in randomly contained polyhedra but a  random $P(\varphi;\bde)$ (in the sense defined above) may not be contained in $D$, we choose the following metric whereby the polyhedron is penalised if it is not contained in $D$ (or is unbounded): 
\begin{equation}
\label{e:volmetric}
 \delta_V\big(D,P(\varphi;\bde)\big) :=  \lambda(D \setminus P(\varphi;\bde))\1(P(\varphi;\bde) \subset D)  + \lambda(D) \1(P(\varphi;\bde)\subsetneq D).
 \end{equation}
In the real case, such a  penalty is unnecessary for inscribed polyhedra, but is used for randomly circumscribed polyhedra; see \citet{boroczky2004approximation}. An alternative would be the symmetric difference of $D$ and $P(\varphi;\bde)$, as considered by  \citet{ludwig1999asymptotic}, \citet{ludwig2006approximation} and  \citet{grote2018approximation} in the real case, which is a subject for future investigation.

\subsubsection*{\bf Overview of the results.} The first step towards understanding the random equivalent of $v_{n}(D)$ is to investigate the asymptotics of $\delta_V(D,P(\eta_n;\bde_n))$ and $\delta_V(D,P(\beta_n;\bde_n))$ for a given sequence of functions $\bde_n$.  Since we want to study randomly contained polyhedra, we assume that the polyhedra are contained in $D$ with high probability, i.e., $\BP(P(\eta_n;\bde_n) \subset D) \to 1$ as $n \to \infty$. For this, it is essential to assume that $\bde_n = g(z)(\frac{\log n}{n})^{1/d}$ for a function $g : bD \to (0,\infty)$; see Proposition \ref{lem:rdelta}. For technical reasons, we assume that $g$ is $\alpha$-H{\"o}lder for some $0<\alpha<1$.   

Our first main result, Theorem \ref{thm:cippoi}, is that when $\varphi_n$ is either $\eta_n$ (Poisson) or $\beta_n$ (binomial), and $\bde_n$ is as above,   
\begin{equation}
\label{E:cvgprobintro}  \Big( \frac{n}{\log n} \Big)^{1/d} \delta_V(D,P(\varphi_n;\bde_n)) \overset{p}{\rightarrow} \int_{bD} g(z) \md \sigma(z),  
\end{equation}
as $n\rightarrow\infty$, where $\overset{p}{\rightarrow}$ denotes convergence in probability.  Under further assumptions on the rate of convergence of the containment probabilities $\BP(P(\varphi_n;\bde_n)\subset D)$, we 
	\begin{itemize}
		\item [$1.$] strengthen $\eqref{E:cvgprobintro}$ to $L^1$-convergence (Theorem~\ref{t:expvarpoi}), 
		\item [$2.$] obtain variance (upper and lower) bounds for $\delta_V(D,P(\eta_n;\bde_n))$, and variance upper bounds for $\delta_V(D,P(\beta_n;\bde_n))$ (Theorem~\ref{t:expvarpoi}), and
		\item [$3.$] bound the rate of normal approxiation of $\delta_V(D,P(\eta_n;\bde_n))$ (Theorem \ref{t:cltpoi}). 
	\end{itemize}
Our variance upper and lower bounds for $\de_V(D,P(\eta_n;\bde_n))$ differ only by logarithmic factors. We are unable to obtain good variance lower bounds for $\delta_V(D,P(\beta_n;\bde_n))$. We prove a normal approximation result for $\delta_V(D,P(\beta_n;\bde_n))$ assuming a suitable lower bound on the variance; see Theorem \ref{t:cltpoi}.

\subsubsection*{\bf Optimal random approximations: a conjecture} In problems of approximation, it is of interest to compare the rate of random approximations to that of optimal approximations. Note that if we ignore the domain-dependent constant in \eqref{E:cvgprobintro}, we obtain that the approximation rate is $(n^{-1}\log n)^{1/d}$, which differs from the rate of optimal approximation in \eqref{E:opt} by a logarithmic factor. 
This additional factor is necessitated by the need for the containment condition on the underlying random polyhedra. In fact, once the containment condition is translated into a coverage condition on the boundary (Lemma~\ref{L:inc=cov}), the additional logarithmic factor is no longer surprising in view of coverage results of \citet{flatto1977random},  \citet{hall1985coverage} and \citet{janson1986random} (although these do not apply directly to our setting).  

Now, we compare the limiting constants in \eqref{E:opt} and \eqref{E:cvgprobintro}. At first glance,  it appears that the choice of density $f$ is absent in the limiting constant in \eqref{E:cvgprobintro}. But, it is implicit in the condition that $\BP(P(\varphi_n;\bde_n) \subset D) \to 1$ as $n \to \infty$. This raises the questions of: $(a)$ the best possible limiting constant in \eqref{E:cvgprobintro}, i.e. ,
$$v_D^*(f)=\inf \left\{ \int_{bD} g(z) \md \sigma(z) : \mbox{$g$ is continuous and} \,  \lim_{n \to \infty} \BP(P(\varphi_n;\bde_n) \subset D) = 1 \right\}$$ 
for a fixed continuous density $f$,  and $(b)$ the least possible constant, $v_{\C}^*(D)$, over all such $f$. The proofs of \cite{flatto1977random, hall1985coverage,janson1986random}  for coverage results in the real setting suggest the following conjecture. 
\begin{conjecture}\label{conj}
Let $D\subset\Cd$ be a bounded strongly $\C$-convex domain, and $f:bD\rightarrow(0,\infty)$ be a continuous probability density function. Then, 
	\be\label{E:conj}
	v^*_D(f)=(h_d\kappa_{2d-2})^{-1/d}\int_{bD}
\frac{(16\nu_D(z))^{\frac{1}{2d}}}{f(z)^{\frac{1}{d}}}\,d\sigma(z),
\ee
where $\nu_D$ is the complex-restricted curvature of $bD$ (see Definition~\ref{D:curvature}), $\kappa_d$ is the Lebesgue volume of the unit ball in $\rl^d$, and $h_d=2\int_0^{\pi/2}\cos^d\!\theta\:\md\theta$. 
\end{conjecture}
Note that the right-hand side of \eqref{E:conj} is minimized when 
	\bes
		f=(16\nu_D)^{\frac{1}{2d+2}}\left(
	\int_{bD}(16\nu_D)^{\frac{1}{2d+2}}d\sigma\right)^{-1},
	\ees 
i.e., $f\md \sigma$ is the so-called normalized M{\"o}bius-Fefferman measure, which appears in \cite{Gu21} in relation to $v_{\C}^\sharp(D)$, and thus $v_{\C}^*(D) = (h_d\kappa_{2d-2})^{-1/d}\left(\int_{bD}(16\nu_D)^{\frac{1}{2d+2}}
	d\sigma\right)^{\frac{d+1}{d}}.$
This is expected to differ from the constant $v^{\sharp}_\C(D)$ in \eqref{E:opt} only by a dimensional constant. See Section \ref{s:remltthms} for a remark on this conjecture.

\subsubsection*{\bf Comparison with real random polyhedra.} Needless to say, the literature on random real polyhedra is very large owing to the numerous ways of constructing convex polyhedra. Nevertheless, there are some recurring features, within the context of which we place our results. 

As noted above, the exponent of $n$ in the rate of optimal approximations \eqref{E:opt} and random approximations \eqref{E:cvgprobintro} is the same, i.e., $2/2d$. This matching of exponents in the optimal and random rates is seen in the (real) case of approximations by $P_{\text i}(\varphi)$ and $P_{\text c}(\varphi)$ as well, where the exponent is $2/(d-1)$. The quantities $2d$ and $d-1$ in the exponents $1/d$ in the complex setting and $2/(d-1)$ in the real setting, respectively, can be unified when interpreted in terms of the Hausdorff dimensions of naturally-occuring metric space in these problems; see \cite[Section 1]{Gu17}. The appearance of the logarithmic factor in \eqref{E:cvgprobintro} is a point of departure from the case of $P_{\text i}(\varphi)$ and $P_{\text c}(\varphi)$. However, similar logarithmic factors appear even in the real setting in the case of polyhedral constructions for which containment needs to be imposed; e.g.,  see \cite[Theorems 4 and 5]{glasauer1996asymptotic}. Finally, though the optimal and random limiting constants in the complex setting ($v^\sharp_\C(D)$ and $v^*_D$) are not as well-understood as in many models of real polyhedra, we have some evidence to believe that they too will only differ by a dimensional constant, as in the real case, e.g., compare the constants in \cite{GruasymptoticII} (best inscribed and circumscribed) and \cite{schutt2000random, reitzner2002random,boroczky2004approximation} (random inscribed and circumscribed). 

So far, we have only placed the convergence in probability result (Theorem~\ref{thm:cippoi}) into context. Our expectation, variance and normal approximation results also compare well with what is known in the real setting, but we postpone the details to Section~\ref{s:remltthms}.  

\subsubsection*{\bf Overview of the proofs.} The two key ideas underpinning our proofs are that the complement of an induced polyehdron in $D$ can be expressed as a union of finitely many `cuts' i.e., sets of the form $\left\{z\in D:\text{dist}(z,\mathcal H_{w}bD) < \bde(w)\right\}$, 
and that a cut is well-approximated by a `tubular cut'. The volume of a union of tubular cuts is then esitmated via a Weyl-type tube formula for non-uniform tubes due to  \citet{roccaforte2013volume}. This requires analysis of an appropriate depth function and the so-called visibility regions. The results on convergence in probability and $L^1$-convergence follow from these estimates (and suitable assumptions on the containment probabilities $\BP(P(\varphi_n;\bde_n) \subset D)$). This approach is in the spirit of \cite{boroczky2004approximation}, where Steiner's tube formula is used to study volume approximations of randomly circumscribed (real) polyhedra. For variance bounds and normal approximation,  we use the Efron--Stein,  Poincar\'e, and second-order Poincar\'e inequalities along with first-order estimates on the volumes of cuts and tubular cuts.   Some of these geometric estimates were already available in the literature (Lemmas~\ref{L:deffn}, \ref{L:Model}, \ref{L:Lest}, \ref{L:quasimetric} and \ref{L:Vitali}), while the rest are either new or are quantitatively precise versions of what was known previously.  We believe that these estimates are of independent interest in the study of $\C$-convex domains.

\section{Limit theorems for random polyhedra}
\label{s:modelmain}

 In this section,  we formally introduce our model of random polyhedra and state our main results.  In Section~\ref{SS:DomPoly}, we define the class of polyhedral objects under consideration. In Section~\ref{s:poissonrandpoly}, we state the main limit theorems for the volume approximation.  In Section \ref{s:remltthms},  we make some observations related to the main theorems.   For reference, a list of notation (such as for complex derivatives and gradients) is provided at the beginning of Section~\ref{S:GeomEst}.

\subsection{Polyhedra induced by strongly $\C$-convex domains}\label{SS:DomPoly} Recall that a set $K\subset \rl^d$ is (real) convex if and only if the intersection of $K$ with every real line is connected. In analogy with this, a subset $D\subset\Cd$ is said to be {\em $\C$-convex} if for every {\em complex} affine line $L$, $D\cap L$ is connected and simply connected, i.e., the complement of $D$ in the extended line $L\cup\{\infty\}$ is connected.  A domain (i.e., a connected and open set) in $\C$ is $\C$-convex precisely when it is simply connected. In higher dimensions, while all $\C$-convex domains in $\Cd$ are topologically trivial, i.e., homeomorphic to the unit ball in $\Cd$, the converse is far from true. $\C$-convex sets satisfy many complex analogues of properties displayed by convex sets.  As mentioned in Section~\ref{S:intro}, $\C$-convexity is invariant under projective (or linear fractional) transformations, i.e., mappings of the form \eqref{E:LFT}. Akin to the polar of a convex set, one may define the dual complement of a $\C$-convex set; see \cite[Def. 2.1.1]{APR}. Moreover, any open or compact $\C$-convex set $D\subset\Cd$ satisfies a complex version of the Hahn--Banach separation property: $\Cd\setminus D$ is a union of complex hyperplanes in $\Cd$. We direct the reader to \cite{APR} for a more detailed discussion of $\C$-convexity and related notions. 

We now consider $\cont^2$ domains, i.e., those of the form $D=\{z\in\Cd:\rho(z)<0\}$, where $\rho$ is a $\cont^2$ function on $\Cd$ whose gradient is nonvanishing on $bD$. Here, $\rho$ is called a {\em defining function of $D$}. At any  point $w\in bD$, the {\em real tangent space} of $bD$ at $w$ is the $\rl$-vector space
	\bes
			\mathcal T_wbD = \{v\in\Cd:\rea\left<\bdy\rho(w),v\right>=0\},
	\ees
while the {\em complex tangent space} of $bD$ at $w$ is the $\C$-vector space
	\bes
			\mathcal H_wbD = \{v\in\Cd:\left<\bdy\rho(w),v\right>=0\}.
	\ees
We note that the complex tangent space at $w$ is the maximal subspace of the real tangent space at $w$ that is invariant under the action of multiplication by $i$. 

Recall that $D\subset\Cd$ as above is real convex if and only if the Hessian of $\rho$ when restricted to $\mathcal T_wbD$ is semi-positive definite for every $w\in bD$. In the same way, $D\subset\Cd$ as above is $\C$-convex if and only if the Hessian of $\rho$ when restricted to $\mathcal H_wbD$ is semi-positive definite for every $w\in bD$. This can be expressed in terms of complex derivatives as follows:
	\bes
	Q_w(v,v) := \sum_{j,k=1}^d\secpartl{\rho}{z_j}{\zbar_k}(w)v_j\overline v_k+
			\rea \sum_{j,k=1}^d\secpartl{\rho}{z_j}{z_k}(w)v_jv_k
\geq 0\quad \forall w\in b D, v\in \mathcal H_wbD.
	\ees
	\begin{definition} Let $D\subset\Cd$, $d\geq 2$, and $\rho$ be as above. Then $D$ is said to be {\em strongly $\C$-convex} if, for each $w\in bD$, $Q_w$ is positive definite when restricted to $\mathcal H_wbD$. 
\end{definition}

Any such $D$ has the property that the complex tangent hyperplane $w+\mathcal H_wbD=\{z\in\Cd:\left<\bdy\rho(w),w-z\right>=0\}$ at any $w\in bD$ intersects $\overline D$ only at $w$; see \cite[Sec. 2.5]{APR}.
By the continuity of the second-order derivatives, strong $\C$-convexity is equivalent to the existence of a continuous function $c:bD\rightarrow(0,\infty)$ such that
	\be\label{E:strcvx1}
		Q_w(v,v)\geq c(w)||v||^2\quad \forall w\in b D, v\in \mathcal H_wbD.
	\ee 
Thus, when $D$ is bounded, strong $\C$-convexity is equivalent to {\em strong $\C$-linear convexity} (in the sense of \cite{LaSt}), i.e., there exists a $\tilde c>0$ such that
	\be\label{E:strcvx2}
	\left|\left<\bdy\rho(w),w-z\right>\right|\geq \tilde c||w-z||^2 
		\quad \forall (z,w)\in\overline D\times b D.
	\ee
See Proposition 3.1 (and the subsequent remarks) in \cite{LaSt}. In general, for a fixed $w\in bD$, \eqref{E:strcvx1} implies \eqref{E:strcvx2} only for $z$ near $w$, but since the left hand side of \eqref{E:strcvx2} is positive on $\overline D\setminus\{w\}$, the compactness of $\overline D$ grants \eqref{E:strcvx2} for all $z\in \overline D$. In geometric terms, $D$ is strongly $\C$-convex if $bD$ is locally equivalent to a strongly real convex hypersurface in $\Cd$ via a linear fractional transformation. For more insight into the geometry of such domains, see Lemma~\ref{L:Model}, Lemma~\ref{L:quasimetric}, and Appendix~\ref{SS:curvature}.

Let $D\subset\Cd$ be a bounded strongly $\C$-convex domain and $n\in\N$. Given a defining function $\rho$ of $D$, $\varphi=\{w^1,...,w^n\}\subset bD$ and a function $\bde:bD\rightarrow (0,\infty)$, let $P(\varphi;\bde)$ be the union of those connected components of 
\be\label{E:poly1}
		\bigcap_{j=1}^{n}\left\{z\in \Cd:\frac{\left|\left<\bdy\rho(w^j),w^j-z\right>\right| }
		{||\bdy\rho(w^j)||}>\bde(w^j)\right\},
	\ee
that intersect $D$. To place \eqref{E:poly1} into context, we note that if  $K\subset\rl^d$ is a $\cont^1$ (real-)convex domain with defining function $r$, then any $n$-faceted (real-)convex polyhedron $P\subset\rl^d$ can be written as
	\bes
		P=\left\{x\in\rl^d:\left<\nabla r(y^j),y^j-x\right>>\de_j,
			j=1,...,n\right\}
	\ees
for some $y^1,...,y^n\in bK$ and $\de_1,...,\de_n \in \R$. In particular, if $\de_1=\cdots =\de_n=0$, then $P$ is the circumscribed polyhedron $P_{\text c}(\varphi)$ for $\varphi=\{y^1,...,y^n\}$. 

\begin{definition}\label{D:Poly} Given $D$ as above,  an {\em induced polyhedron} is a nonempty set of the form $P(\varphi;\bde)$. We say that $P(\varphi;\bde)$ is an {\em induced polyhedron contained} in $D$, if $P(\varphi;\bde)\subset D$. We call $\varphi$ the {\em source set} and $\bde$ the {\em size function} of $P(\varphi;\bde)$, respectively. 
\end{definition}

Note that, for any $w\in bD$, the set 
	\bes
	F_w := \left\{z\in\Cd:\frac{\left|\left<\bdy\rho(w),w-z\right>\right| }
		{||\bdy\rho(w)||}=\bde(w)\right\}
	\ees
is a real hypersurface in $\Cd$ that is foliated by complex hyperplanes, and divides $\Cd$ into two unbounded $\C$-convex domains. Thus, $F_w$ can be viewed as a complex analogue of a real hyperplane in $\rl^d$. In view of this, we refer to any nonempty set $F_{w_j}\cap bP(\varphi;\bde)$, $j=1,...,n$, as a {\em facet} of $P(\varphi;\bde)$. The elements of the collection
	\bes
	 \mathscr P_n(D):=\big\{P(\varphi;\bde):\varphi\subset bD,\ |\varphi|\leq n,\ \text{and}\ \bde:bD\rightarrow (0,\infty)\big\}
	\ees
are called induced polyhedra with at most $n$ facets. This class is invariant under complex affine transformations, i.e., given a source set $\varphi$ on $bD$,  size function $\bde$, and $A\in \operatorname{GL}(d;\C)$, there exists a size function $\bde'$ such that $A(P(\varphi;\bde))=P(A(\varphi);\bde')$. This follows from \eqref{eq:MobiusonL}.

\subsection{Main results}
\label{s:poissonrandpoly}
We briefly introduce the Poisson and binomial processes, and then state our main results for random polyhedra constructed on these point processes. See Appendix \ref{s:probtool} for a more formal definition of point processes and a list of results regarding them that we use in this paper. 

Let $f: bD \rightarrow (0,\infty)$ be a probability density function i.e.,  $\int_{bD} f(z)  \md \sigma(z) =1$. For $n \geq 1$,  let $\eta_n$ be a Poisson process with intensity measure $nf\sigma$.  From the standard representation of a Poisson process,  we can assume that $\eta_n= \{W^i \}_{1 \leq i \leq N_n}$, where $W^i$ are i.i.d. with probability density $f$, and $N_n$ is an independent Poisson random variable with mean $n$.  A second point process of interest is the binomial process.  We say that a process $\beta_n$ is a binomial process on $bD$ with $n$ points and density $f$ if $\beta_n = \{W^1,\ldots,W^n\}$ where, once again, $W^i$ are i.i.d.  points with probability distribution $f\sigma$.

For a random polyhedral approximation of $D$, we consider an induced polyhedron whose source set is a point process on $bD$ (either $\eta_n$ or $\beta_n$) and size function is $\bde=\bde_n$, where we impose the following assumption on $\bde_n$:
\begin{equation}
  \label{eq:assrd}
      \bde_n(w) = g(w) \delta_n {,} \quad \mbox{where} \quad  \delta_n :=  \left( \frac{\log n}{n}\right)^{\frac{1}{d}} 
\end{equation}
and $g$ is a fixed positive function on $bD$ such that $g\in\cont^\alpha(bD)$, for some $\alpha\in(0,1]$.

Our main objective in this paper is to understand the asymptotics of the volume of the portion of $D$ that is left uncovered by the random polyhedra $P(\eta_n; \bde_n)$ or $P(\beta_n ; \bde_n)$ 
as $n \rightarrow \infty$. Two of our main results are convergence in probability, and under further assumptions, $L^1$-convergence of the random variables $\de_V(D,P(\eta_n,\bde_n))$ and $\de_V(D,P(\beta_n,\bde_n))$, where recall from \eqref{e:volmetric} that
\bes
\de_V(D,P)=\lambda(D \setminus P)\1(P \subset D)  + \lambda(D) \1(P \subsetneq D).
\ees
The rest are variance bounds and normal approximation for these random variables.   
\begin{theorem}[\bf Convergence in Probability]\label{thm:cippoi} 
 Let $f: bD \rightarrow (0,\infty)$ be a bounded probability density function that is bounded away from $0$, and $\bde_n=g\de_n$ be as in \eqref{eq:assrd}. Further, assume that $g$ is such that
  \begin{equation} \label{eq:covdelta}
    \BP(P(\eta_n; \bde_n) \subset  D) \rightarrow 1 \mbox{ as $n \rightarrow \infty$}.
    \end{equation}
  Then,
  \begin{equation}
  \label{e:cipdvpoi}
     { \left(\frac{n}{\log n}\right)^{1/d}}  \delta_V(D,P(\eta_n; \bde_n)) \stackrel{p}{\longrightarrow} \int_{bD}g(z)\md \sigma(z) \mbox{ as $n \rightarrow \infty$}.
    \end{equation}
If \eqref{eq:covdelta} holds for $\beta_n$, then \eqref{e:cipdvpoi} holds for $\beta_n$.
  \end{theorem}
The above theorem gives a simple condition under which the error in volume approximation decays to zero at a fixed rate as the number of facets of the random polyhedra increase.  While \eqref{e:cipdvpoi} gives the precise rate of volume approximation for a fixed $f$ and $g$, it leaves open the question of the the best possible rate of volume approximation among all random polyhedra; see Conjecture \ref{conj}.
\begin{theorem}[\bf  {$L^1$-Convergence and Variance Bounds}] 
\label{t:expvarpoi}
Let $f: bD \rightarrow (0,\infty)$ be a bounded probability density function that is bounded away from $0$, and $\bde_n=g\de_n$ be as in \eqref{eq:assrd}. Further, assume that $g$ is such that
 	\begin{equation}
 	\label{e:covexpvarpoi}
	 \BP(P(\eta_n; \bde_n) \subset  D) = 1-o\left(n^{-\left(1+\frac{2}{d}+\epsilon\right)}\right)\ \text{as}\ n\rightarrow\infty,
 	\end{equation}
for some $\epsilon > 0$.  Then, 
	\begin{equation}
	\label{e:expasymptotpoi}
	\lim_{n \to \infty} \BE \left |  { \left(\frac{n}{\log n}\right)^{1/d}}\delta_V(D,P(\eta_n; \bde_n)) - \int_{bD} g(z)\md \sigma(z) \right | = 0
	\end{equation}
and there exist $c_1, c_2>0$ such that
	\begin{equation}
  \label{e:varasymptotpoi}
  c_1  {\frac{1}{n^{1+2/d}(\log n)^{2+2/d}}}
	\leq \BV[\delta_V(D,P(\eta_n; \bde_n))] 
		\leq c_2 {\frac{(\log n)^{2+2/d}}{n^{1+2/d}}},  \quad n \geq 1.
	\end{equation}
If \eqref{e:covexpvarpoi} holds for $\beta_n$,  then \eqref{e:expasymptotpoi} and the upper bound in \eqref{e:varasymptotpoi} hold for $\beta_n$.  
\end{theorem}

 {Note that the assumption in \eqref{e:covexpvarpoi} is to ensure that the dominating contribution is from the volume of the random polyhedron and not the non-containment event.  Our proofs give explicit bounds in terms of non-containment probabilities, and it may be possible to weaken \eqref{e:covexpvarpoi}. For instance, a quick inspection of the proof of Theorem~\ref{t:expvarpoi} shows that to obtain the expectation asymptotics in \eqref{e:expasymptotpoi}, it suffices to assume that the non-containment probability is $o\Big( (\frac{\log n}{n})^{1/d} \Big)$ as $n\rightarrow\infty$, but we have assumed \eqref{e:covexpvarpoi} to give a combined result for expectation and variance. }

To state a normal approximation result, we fix a metric of approximation for random variables.  Given two random variables, $X$ and $Y$, {\em the Wasserstein distance} $d_W(X,Y)$ is defined as
\begin{equation}
\label{e:defndW}
d_W(X,Y) := \sup_{h \in \mathrm{Lip}(1)} | \BE[h(X)]  - \BE[h(Y)] |,
\end{equation}
where $\mathrm{Lip}(1)$ is the class of all functions $h : \R \to \R$ with Lipschitz constant at most $1$.  The Wasserstein distance metrizes the topology of weak convergence i.e.,  if $d_W(X_n,X) \to 0$ as $n \to \infty$ then $X_n \stackrel{D}{\longrightarrow} X$ as $n \to \infty$ where $\stackrel{D}{\longrightarrow}$ denotes convergence in distribution or weak convergence. Henceforth, $Z$ denotes a standard normal variable.
\begin{theorem}[\bf Central Limit Theorem]
\label{t:cltpoi}
Let $f: bD \rightarrow (0,\infty)$ be a bounded probability density function that is bounded away from $0$, and $\bde_n=g\de_n$ be as in \eqref{eq:assrd}. Further, assume that $g$ is such that \eqref{e:covexpvarpoi} holds for some $\epsilon > 0$. Then,
  \begin{equation}
  \label{e:normapproxdvpoi}
   d_{W}\left(\frac{ \delta_V(D,P(\eta_n; \bde_n)) - \E[ \delta_V(D,P(\eta_n; \bde_n))]}{\sqrt{\BV[ \delta_V(D,P(\eta_n; \bde_n))]}},  Z\right) 
\leq C  {\frac{(\log n)^{6 + 6/d}}{n^{\frac{\min\{\epsilon,1\}}{2}}}},    \quad n \geq 1. 
    \end{equation}
 If \eqref{e:covexpvarpoi} and the lower bound in \eqref{e:varasymptotpoi} hold for $\beta_n$, then \eqref{e:normapproxdvpoi} holds for $\beta_n$.  
  \end{theorem}
We have made an assumption on the decay of $\delta_n$ in \eqref{eq:assrd}, and on the asymptotic behavior of containment probabilities, as in \eqref{eq:covdelta} and \eqref{e:covexpvarpoi}. The following lemma assures us that these assumptions are compatible as well as essential.  
\begin{Prop}\label{lem:rdelta} Let $f: bD \rightarrow (0,\infty)$ be a bounded probability density function that is bounded away from $0$, and $\bde_n=g\de_n$ with $\delta_n$ as in \eqref{eq:assrd} and $g$ is a continuous function.  Then, the following hold.
  \begin{itemize}
  \item[(a)] There exists a $c_D  >0$ such that if $\min_{z\in bD} g(z)< c_D$,  then
    $$ \BP(P(\eta_n; \bde_n) \subset  D) \rightarrow 0 \mbox{ as $n \rightarrow \infty$}.$$
    \item[(b)] For any $\alpha >0$, there exists a $C_{D} >0$ such that if $\min_{z \in bD}g(z) > C_D$, then
      \begin{equation}\label{eq:crate}  \BP(P(\eta_n; \bde_n) \subset  D) = 1 -   {\mathcal O}(n^{-\alpha}) \mbox{ as $n \rightarrow \infty$}.
        \end{equation}
      \item[(c)] Both (a) and (b) also hold when $\eta_n$ is replaced by $\beta_n$.
    \end{itemize}
\end{Prop}
 It is clear from our proof that Part $(b)$ holds more generally for any measurable $g$. A more careful analysis of our proofs can yield estimates on $c_D,C_D$ but given that these are likely to be far from the sharp bounds,  we avoid quantifying them here.
\subsection{Remarks}
\label{s:remltthms}
We now remark in detail about the assumptions in our results,  the nature of results and how they compare to real inscribed polyhedra.  \\

\noindent {\em Containment condition.}  Theorem \ref{thm:cippoi} assumes that containment happens with high probability i.e.,  \eqref{eq:covdelta} holds.  Suppose not i.e.,  $\liminf_{n \to \infty} \BP(P(\eta_n; \bde_n) \subset  D) < 1$.  Then we can show that along some subsequence,  $\delta_V(D,P(\eta_n;\bde_n))$ converges in probability to a non-degenerate random variable supported on $\{0,\lambda(D)\}$. Thus, the error doesn't converge to $0$ if \eqref{eq:covdelta} is violated.\\

\noindent {\em Optimizing the size function.} As mentioned earlier, non-trivial size functions occur in similar models in the real case.  However,  in the real case,  one is able to obtain the decay rate of the optimal $\bde_n$ with explicit constants because the transition from containment probability to coverage probability on the boundary is done via a suitable Riemannian metric, following which the probabilities are estimated using the classical coverage results of \cite{flatto1977random, hall1985coverage,janson1986random}. In our case, the rate of decay of $\bde_n$, i.e., Lemma \ref{lem:rdelta}, is also obtained by reducing the question of containment to that of coverage of the boundary by the so-called caps of the approximating polyhedra. However, due to the inhomogeneous nature of the caps, this is at best a coverage problem in a sub-Riemannian metric, for which no coverage results of the type in \cite{flatto1977random,hall1985coverage,janson1986random} are available. This is one of the obstacles in proving Conjecture~\ref{conj}. \\

\noindent {\em Comparison with real polyhedra.} In Section~\ref{SS:Summary}, we have already compared Theorem \ref{thm:cippoi} with the results known for randomly inscribed and circumscribed polyhedra in the real setting.  We now compare the other limit theorems in this paper with those for randomly inscribed polyhedra (see \cite{reitzner2002random,reitzner2003random,richardson2007random,richardson2008inscribing, thale2018central, turchi2018limit}).  Recall that in these models, $K\subset\rl^d$ is a smooth convex body,  
the polyhedron, say $P_n$, is constructed as the convex hull of a point process (Poisson or binomial) on $bK$, and one studies the asymptotics of $V_n := \lambda(K \setminus P_n)$. It is known that
\begin{itemize}
\item [$(a)$] $\BE[V_n]\sim v^{*}_{\R} n^{-2/(d-1)}$ as $n\rightarrow\infty$ (\cite[Theorem 1]{reitzner2002random}),
\item [$(b)$] $\BV[V_n]=\Theta(n^{-1 - 4/(d-1)})$ as $n\rightarrow\infty$ (\cite[Theorem 8]{reitzner2003random} and \cite[Theorem 1.1]{richardson2008inscribing}),
\item [$(c)$] $d_K\left(\frac{V_n - \BE[V_n]}{\sqrt{BV[V_n]}},Z\right)  \leq Cn^{-1/2}(\log n)^{3(1 + 2/(d-1))}$ (\cite[Theorem 1 and Section 4.1]{thale2018central}),
\end{itemize}%
where $d_K$ is the Kolmogorov distance and $v^{*}_{\R}$ is an explicit constant described in terms of the Blaschke surface area measure. Up to logarithmic factors,  our expectation asymptotics and variance bounds resemble the above results with the exponent $(d-1)$ replaced by $2d$. The error in normal approximation has an extra $\epsilon$ dependence owing to assumption \eqref{e:covexpvarpoi}.   An expectation result similar to $(a)$ is proven for randomly circumscribed polytopes in \cite{boroczky2004approximation}, but we are not aware of variance bounds or a central limit theorem for this model.   Similar results are known if the point process (Poisson or binomial) is in $K$ (see \cite{reitzner2010random}), where the exponent is $(d+1)$ instead of $(d-1)$.  Exact variance asymptotics have also been obtained for real random polytopes in \cite[Corollary 1.1 and 1.2]{calka2014variance} (see also  \cite[Chapter 5]{stemeseder2014random}) via scaling transforms.  \\

\noindent {\em On the metric of normal approximation.}  Our proof for the normal approximation result (Theorem \ref{t:cltpoi}) uses the second-order Poincar\'e inequalities for Wasserstein distance in \cite{last2016normal,lachieze2017new} (see Appendix \ref{s:probtool}).  One can instead use the second-order Poincar\'e inequalities for the stronger Kolmogorov distance from \cite{last2016normal,lachieze2017new} but this involves bounding some additional terms. Though this should be possible using our current estimates,  we avoid doing so here.  We expect that the rate of convergence will be of same order even under Kolmogorov distance. \\

\section{Geometric estimates}\label{S:GeomEst}

In this section, we lay the groundwork for the deterministic aspects of our main proofs. Subsection~\ref{SS:cutscapsvis} contains some new definitions that will be used throughout the paper. The subsequent subsections contain some key local estimates. For a quick summary of these estimates, the reader may jump to the statements of Theorems~\ref{T:depth} and \ref{T:CutsCapsGeneral} and Propositions~\ref{P:NN1} and \ref{P:NN2}.  We also note that the assumption of $\alpha$-H{\"o}lder continuity on $g$ is only essential to the proofs of \eqref{E:maxdepth} in Theorem~\ref{T:depth}, Corollary~\ref{C:visibility}, and Proposition~\ref{P:NN2}.

\begin{ntn} The following notation will be used throughout the paper. 

\begin{enumerate}[label=(\roman*)]
	\item For a point $z\in\Cd$
		\begin{itemize}
		\item $(z_1,...,z_d)$ denotes $z$ in complex coordinates,
		\item $(x_1,y_1,...,x_d,y_d)$ denotes $z$ in real coordinates, i.e., $z_j=x_j+iy_j$, $j=1,...,d$, 
		\item $z'=(z_1,...,z_{d-1})\in\C^{d-1}$ denotes the projection of $z$ onto the space $\{z:z_d=0\}$, 
		\item $||z||$ denotes the Euclidean norm of $z$.
		\end{itemize}

\item For $z,w\in\Cd$, $\left<z,w\right>=z_1w_1+\cdots+z_dw_d$. Thus, $||z||=\left<z,\zbar\right>^{1/2}$.

\item For $z\in\rl^d$ and $r>0$, $B_w(r)$ denotes the Euclidean ball of radius $r$ centered at $z$. 

\item $\lambda$ and $\lambda_{2d-1}$ denote the Lebesgue volume measures on $\Cd\cong\rl^{2d}$ and $\C^{d-1}\times\rl\cong\rl^{2d-1}$, respectively. 

\item For a $\cont^1$ domain $D\subset\Cd$, 
	\begin{itemize}
		\item $\hat\eta(w)=(\hat\eta_1(w),...,\hat\eta_d(w))$ denotes the inner unit normal to $bD$ at $w\in bD$,
		\item $\sigma$ denotes the Euclidean surface area measure on $bD$.
	\end{itemize}
\item $\GL(m;\C)$ denotes the space of $m\times m$ invertible complex matrices.

\item For a $\cont^1$ function $f:U\rightarrow\rl$ on an open set $U$,
\begin{itemize}
\item  $\bdy f=\left(\smpartl{f}{z_1},\cdots,\smpartl{f}{z_d}\right)
=\left(\frac{1}{2}\smpartl{f}{x_1}-\frac{i}{2}\smpartl{f}{y_1},
\cdots,\frac{1}{2}\smpartl{f}{x_d}-\frac{i}{2}\smpartl{f}{y_d}\right)$, when $U\subset\Cd$,
\item $\nabla f=\left(\smpartl{f}{x_1},\smpartl{f}{y_1},...,\smpartl{f}{x_{d}},\smpartl{f}{y_d}\right)$, when $U\subset\Cd$,
\item $\nabla f=\left(\smpartl{f}{x_1},\smpartl{f}{y_1},...,\smpartl{f}{x_{d-1}},\smpartl{f}{y_{d-1}},\smpartl{f}{x_{d}}\right)$, when $U\subset\C^{d-1}\times\rl$.
\end{itemize}  

\item Given a compact set $X\subset\Cd$ and $\alpha\in(0,1)$, the space of $\alpha$-H{\"o}lder functions on $X$ is
	\bes
		\cont^{\alpha}(X)=\left\{g:X\rightarrow\rl:||g||_{\alpha}:=||g||_\infty+\sup\limits_{\substack {x,y\in X\\ x\neq y}}\frac{|g(x)-g(y)|}{||x-y||^\alpha}<\infty\right\},
	\ees
where $||g||_\infty=\sup_{x\in X}|g(x)|$. 

\item Given a family of functions $\{f_w\}_{w\in bD}$ and a non-negative function $h$ on a normed space $\mathbb X$, $f_w=O_D(h)$ as $||x||_{\mathbb X}\rightarrow 0$ denotes the existence of $M_D,\de_D>0$ such that 
	\bes
		|f_w(x)|\leq M_Dh(x)\quad  \text{for all}\ w\in bD, ||x||_{\mathbb X}<\de_D.
\ees
The notation $f_w^1,f_w^2=O_D(h)$ as $||x||\rightarrow 0 $ simply denotes that $f_w^1=O_D(h)$ as $||x||\rightarrow 0$ and $f_w^2=O_D(h)$ as $||x||\rightarrow 0$.
\item Given real-valued functions $f,h$ on some set $S$, we say that $f\lesssim h$ if $f(x)\leq Ch(x)$ for some $C$ independent of $x\in S$. We abbreviate the statement $f\lesssim h\lesssim f$ to $f\approx h$. 
\end{enumerate}
\end{ntn}

\subsection{Preliminaries}\label{SS:cutscapsvis}

In this subsection, we introduce some sets and measurements associated to induced polyhedra. These will make crucial appearances in our main proofs. See Figure \ref{fig:cutcap} for an illustrative picture of some of these sets and measurements. Towards the end, we also collect some useful results from the literature. 

\begin{definition}\label{D:geom} Let $D\subset\Cd$, $d\geq 2$, be a bounded strongly $\C$-convex $\cont^2$ domain, and $\rho:\Cd\rightarrow\rl$ be a $\cont^2$ defining function of $D$. Let $\bde:bD\rightarrow(0,\infty)$ be a continuous function and $\de>0$. 

	\begin{enumerate}

	\item [$(a)$] For any set $S\subseteq bD$ and bounded function $g:S\rightarrow[0,\infty)$, the {\em $g$-tube of $S$} is the set
		\bes
			T_g(S):=\{w+t\hat\eta(w):w\in S,\ 0\leq t\leq g(w)\},
		\ees
		where the notation is abbreviated to $T_c(S)$ when $g|_S\equiv c$.
	\item [$(b)$] The function $L:\Cd\times bD\rightarrow\C$ is defined as 
		\bes
			L(z,w)=L_\rho(z,w):=\left<\frac{\bdy\rho(w)}{||\bdy\rho(w)||},w-z\right>
			=\frac{\left<\bdy\rho(w),w-z\right>}{||\bdy\rho(w)||}.
		\ees
Note that $\{z\in\Cd:L(z,w)=0\}$ is $\mathcal H_w bD$, the complex tangent plane to $bD$ at $w$. 
	\item [$(c)$] The {\em $\de$-cut at $w\in bD$} is the (open) set 
					\bes
						C(w;\de):=\left\{z\in D:|L(z,w)|<\de\right\}.
					\ees

	\item [$(d)$] The {\em $\de$-cap at $w\in bD$} is the (closed) set
					\bes
						S(w;\de):=\overline{C(w;\de)}\cap bD
						=\left\{\zt\in bD:|L(\zt,w)|\leq \de\right\}.
					\ees
	\item [$(e)$] For $w,\zt\in bD$, the {\em inner depth of $C(w;\de)$ at $\zt$} is 
				\bes
					r_w(\zt;\de):=\sup\{s\geq 0:\zt+t\hat\eta(\zt)\in \overline{C(w;\de)},
					\ \text{for all}\ 					 0\leq t\leq s\},
				\ees
and the {\em maximum ${\bde}$-visibility at $\zt$} is
				\bes
					r_\infty(\zt;\bde):=\sup\{r_w(\zt;\bde(w)):w\in bD\}.
				\ees

\item [$(f)$] The {\em outer depth of $C(w;\de)$} is
				\bes
					R_w(\de):=\inf\{r\geq 0:\overline{C(w;\de)}\subset T_r(bD)\}.
				\ees
and $R_\infty(\bde)=\sup\{R_w(\bde(w)):w\in bD\}$.

			\item [$(g)$] The {\em tubular $\de$-cut at $w\in bD$} is the set
		\bes
		TC(w;\de):=T_{r_w(\cdot;\de)}(S(w;\de))=	
	\left\{\zt+t\hat\eta(\zt):\zt\in S(w;\de),\ 0\leq t \leq r_w(z;\de)\right\}.
		\ees
			\item [$(h)$] For $s\geq 0$, the {\em $(\bde,s)$-visibility region at $\zt\in bD$} is the set
					\bes
						G_s(\zt;\bde):=\{w\in bD:r_w(\zt;\bde(w))\geq s\}
					=\left\{w\in bD:|L(\zt+t\hat\eta(\zt),w)|\leq \bde(w)\
						 \text{for}\ t\in[0,s]\right\}.
					\ees
	\end{enumerate}
\end{definition}
\bigskip

\begin{figure}[H]
\begin{center} 
\begin{overpic}[grid=false,tics=10,scale=0.5]{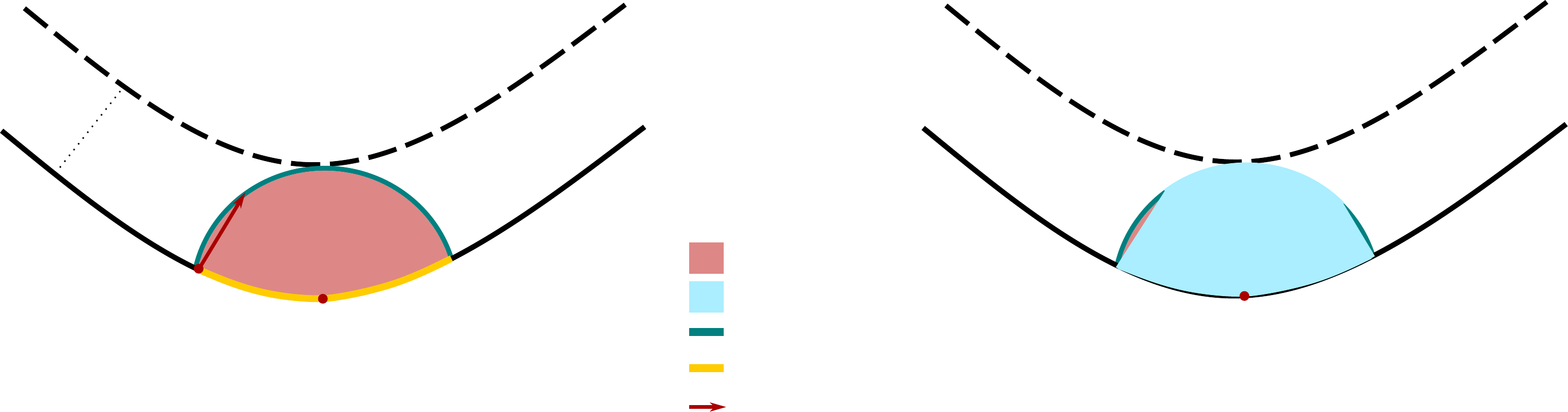} 
\put(187,51){\scriptsize Key}
\put(200,39){\scriptsize $C(w;\delta)$}
\put(200,29){\scriptsize $TC(w;\delta)$}
\put(200,19){\scriptsize $bC(w;\delta)\cap D$}
\put(200,10){\scriptsize $S(w;\delta)$}
\put(200,0){\scriptsize $z+r_w(z;\delta)\hat\eta(z)$}
\put(9,68){\rotatebox{56}{\scriptsize $R_w(\de)$}} 
\put(50,27){$z$}
\put(83,20){$w$}
\put(335,20){$w$}
\end{overpic}
\bigskip
\caption{A schematic of a $\de$-cut, $\de$-cap and tubular $\de$-cut with source $w$. In the figure on the right, the curve in teal and the region in pink are the portions of $bC(w;\delta)$ and $C(w;\delta)$, respectively, that are not contained in $TC(w;\delta)$.}
\label{fig:cutcap}
\end{center}
\end{figure}
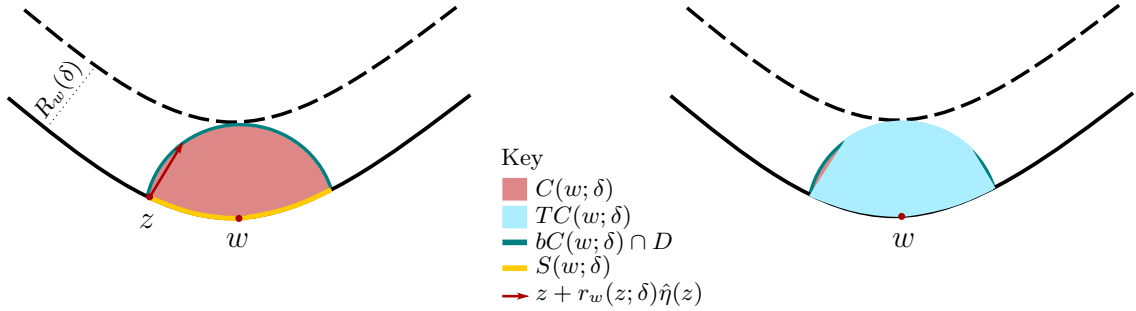

The above definitions are independent of the choice of defining function for $D$ (see Lemma~\ref{L:deffn} below). For a fixed 
$\bde:bD\rightarrow (0,\infty)$, and $w\in bD$, $C(w;\bde)$, $S(w;\bde)$, $TC(w;\bde)$, $r_w(\cdot;\bde)$ and $R_w(\bde)$ denote $C(w;\bde(w))$, $S(w;\bde(w))$, $TC(w;\bde(w))$, $r_w(\cdot;\bde(w))$ and $R_w(\bde(w))$, respectively. Lastly, note that for all $w,\zt\in bD$,
\begin{itemize}
\item [$*$] $r_w(\zt;\bde(w))=-\infty$ if $\zt\notin S(w;\bde(w))$;
\item [$*$] $TC(w;\bde(w))\subseteq \overline{C(w;\bde(w))}\subseteq T_{R_\infty(\bde)}(bD)$; 
\item [$*$] $G_0(\zt;\bde)=\{w\in bD:\zt\in S(w;\bde(w))\}$, and $G_s(\zt;\bde)=\emptyset$ if $s>r_\infty(\zt;\bde)$.
\end{itemize}

The following two lemmas allow us to (locally) simplify the defining function of $D$. These results are well-known, but fragmented throughout the literature. For the reader's convenience, we state them here and relegate their proofs to Appendix~\ref{SS:curvature}.

\begin{lemma}\label{L:deffn} Let $D$ and $\rho$ be as in Definition~\ref{D:geom}. 
\begin{itemize}
\item [$(a)$] For any defining function $\wt\rho$ of $D$,  $L_{\wt\rho}(z,w)=L_{\rho}(z,w)$ for all $(z,w)\in\Cd\times bD$.
\item [$(b)$] For $M\in\GL(d;\C)$, if $\rho^*$ denotes the defining function $\rho\circ M^{-1}$ of $M(D)$, then
	\bes
		L_{\rho}(z,w)=\frac{||\bdy\rho^*(M(w))||}{||\bdy\rho(w)||}
		L_{\rho^*}(M(z),M(w)),\quad z\in \Cd, w\in bD. 
	\ees
\end{itemize}  
In particular, if $M$ is unitary, then $||\bdy\rho^*(M(w))||=||\bdy\rho(w)||$ for all $w\in bD$. Thus, the cuts, caps, depths, and visibility regions are preserved by $M$.
\end{lemma}

\begin{lemma}\label{L:Model} Let $D$ be as in Definition~\ref{D:geom}. Then there exists an $\eps_D>0$ so that for each $w\in bD$, there is a unitary transformation $U_w$ on $\Cd$ such that $U_w(w)=0$, and  $\alpha_j, \beta_j, e\in\rl$ and $\gamma_j\in\C$ such that 
	\begin{itemize}
\item [$(a)$]  $\alpha_j, \beta_j, e,$ and $\gamma_j$ depend continuously on $w$;
\item [$(b)$]  $0\leq \beta_j<\alpha_j$ for $j=1,...,d-1$; and 
\item [$(c)$] $D^w:=B_0(\eps_D)\cap U_w(D)=\left\{(z',x_d+iy_d)\in B_0(\eps_D):y_d>f_w(z',x_d)\right\}$ for some $\rl$-valued $f_w$ such that
	\bes
		f_w(z',x_d)=\sum_{j=1}^{d-1}\alpha_j|z_j|^2+\sum_{j=1}^{d-1} 		
		\beta_j\rea(z_j^2)+
			\ima\sum_{j=1}^{d-1}\gam_j z_jx_d+ex_d^2
			+O_D(||(z',x_d)||^3)
	\ees
as $||(z',x_d)||\rightarrow 0$.
\end{itemize} 
In particular, the complex-restricted curvature of $D$ at $w$ is  $\nu_{D}(w)=\frac{1}{16}\prod_{j=1}^{d-1}(\alpha_j^2-\beta_j^2)$.
\end{lemma}
\begin{remark}\label{R:simplification} In the subsequent subsections, the above lemma will be often used to reduce any local computation at a $w\in bD$ to a corresponding computation for $D^w$ at $0$. In this case, for all $z\in bD^w=\left\{(z',x_d+iy_d)\in B_0(\eps_D):y_d=f_w(z',x_d)\right\}$, 
\bes
	L(0,z)=\left<(0,...,0,i),z\right>=iz_d\qquad \text{and}\qquad \hat\eta(z)=\frac{(-\nabla f_w(z',x_d),1)}{\sqrt{1+||\nabla f_w(z',x_d)||^2}}.
\ees
The above quantities are independent of the choice of defining function for $D^w$ (see Lemma \ref{L:deffn}) and, hence, can be computed by taking a particular choice, e.g., $\rho_w(z):=f_w(z',x_d)-y_d$. Also, note that for $z = (z',z_d) \in bD^w$,
\be
\label{e:bdfeqn}
y_d-x_d\partl{f_w}{x_d}(z',x_d) =f_{\balp,\bbet}(z')-ex_d^2	+ O_D(||(z',x_d)||^3)\qquad \text{as}\ ||(z',x_d)||\rightarrow 0,
\ee
where $f_{\balp,\bbet}(z')=\sum_{j=1}^{d-1}\alpha_j|z_j|^2+\sum_{j=1}^{d-1} 		
		\beta_j\rea(z_j^2)$. Due to the strong $\C$-convexity of $D$ --- see \eqref{E:strcvx1}, there is a $c_D>0$ (independent of $w\in bD$) such that
	\be\label{e:hessianbound}
	Q_0^{D^w}(z',z')=f_{\balp,\bbet}(z')\geq c_D||z'||^2,\quad  \text{for all}\ 
		z'\in\C^{d-1}\cong\mathcal H_0bD^w.
	\ee
\end{remark}

In our setting, $|L(z,w)|$ is the distance of $z\in\overline D$ from the affine space $w+\mathcal H_wbD$, $w\in bD$. The following estimates on $|L(z,w)|$ from \cite{LaSt} are useful.
\begin{lemma}\label{L:Lest} Let $D$ be as in Definition~\ref{D:geom}. Then,
	\be\label{E:sandwich}
	||w-z||^2 \lesssim |L(z,w)|\lesssim ||w-z||,
		\quad z\in\overline D, w\in bD.
	\ee
Moreover, there exists a $t_D>0$ such that,
	\be\label{E:normal}
		|L(\zt,w)|+t \approx |L(\zt+t\hat\eta(\zt),w)|,\quad \zt,w\in bD, t<t_D.
	\ee
\end{lemma}
\begin{proof}  The first estimate \eqref{E:sandwich} is essentially proved at the end of \cite[Section 3.3]{LaSt}. Since it is only stated for $z,w\in bD$ therein, we recall the argument here. The upper bound follows easily since $|L(z,w)|=\left\|\left<\frac{\bdy\rho(w)}{||\bdy\rho(w)||},w-z\right>\right\|\leq ||w-z||$. The lower bound is the strong $\C$-convexity condition; see \eqref{E:strcvx2}.

The second estimate \eqref{E:normal} is exactly \cite[Lemma 4.3]{LaSt}. 
\end{proof}

Lastly, we recall Roccaforte's variation of Weyl's tube formula for nonuniform tubes. The theorem is originally formulated for smooth oriented hypersurfaces in $\rl^d$, but we state it for boundaries of $\cont^2$ domains in $\rl^{2d}$, as these are the relevant hypersurfaces in our case. We note that Roccaforte's proof does not require higher than $\cont^2$-regularity.

\begin{theorem}[{\cite[Theorem 1]{roccaforte2013volume}}] \label{thm:rf} Let $D\subset \rl^{2d}$ be a bounded $\cont^2$ domain. Let $g:bD\rightarrow[0,\infty)$ be a bounded, measurable function. Then, for $\sup_{bD}|g|$ sufficiently small,
	\bes
	\lambda(T_g(bD))= 
		\sum_{j=0}^{d-1}\frac{(-1)^j}{j+1}\int_{bD}g(x)^{j+1}s_j(x)\,d\sigma(x),
	\ees
where $s_j(x)$ is the $j^{\text{th}}$ symmetric polynomial in the principal curvatures (eigenvalues of the shape operator) of $bD$ at $x$. Here, $bD$ is oriented according to the inner unit normal field. 
\end{theorem}

\subsection {Depth estimates} 

In view of Theorem~\ref{thm:rf} above, it is useful to estimate the inner and outer depths of the cuts of an induced polyhedron. This allows us to squeeze the complement of the polyhedron in $D$ between two (inner) tubular neighborhoods of the boundary of $D$, one uniform and one nonuniform. 

\begin{theorem}\label{T:depth}
Let $D$ be as in Definition~\ref{D:geom}, $\de>0$, and $\boldsymbol\de=g\de$, for some positive function $g\in\cont^{\alpha}(bD)$, where $\alpha\in(0,1)$. There exist $C_D,\de_D>0$ such that, if $\zt\in bD$ and $||\boldsymbol\de||_\infty<\de_D$, then
	\bea
		&R_\infty(\boldsymbol\de) \leq C_D||\boldsymbol\de||_{\infty},& \label{E:outer}
\\ 
		&0\leq  r_\infty(\zt;\boldsymbol\de)-g(\zt)\de\leq C_D||g||_{\alpha}^2\de^{1+\frac{\alpha}{2}}. &	\label{E:maxdepth}
	\eea
\end{theorem}
To prove the above theorem, we need estimates on $r_w(\zt;\de)$. We already have that $r_w(\zt;\de)\lesssim \de-|L(\zt,w)|$ for $\zt\in S(w;\de)$ from \eqref{E:normal}. This does not suffice for our purposes, and we improve this estimate as follows.

\begin{lemma}\label{L:depth} Let $D$ be as in Definition~\ref{D:geom}.
There exist $C_D,\de_D>0$ such that  for any $w\in bD$, $\de<\de_D$, and $z\in S(w;\de)$, 
\be\label{E:depthest}
		\de-|L(z,w)|\leq r_w(z;\de)
			\leq \sqrt{\de^	2-|L(z,w)|^2}+C_D\de^2.
\ee
\end{lemma}
\begin{proof}
By Lemmas~\ref{L:deffn} and \ref{L:Model}, we may assume that $w=0$ and $D=D^w$. We denote $f_w$ by $f$, $\tfrac{\bdy f}{\bdy x_d}(z',x_d)$ by $\bdy_{x_d}f(z)$, and the  $d$-th component of $\hat\eta(z)$ by $\hat\eta_d(z)$. We will invoke the properties of $D^w$ and $f_w$ collected in Remark~\ref{R:simplification}.

Owing to the upper bound in \eqref{E:sandwich}, we may choose $\de_D$ small enough so that if $\de<\de_D$ and $z\in S(0;\de)$, then $z+ r\hat\eta(z)\notin S(0;\delta)$ for any $r>0$. Thus, $z+ r\hat\eta(z)\in bC(0;\de)$ if and only if $|L(z+ r\hat\eta(z),0)|=\de$. Therefore, we must have that
\beas
	r_0(z;\de)&=&	\min\left\{r\geq 0:|z_d+r\hat\eta_d(z)|=\de\right\}. 
\eeas
Thus, $\de=|z_d+r_0(z;\de)\hat\eta_d(z)|\leq |z_d|+r_0(z;\de)||\hat\eta(z)||=|L(z,0)|+r_0(z;\de)$. This gives the lower bound in \eqref{E:depthest}. 

For the upper bound, we solve for $r_0(z;\de)$ in the quadratic equation $|z_d+r_0(z;\de)\hat\eta_d(z)|^2=\de^2$ to obtain that
\bea
		r_0(z;\de)&=&\frac{
\sqrt{\rea(\overline{z_d}\hat\eta_d(z))^2+|\hat\eta_d(z)|^2(\de^2-|z_d|^2)}
	-\rea(\overline{z_d}\hat\eta_d(z))}{|\hat\eta_d(z)|^2} 
\label{E:indepth}\\
		&\leq &|\hat\eta_d(z)|^{-1}\sqrt{(\de^2-|z_d|^2)}+
		|\hat\eta_d(z)|^{-2}\left(|\rea(\overline{z_d}\hat\eta_d(z))|
	-\rea(\overline{z_d}\hat\eta_d(z))\right)
\notag\\
&\leq &|\hat\eta_d(z)|^{-1}\sqrt{(\de^2-|z_d|^2)}+2|\hat\eta_d(z)|^{-2}\max\{0,(x_d\bdy_{x_d}f(z',x_d)-y_d)\}.
\label{e:depth}
	\eea
Now, let $S'(\de):=\{z\in S(0;\de): x_d\bdy_{x_d}f(z',x_d)-y_d> 0\}$. From \eqref{e:bdfeqn} and \eqref{e:hessianbound}, we have that 
	\bea
		x_d\bdy_{x_d}f(z',x_d)-y_d&=&-f_{\balp,\bbet}(z')+ex_d^2
			+h(z',x_d)\notag\\
		&\leq& -c_D||z'||^2+ex_d^2+h(z',x_d),\label{E:xdf-yd}
	\eea
where $h(z',x_d)=O_D(||(z',x_d)||^3)$ as $||(z',x_d)||\rightarrow 0$.  Thus, \eqref{E:xdf-yd} yields that if $e\leq 0$ and $\de>0$ is sufficiently small, then $x_d\bdy_{x_d}f(z',x_d)-y_d\leq 0$ on $S(0;\de)$, i.e., $S'(\de)$ is empty. On the other hand, if $e>0$, then by \eqref{E:xdf-yd}, $||z'||\leq C_D(|x_d|)$ on $S'(\de)$ for some $C_D>0$ and sufficiently small $\de>0$. But $\sup\{||z'||^2,|x_d|:z\in S(0;\de)\}=O_D(\de)$ as $\de\rightarrow 0$.  Thus, 
	\be\label{E:S'}
		\sup\{||z'||,|x_d|:z\in S'(\de)\}
			=O_D(\de)\qquad \text{as}\ \delta\rightarrow 0,
	\ee
and $\sup\{\max\{0,x_d\bdy_{x_d}f(z',x_d)-y_d\}:z\in S(0;\de)\}= O_D({\de^{2}})$ as $\de\rightarrow 0$. Furthermore,
	\bes
		|\hat\eta_d(z)|^{-1}=\sqrt{\frac{1+||\nabla f(z',x_d)||^2}{1+\bdy_{x_d}f(z',x_d)^2}}=1+O(||z',x_d||^2)\ \text{as}\ ||(z',x_d)||\rightarrow 0.
	\ees
Thus, $\sup\{|\hat\eta_d(z)|^{-1}:z\in S(0;\de)\}=1+O_D(\de)$ as $\de\rightarrow 0$. Substituting these estimates into \eqref{e:depth} gives the desired upper bound, and completes the proof. 
\end{proof}

\noindent{\em Proof of Theorem~\ref{T:depth}.} We use Lemmas~\ref{L:Lest} and \ref{L:depth} to prove \eqref{E:outer} and \eqref{E:maxdepth}, respectively. First, let $t_D$ be as in \eqref{E:normal}.  By \eqref{E:sandwich}, we may choose $\de_D>0$ sufficiently small so that for any fixed $w\in bD$ and $\de<g(w)^{-1}\de_D$, every $z\in C(w;g(w)\de)$ admits a unique projection onto $bD$, i.e., there is a unique $\zt=\zt_z\in bD$ and $t=t_z>0$ such that $z=\zt+t\hat\eta(\zt)$. Moreover, owing to the continuous dependence of $t_z$ on $||z-w||$, we may also assume that $t<t_D$ for all $z\in C(w;g(w)\de)$. So, we have that $R_w(\bde(w)) \leq t_D$. Then, by \eqref{E:normal} 
	\bes
		t\leq C_D|L(\zt+t\hat\eta(\zt),w)|=C_D|L(z,w)|\leq C_Dg(w)\de.
	\ees
Thus, $t_D \leq C_Dg(w)\delta$, which implies that $R_w(\bde(w))\leq C_Dg(w)\de$. Taking supremum in $w\in bD$, we obtain \eqref{E:outer}. 

Next, for \eqref{E:maxdepth}, observe that since 
	\beas
		r_w(\zt;\bde(w))=	
			\begin{cases}
				-\infty,\ & \text{when}\ \zt\notin S(w;\bde(w)),\\ 
				g(w)\de,\ & \text{when}\ \zt=w,
			\end{cases}
	\eeas
we have that $g(w)\de\leq r_\infty(\zt,\bde)=\sup\{r_w(\zt;\bde(w)):w \in bD,  \ \text{such that} \  \zt\in S(w;\bde(w))\}.$  
By Lemma~\ref{L:depth}, there exist $C_D,\de_D>0$ such that, for $||\bde||_\infty<\de_D$,
	\beas
		r_w(\zt;\bde(w))&\leq& g(w)\de+C_D||g||_\infty\de^2\\
&=&g(\zt)\de+(g(w)-g(\zt))\de+C_D||g||_\infty\de^2\\
	&\leq& g(\zt)\de+||g||_\alpha||\zt-w||^\alpha \de+C_D||g||_\infty\de^2.
\eeas
Thus, shrinking $\de_D$ if necessary, we have from Lemma~\ref{L:Lest} that 
\beas
r_w(\zt;\bde(w))-g(\zt)\de\leq ||g||_\alpha||\zt-w||^\alpha \de+C_D||g||_\infty\de^2
\leq ||g||_\alpha C_D|L(z,w)|^{\frac{\alpha}{2}}\de
\leq C_D||g||_\alpha^2 \de^{1+\frac{\alpha}{2}},
	\eeas
for $||\bde||_\infty<\de_D$. Taking supremum over $w$, we obtain \eqref{E:maxdepth}.
\qed

\begin{cor}\label{C:outertube}
Let $D$ be as in Definition~\ref{D:geom} and $\de>0$. There exist $C_D,\de_D>0$ such that for any $w\in bD$ and $\de<\de_D$,
	\bes
		C(w;\de)\subset T_{C_D\de}(S(w;C_D\de)).	
\label{E:outertube}
	\ees
\end{cor}
\begin{proof} Let $C_D,\de_D$ be as in Theorem~\ref{T:depth} (for $\bde\equiv \de$) and $t_D$ be as in Lemma~\ref{L:Lest}. Shrinking  $\de_D>0$ further, if necessary, we may assume that $C_D\de_D<t_D$, and each $z\in C(w;\de)$ has a unique projection $\zt$ onto $bD$ whenever $\de<\de_D$. Thus, $z=\zt+t\hat\eta(\zt)$ for some $t<R_\infty(\bde)\leq C_D\de<t_D$. Now, owing to \eqref{E:normal}, $|L(z,w)|\lesssim \de$, $z,w\in bD$,  and the proof is complete.
\end{proof}

\subsection{Area and volume estimates} The depth estimates obtained in the previous subsection give bounds on the volumes of the unions of caps and tubular caps of an induced polyhedron. We now obtain more precise area and volume estimates for individual caps, cuts, and visibility regions of a strongly $\C$-convex domain.  Such estimates for model quadratic domains have been obtained in Section \ref{SS:Models}. We now extend these estimates to general strongly $\C$-convex domains via second-order approximation (see Remark~\ref{R:simplification}). 

\begin{theorem} \label{T:CutsCapsGeneral}  Let $D$ be as in Definition~\ref{D:geom}. Let $\nu_D$ be the complex-restricted curvature of $bD$, $\kappa_d$ be the Lebesgue volume of the unit ball in $\rl^d$, and $h_d=2\int_0^{\pi/2}\cos^d\!\theta\:\md\theta$. Then there exist $C_D,\de_D>0$ such that  for $w\in bD$, $\de<\de_D$, and $\bde\equiv\de$, 
	\bea
		\left|\sigma\big(S(w;\de)\big)-h_dk_d(w)\,\de^{d}\right|
		\leq  C_D\de^{d+\frac{1}{2}},
\label{eq:cap}
	\eea
	\bea
		\left|\lambda\big(C(w;\de)\big)-\frac{h_{d+1}k_d(w)}{d}\,\de^{d+1}\right|
	\leq C_D\de^{d+\frac{5}{4}},
\label{eq:cut}
	\eea
	\bea
\label {eq:approxcut}
		0\leq \lambda\big(C(w;\de)\big)-\lambda\big(TC(w;\de)\big)\leq C_D\de^{2d+1},
	\eea
	\bea
	  \left|\sigma\big(G_{t\de}(w;\bde)\big)- 
		 h_d(t)k_d(w)\,\de^d \right|\leq C_D\de^{d+\frac{1}{2}},\ t\in \left[0,1-C_D\sqrt{\de}\right],
\label{eq:vis} 
	\eea
where $k_d(w)=4\kappa_{2d-2}\nu_{D}(w)^{-1/2}$ and $h_d(t)=2\int_0^{\sin^{-1}\sqrt{1-t^2}}(\cos\theta-t)^{d-1}\cos\theta\, d\theta$, $t\in[0,1]$.
\end{theorem}
\begin{remark}\label{R:G=S}
While the surface areas of $S(w;\de)$ and $G_0(w;\bde)$, $\bde\equiv\de$, coincide up to first order in $\de$ (since $h_d(0)=h_d$), they are not the same sets unless the quasimetric $\mathtt d$ defined in Section~\ref{SS:quasimetric} is a metric.  	
\end{remark}
\begin{proof} First, we introduce some notation. The map $\pi_{2d-1}:\Cd\rightarrow \C^{d-1}\times\rl$ denotes the projection $(z_1,...,z_d)\mapsto (z',x_d)=(z_1,...,z_{d-1},x_d)$. 
The set $\pi_{2d-1}(C)$ is denoted by $\wt C$ for any $C\subset\Cd$.  
Whenever needed, the dependence of cuts, tubular cuts, caps and visibility regions on the underlying domain is be indicated via a superscript, as in $C^D(w;\de)$, $S^D(z;\de)$, etc. 

Let $\eps_D>0$ be as in Lemma~\ref{L:Model}. 
Let $\de^*_D>0$ be such that $B_w(\de^*_D)\subseteq U_w^{-1}(B_0(\eps_D))$ for all $w\in bD$, where $U_w$ is the unitary transformation granted by Lemma~\ref{L:Model}. Owing to the estimates \eqref{E:sandwich}, \eqref{E:normal} and \eqref{E:maxdepth}, there is a $\delta_D>0$ such that for all $\de<\de_D$, $t>0$, 
	\bes
		C(w;\de), S(w;\de), G_{t\de}(w;\de)\subset B_w(\de^*_D).
	\ees
In view of this, we assume that $w=0$ and $D=D^w=\{z\in B_0(\eps_D):y_d>f_w(z',x_d)\}$. 

We denote $f_w$ by $f$, and refer the reader to Remark~\ref{R:simplification} for some useful observations. It is also useful to note that, as $||(z',x_d)||\rightarrow 0$ and $\de\rightarrow 0$, 
	\begin{itemize}
	\item [$(a)$] $f(z',x_d)=f_{\balp,\bbet}(z')+O_D(||x_dz'||+|x_d|^2+||(z',x_d)||^3)$, where $f_{\balp,\bbet}$ is as in Remark~\ref{R:simplification},
	\item [$(b)$] $||\nabla f(z',x_d)||, ||\nabla f(z',x_d)-\nabla f_{\balp,\bbet}(z',x_d)||,  |\bdy_{x_d}f(z',x_d)|=O_D(||(z',x_d)||)$, 
	\item [$(c)$]  $\dfrac{\md (\pi_{2d-1})_*\sigma}{\md\lambda_{2d-1}}(z',x_d)=\left(1+||\nabla f(z',x_d)||^2\right)^{\frac{1}{2}}=1+O_D(||(z',x_d)||^2)$,
	\item [$(d)$] 
 $\sup_{z\in K_\de}\{||z'||^2,|x_d|\}=O_D(\de)$, where $K_\de$ is either $C(0;\de)$, $S(0;\de)$, or $G_{t\de}(0;\de)$, $t>0$.
	\end{itemize}
We now note a common theme in this proof. The quantities $\sigma\left(S(0;\de)\right)$, $\lambda\left(TC(0;\de)\right)$, and $\sigma\left(G_{t\de}(0;\de)\right)$ can be expressed as certain integrals (with respect to $\sigma$) on suitable subsets of $bD$. If these subsets are small, then due to $(b)$ and $(c)$ above, these integrals can be expressed as integrals (with respect to $\lambda_{2d-1}$) over the projections of these subsets onto $\{y_d=0\}$ without a significant loss, making them easier to tackle.

The estimates above suggest that it is useful to compute the volumes and areas of $\de$-cuts, $\de$-caps and $(\de,s)$-visibility regions at $0$ for model domains of the form 
\be\label{E:models}
		D_{\balp,\bbet}:=\left\{(z_1,...,z_d)\in\Cd:
			y_d>f_{\balp,\bbet}(z')\right\}.
	\ee 
These computations are relatively straightforward, and have been placed in Appendix~\ref{SS:Models}. 

We first prove \eqref{eq:cap}.   Since $S(0;\de)\subset bD^w$ and $y_d = f(z',x_d)$ for $z_d \in bD^w$,  $S(0;\de)$ is the graph of $f$ over 
	\bes
		\wt S(0;\de) =  \{(z',x_d):x_d^2+f(z',x_d)^2\leq\de^2\}.
	\ees
Similarly,  $S^{D_{\balp,\bbet}}(0;\de)$ is the graph of $f_{\balp,\bbet}$ over 
$$ \wt S^{D_{\balp,\bbet}}(0;\de)  = \{ (z',x_d): x_d^2+f_{\balp,\bbet}(z')^2\leq\de^2\}.  $$
Thus, by $(a)$ and $(d)$ above, there exist $a_D,\de_D>0$, such that 
	\be\label{E:nest1}
		\wt S^{D_{\balp,\bbet}}(0;\de-a_D\de^{3/2})\subseteq\wt S(0;\de)
			\subseteq \wt S^{D_{\balp,\bbet}}(0;\de+a_D\de^{3/2}),\quad  \de<\de_D. 
	\ee
Observe that when $S$ and $f$ are either $S(0;\de)$ and $f$, respectively, or $S^{D_{\balp,\bbet}}(0;\de)$ and $f_{\balp,\bbet}$, respectively, then
\bea\label{E:voltosa}
|\lambda_{2d-1}(\wt S)-\sigma(S)|
=\left|\int_{\wt S}d\sigma(z)-\int_Sd\sigma(z)\right|
&=&\int_{\wt S}\left(\sqrt{1+||\nabla f(z',x_d)||^2}-1\right) d\lambda_{2d-1}
\notag \\
&=&\lambda_{2d-1}(\wt S)O_D(\de)\quad \text{as}\ \de\rightarrow 0.
\eea Now, \eqref{eq:cap} follows from \eqref{E:nest1} and  \eqref{E:modelcap}  once we recall from Lemma~\ref{L:deffn} that $16\nu_D(w)=v_{\balp,\bbet}^2$.

Next, we prove \eqref{eq:cut} for $TC$ instead of $C$. Then, once we prove \eqref{eq:approxcut}, \eqref{eq:cut} will automatically follow. Recall that $TC(0;\de)=T_{r_0(\cdot;\de)}(S(0;\de))$. Now, setting $n(z)=\sqrt{1+||\nabla f(z',x_d)||^2}$ and $\eta_d(z)=-\bdy_{x_d}f(z',x_d)+i$, we have by \eqref{E:indepth} and \eqref{e:bdfeqn} that for $z\in S(0;\de)$,
{
	\bea
r_0(z;\de)&=&|\hat\eta_d(z)|^{-2}\left(
\sqrt{(\rea(\hat\eta_d(z)\overline{z_d}))^2+|\hat\eta_d(z)|^2(\de^2-|z_d|^2)}
	-\rea(\hat\eta_d(z)\overline{z_d}) \right)\notag\\
	&=&	\frac{\sqrt{(\rea(\eta_d(z)\overline{z_d}))^2+(1+(\bdy_{x_d}f(z',x_d))^2)(\de^2-x_d^2-y_d^2)}
		-y_d+x_d\bdy_{x_d}f(z',x_d)}{|\hat\eta_d(z)|^2n(z)}   \notag\\
	&=&	\frac{\sqrt{\de^2-x_d^2+(\de^2-y_d^2)(\bdy_{x_d}f(z',x_d))^2-2y_dx_d\bdy_{x_d}f(z',x_d)}
		-f_{\balp,\bbet}(z')+O_D(\de^{\frac{3}{2}})}{|\hat\eta_d(z)|^2n(z)}   \notag\\
		&=&\sqrt{\de^2-x_d^2}-f_{\balp,\bbet}(z')+O_D(\de^{5/4}), \label{E:depth2}  
	\eea }
as $\de\rightarrow 0$.
Here, we have used that $\sqrt{|a|}-\sqrt{|b|}\leq \sqrt{a+b}\leq \sqrt{|a|}+\sqrt{|b|}$, and $x_d,y_d,(\bdy_{x_d}f)^2=O_D(\de)$ on $S(0;\de)$. Thus, by Roccaforte's formula (see Theorem~\ref{thm:rf}) and arguing as in the proof of \eqref{eq:cap} above, 
	\beas
	\lambda\big(TC(0;\de)\big)
&=&\int_{\wt S(0;\de)} \left( \sqrt{\de^2-x_d^2}-f_{\balp,\bbet}(z') \right) \,d\lambda_{2d-1}(z',x_d)
		+O_D(\de^{d+5/4}),\\
&=&\int_{\wt S^{D_{\balp,\bbet}}(0;\de)} \left( \sqrt{\de^2-x_d^2}-f_{\balp,\bbet}(z') \right) \,d\lambda_{2d-1}(z',x_d)
	+O_D(\de^{d+5/4}),\\
	&=& \lambda\big(C^{D_{\balp,\bbet}}(0;\de)\big)+O_D(\de^{d+5/4})\qquad \text{as}\ \de\rightarrow 0.
	\eeas
Now, invoking Theorem~\ref{T:modelcomp}, we have that for some $\de_D,C_D>0$, 
	\bes
		\left|\lambda\big(TC(w;\de)\big)-
\frac{h_{d+1}}{2}k_d(w)\,\de^{d+1}\right|
	\leq C_D\de^{d+\frac{5}{4}},\qquad \de<\de_D. 
	\ees

To prove \eqref{eq:approxcut}, we must consider the portion of $C(0;\de)$ that is not contained in $TC(0;\de)$.  By Corollary~\ref{C:outertube}, we may assume that $C_D,\de_D>0$ are such that $C(0;\de)\subset T_{C_D\de}(S(0;C_D\de))$, and each $\zt\in C(0;\de)$ has a unique projection $z$ onto $S(0;C_D\de)$. Thus, $\zt=z+t\hat\eta(z)$ for some $z\in S(0;C_D\de)$ and $t<C_D\de$.  Writing $z=(z',x_d+iy_d)$, the condition $z+t\hat\eta(z)\in C(0;\de)$ holds if and only if $t>0$, and $|L(\zt,0)|<\de$, i.e.,
	\be\label{E:quadratic}
	x_d^2+y_d^2+\frac{t}{n(z)}(2y_d-2x_d\bdy_{x_d}f(z',x_d))
	+\frac{t^2}{n(z)^2}(1+(\bdy_{x_d}f(z',x_d))^2)^2< \de^2,
	\ee
where $n(z)=\sqrt{1+||\nabla f(z',x_d)||^2}$. The solution set of the above pair of inequalities is $(0,r_0(z;\de))$ when $x_d^2+y_d^2\leq \de^2$, i.e., when $z\in S(0;\de)$. In this case, $\zt\in TC(0;\de)$. On the other hand, if $x_d^2+y_d^2>\de^2$ and $y_d\geq x_d\bdy_{x_d}f(z',x_d)$, the solution set is empty. Thus, 
\bes
	C(0;\de)\setminus TC(0;\de)\subseteq
	T_{t_{\max}}\big(S''(\de)\big),
	\ees
where $S''(\de)=\{z\in S(0;C_D\de)\setminus S(0;\de):x_d\bdy_{x_d}f(z',x_d)-y_d>0\}$ and $t_{\max}(z)$ is the maximum positive solution of \eqref{E:quadratic} for $z=(z',x_d+if(z',x_d))$. As in the  proof of \eqref{E:S'} in Lemma~\ref{L:depth}, $\sup\{||(z',x_d)||:(z',x_d)\in S''(\de)\}=O_D(\de)$ as $\de\rightarrow 0$. Thus, $\lambda_{2d-1}(S''(\de))=O_D(\de^{2d-1})$
 Also, since 
	\bes
		t_{\max}(z)\leq \frac{(x_d\bdy_{x_d}f(z',x_d)-y_d)n(z)}{(1+(\bdy_{x_d}f(z',x_d))^2)^2}\qquad
	\text{when}\ z\in S''(\de), 
	\ees
$\sup_{S''(\de)}(t_{\max})=O_D(\de^2)$. Thus, 
\beas\lambda\big(C(0;\de)\setminus TC(0;\de)\big)=O_D(\sigma(S''(\de))\cdot \de^2)\\=O_D(\lambda_{2d-1}(S''(\de))\cdot \de^2)=O_D(\de^{2d+1})\eeas as $\de\rightarrow 0$. This completes the proof of \eqref{eq:approxcut}.  
 
Lastly, we prove \eqref{eq:vis}. {Note that $G_{t\de}(0;\de)=\{z\in bD:|L(0+t\de\hat\eta(0),z)|<\de\}$ is the graph of $f$ over 
	\beas
\wt G_{t\de}(0;\de)&=&\left\{(z',x_d):\left(t\de+q_1(z',x_d)\right)^2
+\left(x_d+q_2(z',x_d)\right)^2<\de^2n(z)^2 \right\},
	\eeas 
where 
\beas
q_1(z',x_d)&=&2\sum_{j=1}^{d-1}\rea( z_j\bdy_{z_j}f(z',x_d))-f(z',x_d)+x_d\bdy_{x_d}f(z',x_d),\\ q_2(z',x_d)&=&2\sum_{j=1}^{d-1}\ima( z_j\bdy_{z_j}f(z',x_d))-\bdy_{x_d}f(z',x_d)(t\de-f(z',x_d)).
\eeas
Similarly, $G^{D_{\balp,\bbet}}_{t\de}(0;\de)$ is the graph of $f_{\balp,\bbet}$ over 
	\bes
\wt G^{D_{\balp,\bbet}}_{t\de}(0;\de)
=\left\{(z',x_d):(t\de+q_1^{\balp,\bbet}(z',x_d))^2
+(x_d+q_2^{\balp,\bbet}(z',x_d))^2<\de^2n^{\balp,\bbet}(z)^2\right\},
	\ees
where $q_1^{\balp,\bbet}(z',x_d)=\sum_{j=1}^{d-1}\left(\alpha_j|z_j|^2+\beta_j\rea\zbar_j^2\right)$, $q_2^{\balp,\bbet}(z',x_d)=2\sum_{j=1}^{d-1}\beta_j\ima z_j^2$, and $n^{\balp,\bbet}(z',x_d)=\sqrt{1+4\sum_{j=1}^{d-1}|\alpha_jz_j+\beta_j\zbar_j|^2}$. Now, from $(a)$, we have that
\bes
q_j(z',x_d)=q_j^{\balp,\bbet}(z',x_d)+O_D(||(z',x_d)||),\quad j=1,2.  
\ees}
Thus, owing to $(d)$, there exist $a_D,\de_D>0$ such that
	\be
	\label{E:nestG}
\wt G_{t\de}^{D_{\balp,\bbet}}(0;\de(1-a_D\de^{\frac{1}{2}}))\subseteq \wt G_{t\de}(0;\de)
		\subseteq
 \wt G_{t\de}^{D_{\balp,\bbet}}(0;\de(1+a_D\de^{\frac{1}{2}})),\quad 0\leq t\leq 1-a_D\de^\frac{1}{2},
	\ee
for all $\de<\de_D$. Once again, as in \eqref{E:voltosa}, we have that  
\bes
|\lambda_{2d-1}(\wt G)-\sigma(G)|=\lambda_{2d-1}(\wt G)O_D(\de)\quad\text{as}\ \de\rightarrow 0
\ees
for both $G=G_{t\de}(0;\de)$ and $G=G_{t\de}^{D_{\balp,\bbet}}(0;\de)$, $t<1$. Moreover, since $h_d$ is Lipschitz on $[0,1]$, 
	\bes
\left|h_d\left(\frac{t}{1+O_D(\de^\frac{1}{2})}\right)-h_d\left(t\right)\right|	
			=O_D(\sqrt\de)\quad\text{as}\ \de\rightarrow 0,
	\ees
with constants independent of $t\in[0,1]$. Thus, the claim now follows from \eqref{E:nestG} and Theorem~\ref{T:modelcomp2}.
\end{proof}

Note that the estimate \eqref{eq:vis} has been obtained under the assumption that the depth function $\bde$ is a constant function. A more general estimate can be deduced from this special case. 

\begin{cor}\label{C:visibility}
Let $D$, $\nu_{D}$, $\kappa_d$, $h_d(t)$, and $k_d(w)$ be as in Theorem~\ref{T:CutsCapsGeneral}. Let $\bde=g\de$ for some positive function $g\in\cont^\alpha(bD)$, $\alpha\in(0,1)$, and $\de>0$. Then there exist $C_D,\de_D>0$ such that  for $w\in bD$ and $||\bde||_\infty<\de_D$, 
	\bea
	 \qquad 	 \left|\sigma\big(G_{t\bde(w)}(w;\bde)\big)- 
		 h_d(t)k_d(w)g(w)^d\,\de^d \right|\leq C_D||g||_\alpha^{d+1}\de^{d+\frac{\alpha}{2}},\ t\in \left[0,1-C_D\sqrt{\de}\right].
\label{eq:inhomvis} 
	\eea
\end{cor}
\begin{proof} First, as in the proof of \eqref{E:maxdepth}, there exist $C_D,\de_D>0$ such that for any $w\in bD$, $||\bde||_\infty<\de$, and $\zt\in G_{t\bde(w)}(w;\bde)$,
	\bes
		0< g(w)-C_D||g||_\alpha^2\de^{\frac{\alpha}{2}}\leq g(\zt)\leq g(w)+C_D||g||_\alpha^2\de^{\frac{\alpha}{2}}.
	\ees
Now, for a fixed $w\in bD$, let $\bde_1\equiv \de_1$ and $\bde_2\equiv \de_2$ be constant functions on $bD$ given by
	\beas
\de_1&:=& g(w)\de\left(1-C_D||g||_\alpha^2\de^{\frac{\alpha}{2}}\right),\\
\de_2&:=& g(w)\de\left(1+C_D||g||_\alpha^2\de^{\frac{\alpha}{2}}\right). 
	\eeas
Then,
	\bes
			G_1:=G_{t_1\bde_1(w)}(w;\bde_1)\subseteq G_{t\bde(w)}(w;\bde)
	\subseteq G_{t_2\bde_2(w)}(w;\bde_2)=:G_2,
	\ees 
where $t_1=t(1-C_D||g||_\alpha^2\de^{\frac{\alpha}{2}})^{-1}$ and $t_2=t(1+C_D||g||_\alpha^2\de^{\frac{\alpha}{2}})^{-1}$. Since $t\in[0,1-C_D\sqrt{\de}]$, shrinking $\de_D>0$ if necessary, we obtain that $t_j\in [0,1-C_D\sqrt{\de_j}]$ for $j=1,2$. Thus, the estimate \eqref{eq:vis} applies to both $G_1$ and $G_2$, and we obtain \eqref{eq:inhomvis}.
\end{proof}

\subsection{Containment and coverage}\label{SS:quasimetric} 

In this section, we express the containment condition on an induced polyhedron $P$ {equivalently} as a purely {boundary}-based coverage condition on the caps of $P$. In \cite{LaSt}, the authors consider a natural quasimetric on the boundary of a strongly $\C$-convex domain. We consider an appropriate modification of their construction. The balls in this quasimetric coincide with caps of induced polyhedra. This allows us to convert the containment condition on an induced polyhedron into a covering condition on $bD$.

\begin{lemma}[{\cite[Proposition 3.3]{LaSt}}] \label{L:quasimetric} Let $D$ and $L$ be as in Definition~\ref{D:geom}. The function $\mathtt d:bD\times bD\rightarrow [0,\infty)$ given by
\bes	
	\mathtt d(\zt,w)=|L(\zt,w)|^{\frac{1}{2}}
\ees
is a quasimetric on $bD$, i.e., there is a $\mathfrak q_D\geq 1$ such that 
	\begin{enumerate}
		\item [$(a)$] $\mathtt d(\zt,w)=0$ if and only if $\zt=w$,
		\item [$(b)$] (quasi-symmetry) $\mathtt d(\zt,w)\leq \mathfrak q_D\, \mathtt d(w,\zt)$ for all $\zt,w\in bD$, and
		\item [$(c)$] (quasi-triangle inequality) $\mathtt d(\zt,w)\leq \mathfrak q_D(\mathtt d(\zt,z)+\mathtt d(z,w))$, for all $z,w,\zt\in bD$. 
	\end{enumerate}
\end{lemma}

\begin{proof} Property $(a)$ follows from \eqref{E:sandwich}. To obtain $(b)$, note that by the $\cont^2$-smoothness of $\rho$,
	\bes
		|\mathtt d(\zt,w)^2-\mathtt d(w,\zt)^2|\leq \big||L(\zt,w)|-|L(w,\zt)|\big|
	=\left|\left<\frac{\bdy\rho(w)}{||\bdy\rho(w)||}
	-\frac{\bdy\rho(\zt)}{||\bdy\rho(\zt)||},\zt-w\right>\right|\lesssim ||\zt-w||^2.
	\ees
Now, the claim follows from \eqref{E:sandwich}. Finally, for claim $(c)$, observe that
	\beas
		&&\mathtt d^2(\zt,w)-\mathtt d^2(z,w)-\mathtt d^2(\zt,z)\\
&&\leq \left|\left<\frac{\bdy\rho(w)}{||\bdy\rho(w)||},w-\zt\right>
-\left<\frac{\bdy\rho(w)}{||\bdy\rho(w)||},w-z\right>
-\left<\frac{\bdy\rho(z)}{||\bdy\rho(z)||},z-\zt\right>\right|\\
&&\lesssim
\left|\left<\bdy\rho(w)-\bdy\rho(z),z-\zt\right>\right|\\
&&\lesssim ||w-z|| \times ||z-\zt||\\
&&\lesssim \mathtt d(z,w)\mathtt d(\zt,z),
\eeas
where the final two steps follow from the $\cont^2$-smoothness of $\rho$ and \eqref{E:sandwich}, respectively. Thus, we obtain $(c)$.
\end{proof}

We now observe that the containment condition $P(\varphi;\de)\subset D$ can indeed be reduced to a covering condition on $bD$ in terms of the caps of $P(\varphi;\de)$, which are $\mathtt d$-balls since 
	\bes
		S(w;\de):=\{\zt\in bD:\mathtt d(\zt,w)\leq \sqrt{\de}\}.
	\ees

\begin{lemma}\label{L:inc=cov}
Let $D$ be as in Definition~\ref{D:geom}, $\varphi=\{w^1,...,w^n\}\subset bD$, and $\bde:bD\rightarrow(0,\infty)$. Then, $P(\varphi;\bde)$ is contained in $D$ if and only if $bD=S(\varphi;\bde)$, where
	\bes
		S(\varphi;\bde):=\bigcup_{j=1}^{n}S(w^j;\bde(w^j)).
	\ees
\end{lemma}

\begin{proof}Recall that $P(\varphi;\bde)$ is the union of those connected components of 
\bes
	\wt P(\varphi;\bde)=
		\bigcap_{j=1}^{n}\left\{z\in \Cd:\left|L(z,w^j)\right|>\bde(w^j)\right\}
	\ees
that intersect $D$. Since $S(\varphi;\bde)=bD\setminus \wt P(\varphi;\bde)$, $bD\neq S(\varphi;\bde)$ if and only if $bD\cap \wt P(\varphi;\bde)$ is nonempty. However, if $z\in bD\cap \wt P(\varphi;\bde)$, then $z$ must be contained in a connected component of $\wt P(\varphi;\bde)$ that intersects $D$, i.e., $ bD\cap \wt P(\varphi;\bde)=bD\cap P(\varphi;\bde)$. Thus, it suffices to show that $P(\varphi;\bde)\subset D$ if and only if $bD\cap P(\varphi;\bde)$ is empty. The forward implication is trivial. For the converse, we argue by contraposition: if $P$ is a connected component of $P(\varphi;\de)$ that is not contained in $D$, then it must intersect $bD$. Otherwise $P$ is the disjoint union of the nonempty open sets $P\cap D$ and $P\cap(\Cd\setminus \overline D)$ which is impossible.
\end{proof}

In order to study the covering problem on $bD$, we need the existence of a `good' cover of $bD$. The existence of this cover is due to the fact that balls in the quasimetric $\mathtt d_D$ satisfy a Vitali-type covering lemma; see \cite[Lemma 4.1.1]{KrPa} for a proof that generalizes to quasimetrics. 

\begin{lemma}\label{L:Vitali}  Let $D$ be as in Definition~\ref{D:geom}. There is a constant $0< \mathfrak K=\mathfrak K_D \leq 1/9$ such that for any $\de>0$, there exist $m=m_\de\in\N$ and $p^1,...,p^m\in bD$ such that 
	\begin{itemize}
\item [$(a)$] $bD=S(p^1;\de)\cup\cdots\cup S(p^m;\de)$, and
\item [$(b)$] $S(p^1,\mathfrak K \:\de),...,S(p^m; \mathfrak K\:\de)$ are pairwise disjoint.
\end{itemize}
\end{lemma}
The dilation constant in the above lemma depends on the quasimetric constant $\mathfrak q_D$.  In the case of the unit ball $\mathbb B^d$, $\mathtt d$ is in fact a metric on $b\mathbb B^d$ (e.g., see \cite[Section~5.1]{Ru08}), so we may choose $\mathfrak q_{\mathbb B^d}$ and $\mathfrak K_{\mathbb B^d}$ as $1$ and $1/9$,  respectively. 

\subsection{Nearest-neighbor estimates}\label{SS:NN}

So far,  we have focussed on estimates for individual cuts and caps, but for variance bounds and normal approximation results,  it is crucial to understand volume and area contributions of intersecting cuts and caps. In this subsection, we collect a couple of such estimates.

First, we quantify the following statement: cuts based at points that are sufficiently far away from each other (in $\mathtt d$)  cannot intersect.  

\begin{Prop}\label{P:NN1} Let $D$ be as in Definition~\ref{D:geom}. There are constants $\de_D,A_D>0$ such that 
	\bes
		\sigma\{\zt\in bD:C(\zt;\de)\cap C(w;\de)\neq \emptyset\}\leq A_D\de^{d}
			\qquad  \forall w\in bD,\ \de<\de_D. 
	\ees
\end{Prop}
\begin{proof}  Let $C_D,\de_D$ be as in Corollary~\ref{C:outertube}. Let $w,\zt\in bD$ be such that $C(w;\de)\cap C(\zt;\de)\neq \emptyset$ for some $\de<\de_D$. Then, $T_{C_D\de}(S(w;C_D\de))\cap T_{C_D\de}(S(\zt;C_D\de))\neq \emptyset$. We may assume that each $z\in T_{C_D\de}(bD)$ admits a unique projection onto $bD$. Then, there is a $z\in S(w;C_D\de)\cap S(\zt;C_D\de)$. By the quasi-symmetry of $\mathtt{d}$, $\zt\in S(w;C_D\de)$, and the claim follows from \eqref{eq:cap}.
\end{proof}

Next, we bound (from below) the volume increment caused by adding a cut at a point that is sufficiently isolated from the rest of the source points. Note that this result is stated for tubular cuts, which suffices for our purposes. 

\begin{Prop}\label{P:NN2} Let $D$ be as in Definition~\ref{D:geom}. Let $\bde=g\de$ for some positive function $g\in\cont^\alpha(bD)$, $\alpha\in(0,1)$, and $\de>0$. There exist $b_{D}, B_{g,D}, t_{g,D}>0$ such that for all $w\in bD$ and $t<t_{g,D}$,
	\bes
		\lambda\left(TC(w;\bde(w))\setminus\bigcup_{\zt\in bD\setminus S(w;t\bde(w))}TC(\zt;\bde(\zt))\right)
			\geq b_D\left(t^2\bde(w)\right)^{d+1},\quad 
				 \de<B_{g,D}t^{4/\alpha}.
	\ees
\end{Prop}
\begin{proof} Let $C_D$ and $\de_D$ be as in the proof of Proposition~\ref{P:NN1}. Observe that for any $w\in bD$,
	\bes
		\{\zt\in bD:TC(w;\bde(w))\cap TC(\zt;\bde(\zt))\neq \emptyset\}\subset\{\zt\in bD:C(w;||g||_\infty\de)\cap C(\zt;||g||_\infty\de)\neq \emptyset\},
	\ees
and the latter set is contained in $S(w;C_{g,D}\de)$, as shown in the proof of Proposition~\ref{P:NN1}. Thus, it suffices to work with $\zt\in S(w;C_{g,D}\de)\setminus S(w;t\bde(w))$ to establish the claim. 

We now bound from below $r_w(z;\bde(w))-r_\zt(z;\bde(\zt))$ for certain $z$ and $\zt$ near $w$. Paraphrasing \eqref{E:depthest} in terms of the quasimetric $\mathtt d$, we have that for $w\in bD$, $z\in S(w;\bde(w))$, and $\de<\de_{g,D}$,
	\be\label{E:depthest2}
			\bde(w)-\mathtt d(z,w)^2\leq r_w(z;\bde(w))\leq
			\sqrt{\bde(w)^2-\mathtt d(z,w)^4}+C_{g,D}\de^{2}.
	\ee
Fix $s,t>0$ such that $s<t<1$. Then, if $z\in S(w; s\bde(w))$ and $\zt\in bD\setminus S(w;t\bde(w))$, 
	\bes
		\sqrt{t\bde(w)}\leq \mathtt d(\zt,w)\leq 
		\mathfrak q^2\mathtt d(z,\zt)+\mathfrak q\mathtt d(z,w)
		\leq  \mathfrak q^2\mathtt d(z,\zt)+\mathfrak q^2\sqrt{s\bde(w)},
	\ees
where $\mathfrak q=\mathfrak q_D\geq 1$ is the quasi-metric constant in Lemma~\ref{L:quasimetric}. Thus, 
	\bes
	\mathtt d(z,\zt)\geq\left(\mathfrak q^{-2}\sqrt{t}-\sqrt{s}\right)
	\sqrt{\bde(w)}>0, \qquad \text{for}\ s<t/\mathfrak q^4.
	\ees 
Now, setting $s=t^2/4\mathfrak q^8$, it follows from \eqref{E:depthest2} and the lower bound on $\mathtt d(z,\zt)$ that for $\zt \in S(w;C_{g,D}\de)\setminus S(w;t\bde(w))$ and $z\in S(w;s\bde(w))\cap S(\zt;\bde(\zt))$, 
	\be\label{E:depthest1}
		r_w(z;\bde(w))\geq \bde(w) R(t)\qquad \text{and}\qquad	
		r_\zt(z;\bde(\zt))\leq \bde(w) r_{\zt}(t), 
	\ee
where $R(t)=1-t^2/4\mathfrak q^8$ and 
		$r_{\zt}(t)=\sqrt{\frac{g(\zt)^2}{g(w)^2}-\left(\frac{\sqrt{t}}{\mathfrak q^2}-\frac{t}{2\mathfrak q^{4}}\right)^4}+C_{g,D}\de$. 

 Now, by the assumption on $g$, we obtain that $|g(\zt)-g(w)|\leq C_{g,D}\de^{\frac{\alpha}{2}}$, for all $w\in bD$, $\zt\in S(w;C_{g,D}\de)$, and $\de<\de_{g,D}$. Thus, for $w\in bD,\zt\in S(w;C_{g,D}\de)$, and $\de<\de_{g,D}$,
	\bes 
		g(\zt)^2\leq g(w)^2+C_{g,D}\de^{\frac{\alpha}{2}}.
	\ees
Therefore,
	\bea\label{E:depthest3}
		r_\zt(t)&\leq& \sqrt{1+C_{g,D}\de^{\frac{\alpha}{2}}
		-\left(\frac{\sqrt{t}}{\mathfrak q^2}-\frac{t}{2\mathfrak q^{4}}\right)^4}
			+C_{g,D}\de\notag \\&\leq& 1
		-\frac{1}{2}\left(\frac{\sqrt{t}}{\mathfrak q^2}-\frac{t}{2\mathfrak q^{4}}\right)^4
			+C_{g,D}\de^{\frac{\alpha}{2}}=:r(t),
	\eea
where we have used that $\sqrt{1-b+a}\leq  1-\frac{b}{2}+a$, whenever $b<3/4$ and $a>0$. Now, for sufficiently small $t$,
	\beas
		R(t)-r(t)= \left(\frac{(\sqrt{t}-t/2\mathfrak q^2)^4}{2\mathfrak q^8}-\frac{t^2}{4\mathfrak q^8}-C_{g,D}\de^{\frac{\alpha}{2}}\right)			
		\geq \left(\frac{t^2}{8\mathfrak q^8}-C_{g,D}\de^{\frac{\alpha}{2}}\right).
	\eeas
Thus, there are $a_{D}, A_{g,D}, t_{g,D}>0$ so that $R(t)-r(t)\geq a_Dt^2\bde(w)$ for all $\de<A_{g,D}t^{4/\alpha}$, $t<t_{g,D}$. 

Now, let $\pi_{bD}(\xi)$ be the (uniquely determined) projection of $\xi$ onto $bD$ for $\xi$ sufficiently close to $bD$. Define	$T:=T_{R(t)\bde(w)}(S(w;s\bde(w))) \setminus T_{r(t)\bde(w) }(S(w;s\bde(w)))$,
where $T_*(\cdot)$ is as in Definition~\ref{D:geom}.  By \eqref{E:depthest1}, we have that 
\beas
	T_{R(t)\bde(w)}\big(S(w;s\bde(w))\big) &\subset& TC(w;\bde(w)),\\ 
	T_{r(t)\bde(w)}\big(S(w;s\bde(w))\big) &\supset& 
	\bigcup_{\zt\in bD\setminus S(w;t\bde(w))}\left\{\xi\in TC(\zt;\bde(\zt))
	:\pi_{bD}(\xi)\in S(w;s\bde(w))\right\}. 
\eeas
Thus,  $T  \subset TC(w;\bde(w))\setminus\bigcup_{\zt\in bD\setminus S(w;t\bde(w))}TC(\zt;\bde(\zt))$, and it suffices to produce a lower bound for $\lambda(T)$. We will use Roccaforte's formula, i.e., Theorem~\ref{thm:rf} for this.

From \eqref{eq:cap},  $\sigma(S(w;s\bde(w)))\geq \wt a_D(t^2\bde(w))^d$ for some $\wt a_D>0$.  Moreover, $|R(t)|, |r(t)|\leq 2$ for all $t < t_{g,D}$, and the functions $s_1,...,s_{2d-1}$ as in Theorem \ref{thm:rf} are bounded on $bD$.  Thus, by Theorem \ref{thm:rf} and the above estimates, there are constants $b_D,B_{g,D}>0$ so that 
				\beas
	\lambda(T)&=&\bde(w)(R(t)-r(t)) \int\limits_{S(w;s\de)} \left(1-\sum_{j=1}^{d-1}\frac{(-\bde(w))^{j+1}}{j+1}
		\frac{R(t)^{j+1}-r(t)^{j+1}}{R(t)-r(t)}s_j(z)\right)d\sigma(z)\\
		&\geq & b_D \bde(w)(R(t)-r(t))\sigma(S(w;s\bde(w))) 
			\geq  b_D(t^2\bde(w))^{d+1}
	\eeas
for all $\de<B_{g,D}t^{4/\alpha}$ and $t<t_{g,D}$. 
\end{proof}


\section{Proofs of limit theorems}
\label{s:proofsmain}
In this section,  we present the  proofs of the main results of this paper, namely those stated in Section~\ref{s:poissonrandpoly}. We first collect some useful observations regarding the notions introduced in Definition~\ref{D:geom}.
We use the notation introduced in Definition~\ref{D:geom}.

Let $D \subset\Cd$ be a strongly $\C$-convex $\cont^2$-domain as in Definition~\ref{D:geom}, $\varphi = \{w^1,\ldots,w^m\} \subset bD$ be a source set, and $\bde : bD \to (0,\infty)$ be a size function. Since $bD$ and $\{z\in\Cd:|L(z,w)|=\de\}$, $w\in bD$, $\de>0$, are smooth real hypersurfaces in $\Cd$, they have zero Lebesgue volume. Thus,
	\be\label{E:equalvol}
		\lambda(D\setminus P(\varphi;\bde))=\lambda\left(\overline D\setminus P(\varphi;\bde)\right)=\lambda\left(D\setminus\overline{P(\varphi;\bde)}\right).
	\ee
Note that
	\be\label{E:unioncuts}
		D\setminus \overline {P(\varphi;\bde)}=\bigcup_{j=1}^m\left\{z\in D:|L(z,w^j)|< \bde(w^j)\right\}= \bigcup_{j=1}^m C(w^j;\bde(w^j)).
	\ee
On the other hand, recalling that $\overline{C(w;\de)}\supset TC(w;\de)$, we have that 
	\be\label{E:uniontubcuts}
		\overline D\setminus P(\varphi;\bde)=\bigcup_{j=1}^m\left\{z\in\overline D:|L(z,w^j)|\leq \bde(w^j)\right\}\supset \bigcup_{j=1}^m TC(w^j;\bde(w^j)).
	\ee
Setting
	\bes
		C(\varphi;\bde):= \bigcup_{j=1}^m C(w^j;\bde(w^j))\qquad 
\text{and}\qquad
		TC(\varphi;\bde):= \bigcup_{j=1}^m TC(w^j;\bde(w^j)),
	\ees
we obtain from \eqref{E:equalvol}, \eqref{E:uniontubcuts}, and \eqref{E:unioncuts} that 
	\be\label{E:upperlowerbounds}
		\lambda\left(TC(\varphi;\bde)\right)\leq \lambda\left(D\setminus P(\varphi;\bde)\right)
		=\lambda\left(C(\varphi;\bde)\right)=\lambda\left(\overline{C(\varphi;\bde)}\right).
	\ee
In order to compute $\lambda\left(TC(\varphi;\bde)\right)$ using the tube formula in Theorem~\ref{thm:rf}, we express $TC(\varphi;\bde)$ as a $g$-tube of $bD$ for some non-negative $g$. For this, set $r^+(z)=\max\{r(z),0\}$, where $r(z):=\max_{1\leq j\leq m}r_{w^j}(z;\bde(w^j))$. Then, $
T_{r}=TC(\varphi;\bde)$ and $T_{r^+}=TC(\varphi;\bde)\cup bD$.
Thus, 
\be\label{E:tubevol}
\lambda(TC(\varphi;\bde))=\lambda(T_r)=\lambda(T_{r^+}).
\ee 
Finally, we note that \eqref{E:upperlowerbounds} helps estimate $\delta_V(P(\varphi ; \bde),D)$ in terms of volumes of cuts since 
\begin{eqnarray}
\label{e:delvvolp}
\delta_V(P(\varphi ; \bde),D)  -  \lambda(D \setminus P(\varphi ; \bde)) &=& \lambda(P(\varphi ; \bde) \cap D) \1[P(\varphi ; \bde) \not \subset D] \nonumber \\&\leq&  \lambda(D) \1[P(\varphi ; \bde) \not \subset D].
\end{eqnarray}

In all the proofs below, we first handle the case when the source set of the random polyhedron is $\eta_n$, a Poission point process. Subsequently, the changes required to tackle the case of the source set being $\beta_n$, a binomial point process, is discussed. For brevity, we set
	\bes
		\de_V(n):=\de_V\left(D,P(\varphi_n,\bde_n)\right),
	\ees
where $\varphi_n$ is either $\eta_n$ or $\beta_n$, depending on the context. We also denote the probability measure $f\sigma$ by $\sigma_f$.

\subsection{Theorem~\ref{thm:cippoi}: convergence in probability}
\label{s:proofcvgprob}

\begin{proof}[Proof of Theorem \ref{thm:cippoi}]
Let $\eta_n$ be a Poisson point process on $bD$ such that \eqref{eq:covdelta} holds, i.e.,
	\bes
		\mathbb P(P(\eta_n;\bde_n)\subset D)\rightarrow 1\ \text{as}\ n\rightarrow\infty. 
	\ees

\noindent {\em Step 1.  Approximating by tubular cuts}. First, we observe that by the triangle inequality 
\bea\label{e:trineq}
		\left|\frac{\delta_V(n)}{\delta_n}-\int_{bD}g(z)d\sigma(z) \right |
	\leq\left|\frac{\delta_V(n)-\lambda(C(\eta_n;\bde_n))}{\delta_n}\right|
+
\left|\frac{\lambda(C(\eta_n;\bde_n))-\lambda(TC(\eta_n;\bde_n))}{\delta_n}\right|
&&\notag \\
 +\left|\frac{\lambda(TC(\eta_n;\bde_n))}{\delta_n}-\int_{bD}g(z)d\sigma(z) \right|.&&
	\eea 
Now, from \eqref{E:upperlowerbounds} and \eqref{e:delvvolp},  we have that for any $\epsilon > 0$,
\begin{align*}
 \BP\left( \left|\frac{\delta_V(n) - \lambda(C(\eta_n;\bde_n))}{\delta_n} \right| > \epsilon \right)=
 \BP\left( \frac{\delta_V(n) - \lambda(C(\eta_n;\bde_n))}{\delta_n} > \epsilon \right)& \leq \BP\left(P(\eta_n;\bde_n) \not \subset D \right).
\end{align*}
Thus, by the containment assumption on $P(\eta_n;\bde_n)$,
\begin{equation}
\label{e:eqcvgprob0}
\left|\frac{\delta_V(n) - \lambda(C(\eta_n;\bde_n))}{\delta_n} \right| \stackrel{p}{\to} 0\ \text{as}\ n\rightarrow\infty.
\end{equation}
Next, invoking \eqref{E:upperlowerbounds} and \eqref{eq:approxcut}, we have that
\be\label{e:C-TC}
0\leq \lambda(C(\eta_n;\bde_n)) - \lambda(TC(\eta_n ; \bde_n)) \leq \sum_{i=1}^{|\eta_n|} \lambda\big(C(W^i ; \bde_n) \setminus TC(W^i ; \bde_n)\big) \leq C_D |\eta_n| \|\bde_n\|_{\infty}^{2d+1},  
\ee
for some $C_D>0$. Thus, taking expectation on both sides, we have that
\be
\label{e:volcutapproxcut}
\BE\left|\frac{\lambda(C(\eta_n;\bde_n)) - \lambda(TC(\eta_n ; \bde_n)) }{\delta_n}\right| \leq C_D \|g\|_{\infty}^{2d} n\delta_n^{2d}
=C_D\|g\|_{\infty}^{2d}\frac{(\log n)^{2}}{n} \to 0
\ee
as $n \to \infty$.  Thus, by Markov's inequality, \eqref{e:trineq}, \eqref{e:eqcvgprob0}, and \eqref{e:volcutapproxcut}, \eqref{e:cipdvpoi} follows if
\begin{equation}
\label{e:cipCAn}
\BE\left | \frac{ \lambda(TC(\eta_n ; \bde_n))}{\delta_n} - \int_{bD}g(z)\md \sigma(z) \right|  \rightarrow 0\ \text{as}\ n\rightarrow\infty. 
\end{equation}
This reduces to estimating the expectation of a depth function associated to $P(\eta_n;\bde_n)$. 
\medskip 

\noindent {\em Step 2. Replacing volume by depth.} For $z \in bD$, let  $r_n(z) := \max_{1\leq j\leq |\eta_n|}r_{W^j}(z ; \bde_n(W^j))$ and $r_n^+(z) = \max \{r_n(z), 0\}.$ Then, by \eqref{E:tubevol} the statement \eqref{e:cipCAn} is equivalent to 
\begin{equation}
\label{e:expfncvg}
\BE \left | \frac{ \lambda(T_{r^+_n})}{\delta_n} -  \int_{bD} g(z)  \md \sigma(z) \right | \to 0\ \text{as}\ n\rightarrow\infty. 
\end{equation}
Using the tube formula in Theorem \ref{thm:rf} we have that
\begin{equation}
\label{e:tubeformCAn}
\lambda(T_{r^+_n}) =  \sum_{j=0}^{d-1}\frac{(-1)^j}{j+1}\int_{bD} r^+_n(z)^{j+1}s_j(z)\, \md \sigma(z). 
\end{equation}
Now, recalling that $r_n^+(z)\leq r_\infty(z;\bde_n)$ for all $z\in bD$, and that $r_\infty(z;\bde_n) \leq C \delta_n$ (see \eqref{E:maxdepth}), we have that for any $j \geq 1$,
$$ \frac{1}{\delta_n}\left| \int_{bD} r^+_n(z)^{j+1}s_j(z)\md \sigma(z)\right| \leq  \int_{bD} \frac{r_{\infty}(z;\bde_n)^{j+1}}{\delta_n} |s_j(z)|\,\md \sigma(z)  \to  0\ \text{as}\ n\rightarrow\infty.
$$
Thus, since $s_0 \equiv 1$, to prove \eqref{e:expfncvg}, it suffices to show that
\begin{equation}
\label{e:intrncvg}
\BE \left | \frac{1}{\delta_n} \int_{bD} r^+_n(z) \,\md \sigma(z) - \int_{bD} g(z)  \md \sigma(z) \right |\to 0\ \text{as}\ n\rightarrow\infty. 
\end{equation}
Now, by \eqref{E:maxdepth}, $
 r^+_n(z) \leq r_\infty(z;\bde_n)\leq  g(z)\delta_n + C_D \parallel g \parallel_\alpha^2\delta_n^{1+\frac{\alpha}{2}}$ for some $C_D>0$. 
Therefore,
\begin{align*}
\left | \frac{1}{\delta_n} r^+_n(z) - g(z) \right | & \leq \left | \frac{1}{\delta_n} r^+_n(z) - g(z) -  C_D \parallel g \parallel_\alpha^2\delta_n^{\frac{\alpha}{2}}\right | + C_D \parallel g \parallel_\alpha^2\delta_n^{\frac{\alpha}{2}}
\\
&= g(z) -\frac{1}{\delta_n} r^+_n(z) + 2C_D \parallel g \parallel_\alpha^2\delta_n^{\frac{\alpha}{2}}.
\end{align*}
Thus, by the triangle inequality, Fubini's theorem, and the dominated convergence theorem, \eqref{e:intrncvg} follows if we show that for all $z \in bD$,
\begin{equation}
\label{e:rncvg}
{ {
 \lim_{n \to \infty} \frac{\BE[r^+_n(z)]}{\delta_n} = g(z).
 }}
\end{equation}
\medskip

\noindent{\em Step 3. Expectation of depth.} Fix $z \in bD$. Since $\{r_n(z) \geq s\} = \{r^+_n(z) \geq s\}$ for $s > 0$, we have that,
$$  \BE[r^+_n(z)]  = \int_0^{r_{\infty}(z;\bde_n)} \BP(r_n(z) \geq s) \md s.$$
Thus, since $r_\infty(z;\bde_n)\geq g(z)\de_n$, we have that
\begin{equation}
\label{eq:exprnzbds}
 0\leq \BE[r^+_n(z)] -  \int_0^{g(z)\delta_n} \BP(r_n(z) \geq s) \md s\leq (r_{\infty}(z;\bde_n) - g(z)\delta_n).
\end{equation}
By \eqref{E:maxdepth}, it suffices to prove \eqref{e:rncvg} with $\BE[r^+_n(z)]$ replaced by the integral in \eqref{eq:exprnzbds}.

Now, recall that $\{w\in bD:r_w(z;\bde(w)) \geq s\}  = G_s(z ; \bde)$, the $(\bde,s)$-visibility region at $z$. Thus,
\be\label{e:vis}
\BP\left(r_{W^i}(z;\bde(W^i)) < s\right)
	 = \BP(W^i \notin  G_s(z ; \bde)) = 1 -  \sigma_f(G_s(z ; \bde)),
\ee
where recall $\sigma_f = f \md \sigma$ is the probability distribution of the $W^i$'s. Therefore,
\beas
\BP(r_n(z) \geq s) 
= 1 -  \BP(r_n(z) < s) 
&=&1 - \BP\left(\bigcap_{i=1}^{|\eta_n|}\{r_{W^i}(z ; \bde_n(W^i)) < s\}\right) \\
&=&1 -   \BE\left[ \prod_{i=1}^{|\eta_n|}\BP\Big(r_{W^i}(z ; \bde_n(W^i)) < s\Big)\right], 
\eeas
where, in the last step, we have conditioned on  $|\eta_n|$ and used the independence of $|\eta_n|$ and $W^i$'s. Now, by \eqref{e:vis}, and the fact that $\BE[t^{|\eta_n|}] = e^{ -n (1 - t)}$ for $t  \in \R$, we have that
\bea
\int_0^{g(z)\delta_n} \BP(r_n(z) \geq s) \md s 
&=& \int_0^{g(z)\delta_n}  \left(1 -  \BE\Big[ (1 -  \sigma_f(G_s(z ; \bde_n)))^{|\eta_n|}\Big] \right) \md s \notag\\
&=& \int_0^{g(z)\delta_n} \big( 1 - e^{-n\sigma_f(G_s(z ; \bde_n))} \big) \md s  \notag\\
&=& g(z)\delta_n \int_0^1   \left( 1 - e^{-n\sigma_f(G_{tg(z)\delta_n}(z ; \bde_n))} \right) \md t,  \label{e:rnzintbd}
\eea
where, in the final step, we use the change of variable $s\mapsto tg(z)\delta_n$.
Now, setting $c_f := \inf_{w \in bD}f(w)$, we have that
$$ c_f \sigma(G_{tg(z)\delta_n}(z ; \bde_n)) \leq \sigma_f(G_{tg(z)\delta_n}(z ; \bde_n)) \leq \|f\|_{\infty} \sigma(G_{tg(z)\delta_n}(z ; \bde_n)).$$
Now, recall that \eqref{eq:inhomvis} grants constants $C_D,K_D>0$ such that for $n$ large and $t \leq 1 - C_D \sqrt{\delta_n}$,  

$$ |  \sigma(G_{tg(z)\delta_n}(z ; \bde_n))  - h_d(t)k_d(z)g(z)^d \delta_n^d| \leq K_D \delta_n^{d+\alpha/2}.$$
Therefore, since $\delta_n\rightarrow 0$ as $n\rightarrow\infty$, we obtain that for any $t< 1$,    
   \begin{equation} \label{eq:nsf1} n \sigma_f(G_{tg(z)\delta_n}(z ; \bde_n)) \to \infty\ \text{as}\ n \to \infty.
     \end{equation}
In other words, $1 - e^{-n\sigma_f(G_{tg(z)\delta_n}(z ; \bde_n))} \to 1$ as $n \to \infty$ for any $t< 1$. Thus, by \eqref{e:rnzintbd} and the dominated convergence theorem, we have that 
\be\label{eq:rneq}
\lim_{n \to \infty}\frac{1}{\delta_n}\int_0^{g(z)\delta_n} \BP(r_n(z) \geq s) \md s
=\lim_{n \to \infty}g(z) \int_0^1   \left( 1 - e^{-n\sigma_f(G_{tg(z)\delta_n}(z ; \bde_n))} \right) \md t 
= g(z).
\ee
Substituting \eqref{eq:rneq} into \eqref{eq:exprnzbds} yields \eqref{e:rncvg}.  This completes the proof in the case of $\eta_n$.  
\medskip

Now, assume that $\beta_n$ is a binomial point process satisfying (the analogue of) \eqref{eq:covdelta}. The proof of \eqref{e:cipdvpoi} for $\beta_n$ is identical to the above argument until \eqref{eq:exprnzbds} as one does not rely on the properties of Poisson point processes until that stage. For the rest of the proof, it suffices to show that
\be\label{e:limit}
\BP(r_n(z) \geq tg(z)\de_n)\rightarrow 1\ \text{as}\ n\rightarrow\infty
\ee
for any $t<1$, where $r_n(z)=\max_{1\leq j\leq n}r_{W^j}(z;\bde_n(W^j))$. Now,
\beas
\BP(r_n(z) < tg(z)\de_n) &= \BP( \cap_{i=1}^{n}\{r_{W^i}(z ; \bde_n(W^i)) < tg(z)\de_n\})
= \big(1 - \sigma_f(G_{tg(z)\de_n}(z;\bde_n)) \big)^n,
\eeas
which converges to $0$ as $n\rightarrow\infty$ for any $t<1$ by \eqref{eq:nsf1}. Thus, \eqref{e:limit} holds, and the proof is complete.
\end{proof}

\subsection{Theorem~\ref{t:expvarpoi}: $L^1$-convergence and variance bounds}

\label{s:proofexpvar}

\begin{proof}[Proof of  Theorem \ref{t:expvarpoi}] Let $\eta_n$ be a Poisson process on $bD$ satisfying \eqref{e:covexpvarpoi}, i.e., 
\bes
	 \BP(P(\eta_n; \bde_n) \not\subset  D) = o\left(n^{-\left(1+\frac{2}{d}+\epsilon\right)}\right)\ \text{as}\ n\rightarrow\infty,
\ees
for some $\eps>0$. Recall that $\delta_V(n)=\delta_V(P(\eta_n; \bde_n),D)$. \\

\noindent{\em Part 1.  $L^1$-Convergence.} We take expectations on both sides of the inequality \eqref{e:trineq}. The expectations of the second and third terms of the right-hand side of \eqref{e:trineq} converge to $0$ as $n\rightarrow\infty$ as proved in Section~\ref{s:proofcvgprob}; see \eqref{e:volcutapproxcut} and \eqref{e:cipCAn}. Moreover, by \eqref{e:delvvolp} and \eqref{e:covexpvarpoi},  we have that

$$ \BE \left|\frac{\delta_V(n) - \lambda(C(\eta_n;\bde_n))}{\delta_n} \right | \leq \lambda(D) \frac{\BP(P(\eta_n ; \bde_n) \not \subset D) }{\delta_n} =o\left(n^{-\left(1+\frac{1}{d}+\eps\right)}(\log n)^{-\frac{1}{d}}\right)
$$
as $n\to\infty$. Thus, \eqref{e:expasymptotpoi} follows.
 
\noindent{\em Part 2.  Variance bounds.}  The rest of the proof is devoted to establishing \eqref{e:varasymptotpoi}.  This will be done in three steps.  First,  we approximate the variance of $\delta_V(n)$ by that of $\lambda(TC(\eta_n;\bde_n))$,  the volume of tubular cuts.  This approximation is done by  assuming \eqref{e:varasymptotpoi} for $\lambda(TC(\eta_n;\bde_n))$.  Next,  we  derive variance upper bounds for the volume of tubular cuts.  Finally, we derive variance lower bounds for the volume of tubular cuts.  Thus,  we prove \eqref{e:varasymptotpoi} for $\lambda(TC(\eta_n;\bde_n))$ and complete the proof.
\medskip

\noindent{\em Step 1.  Approximating by tubular cuts.} First, we recall the standard variance-covariance inequality. Let $X,Y$ be random variables such that $\BE[X^2], \BE[Y^2] < \infty$.  Then
\begin{align}
| \BV[X] - \BV[Y] | & =  |\BV[X-Y] + 2\BC(X-Y,Y)| \nonumber \\
\label{eq:varcovbd} & \leq  \BV[X-Y]   +  2 \sqrt{ \BV[X-Y]  \BV[Y] }
\end{align}

Now, by \eqref{eq:varcovbd} and \eqref{e:delvvolp},
\begin{align}
& \left|\BV[\delta_V(n)] -  \BV[ \lambda(C(\eta_n;\bde_n))]\right|   \no \\
&\no\\
 & \leq \BV[\delta_V(n)- \lambda(C(\eta_n;\bde_n))]+  2 \sqrt{\BV[\lambda(C(\eta_n;\bde_n))]\BV[\delta_V(n)- \lambda(C(\eta_n;\bde_n))]} \no \\
 &\no\\
 &  \leq \BE[|\delta_V(n)- \lambda(C(\eta_n;\bde_n))|^2] + 2 \sqrt{\BV[ \lambda(C(\eta_n;\delta_n))]\BE[|\delta_V(n)- \lambda(C(\eta_n;\bde_n))|^2]} \no \\
 &\no\\
\label{e:vardvvolc} & \leq \lambda(D)^2 \BP(P(\eta_n ; \bde_n) \not \subset D)  +  2\lambda(D)\sqrt{\BV[ \lambda(C(\eta_n;\bde_n))] \BP(P(\eta_n ; \bde_n) \not \subset D)}.
 \end{align}

We will now complete the proof of \eqref{e:varasymptotpoi}  assuming  that there exist $C_1,C_2>0$, such that
\begin{equation}
\label{e:varasymptofnpoi}
C_1 \frac{1}{n^{(1+2/d)}(\log n)^{2(1+1/d)}} \leq \BV[\lambda(TC(\eta_n;\bde_n))] \leq C_2\frac{(\log n)^{2(1+1/d)}}{n^{(1+2/d)}}.
\end{equation}
We defer the proof  of \eqref{e:varasymptofnpoi} to Step 2 and Step 3 below. Since $\BV[X]\leq \BE[X^2]$, by \eqref{e:C-TC}, we have that
    \bes
        \BV[\lambda(C(\eta_n;\bde_n)) - \lambda(TC(\eta_n;\bde_n))]
        \leq Cn^2\de_{n}^{4d+2},
    \ees
for some $C=C_{g,D}$. Combining this with the lower bound in (\ref{e:varasymptofnpoi}) yields that  %
  \beas 
\frac{\BV[\lambda(C(\eta_n;\bde_n)) - \lambda(TC(\eta_n;\bde_n))]}{\BV[\lambda(TC(\eta_n;\bde_n))]} 
&\leq& Cn^2\de_{n}^{4d+2}\left({\frac{n\delta_n^{2(d+1)}}{(\log n)^{4(d+1)/d}}}\right)^{-1}
 \\
   &=& C\frac{(\log n)^{6+\frac{4}{d}}}{n}.
     \eeas
Thus, we have established
\begin{equation} \BV[\lambda(C(\eta_n;\bde_n)) - \lambda(TC(\eta_n;\bde_n))]  = o( \BV[\lambda(TC(\eta_n;\bde_n))]){\quad \text{as}\ n\rightarrow\infty}. \label{eq:lon}
\end{equation}
Now, \eqref{eq:lon} along with the variance-covariance bound \eqref{eq:varcovbd} implies that
\begin{equation}
\label{e:varfnvarfna}
\lim_{n \to \infty} \frac{\BV[\lambda(C(\eta_n;\bde_n))]}{\BV[\lambda(TC(\eta_n;\bde_n))]} = 1.
\end{equation}
Then, using assumption \eqref{e:covexpvarpoi} in \eqref{e:vardvvolc} and combining it with \eqref{e:varasymptofnpoi} and \eqref{e:varfnvarfna},  \eqref{e:varasymptotpoi} follows.
\medskip

\noindent{\em Step 2. Proof of upper bound for  $\BV[\lambda(TC(\eta_n;\bde_n))]$ in \eqref{e:varasymptofnpoi}.} We use the Poincar\'e inequality i.e.,  the upper bound in \eqref{e:variancebounds}.  Now, 
\begin{align*}
D_z\lambda(TC(\eta_n;\bde_n)) & =   \lambda(TC(\eta_n \cup \{z\};\bde_n) \setminus TC(\eta_n;\bde_n)) =  \lambda(TC(z;\bde_n(z)) \setminus TC(\eta_n;\bde_n)) .\nonumber \end{align*}
Thus, by  \eqref{eq:cut}, there is a constant $C>0$ (independent of $z$) such that%
\begin{align}
\label{e:1diffopbounds}  |D_z\lambda(TC(\eta_n;\bde_n))|^2  \leq \lambda(TC(z ; \bde_n(z)))^2 \leq \lambda(C(z ; \bde_n(z)))^2  \leq C\delta_n^{2(d+1)}.
\end{align}
Since $\eta_n$ has intensity measure $nf\sigma$,  we have from  \eqref{e:variancebounds} that
\begin{align*}
\BV[\lambda(TC(\eta_n;\bde_n))] & \leq C n \delta_n^{2(d+1)} \int_{bD}f(z) \md \sigma(z) \leq C n \delta_n^{2(d+1)}.
\end{align*}
The upper bound in \eqref{e:varasymptofnpoi} follows since $\delta_n=(\log n/n)^{1/d}$.
\medskip

\noindent{\em Step 3. Proof of lower bound for $\BV[\lambda(TC(\eta_n;\bde_n))]$ in \eqref{e:varasymptofnpoi}.}  We shall show that there is a $C >0 $ such that
\begin{equation}
\label{eq:varfnalbeasy}
 \BV[\lambda(TC(\eta_n;\bde_n))] \geq  C\frac{n\delta_n^{2(d+1)}}{(\log n)^{4(d+1)/d}}.
\end{equation}    
The lower bound in \eqref{e:varasymptofnpoi} will then follow since $\delta_n=(\log n/n)^{1/d}$.

We use the lower bound in \eqref{e:variancebounds}, for which we must estimate the appropritate difference operator.  Given any $s>0$, we have that 
\begin{align} 
\label{e:dzis}
D_z\lambda(TC(\eta_n;\bde_n)) &= \lambda(TC(\eta_n  \cup \{z\}; \bde_n)  \setminus TC(\eta_n ; \bde_n)) \nonumber\\
&\geq  \lambda(TC(\eta_n \cup \{z\}; \bde_n) \setminus TC(\eta_n; \bde_n))  \1[\eta_n \cap S(z,s\bde_n(z)) = \emptyset]. 
\end{align}
By Proposition \ref{P:NN2}, there exist constants $t,B,c>0$ (depending only on $g$ and $D$) such that for any $s < t$, if $\delta_n< Bs^{4/\alpha}$, $z \in bD$,  $\varphi=\{w^1,...,w^m\}\subset bD$ and $\varphi\cap S(z;s\bde_n(z)) = \emptyset$, then 
\begin{align*}
\lambda\big(TC(\phi \cup \{z\} ; \bde_n) \setminus TC(\phi;\bde_n)\big) & \geq  \lambda \left( TC(z;\bde_n(z)) \setminus \bigcup_{w \in bD \setminus S(z;s\bde_n(z))}TC(w;\bde_n(w)) \right) \\
& \geq cs^{2(d+1)} \delta_n^{d+1}.
\end{align*}
Applying this to \eqref{e:dzis}, we obtain that for any $s < t$, if $\delta_n< Bs^{4/\alpha}$, then
\be\label{e:lowdiff}
D_z\lambda(TC(\eta_n;\bde_n)) \geq cs^{2(d+1)}\delta_n^{d+1} 
 \1[\eta_n \cap S(z,s\bde_n(z)) = \emptyset]. 
\ee
 {Now,}
\begin{align*}
 \BP(\eta_n \cap S(z,s\bde_n(z)) = \emptyset)^2 &=  \exp \{-2n \sigma_f(S(z,s\bde_n(z))) \}  \\
&\geq  \exp \{-cn \sigma(S(z,s\bde_n(z))) \}\qquad  \mbox{($f$ is bounded above)} \\
&\geq   \exp \{-cn s^{d} \bde_n(z)^d \} 
\qquad \qquad \quad \mbox{(by \eqref{eq:cap})}    \\
&\geq  n^{-cg(z)^ds^{d}} 
\qquad \qquad\qquad \qquad \quad\,(\delta_n=(\log n/n)^{1/d})\\
&\geq n^{-\wt cs^d},
\end{align*}
where $\wt c=c||g||_\infty<\infty$. Thus, by \eqref{e:lowdiff}, for some $c_1>0$,
\beas
\int_{bD} \left(\BE[D_z\lambda(TC(\eta_n;\bde_n))]\right)^2 \md \sigma_f(z) 
&\geq&  c s^{4(d+1)}n\delta_n^{2(d+1)}  \int_{bD} \BP(\eta_n \cap S(z,s\bde_n(z)) = \emptyset)^2   \md \sigma_f(z)\\
&\geq& c_1s^{4(d+1)}n^{1-cs^{d}}\delta_n^{2(d+1)},
\eeas
if $\delta_n< Bs^{4/\alpha}$ and $n$ is sufficiently large. 
Thus, invoking the lower bound in \eqref{e:variancebounds}, 
\begin{equation}
\label{eq:varfnalbeasy1}
 \BV[\lambda(TC(\eta_n;\bde_n))] \geq c_1s^{4(d+1)}n^{1-cs^{d}}\delta_n^{2(d+1)},
\end{equation}
if $\delta_n< Bs^{4/\alpha}$ and $n$ is sufficiently large. Since $\de_n=(\frac{\log n}{n})^{\frac{1}{d}}$, \eqref{eq:varfnalbeasy1} holds when $s= \left( \frac{4(d+1)}{cd\log(n)}\right)^{\frac{1}{d}}$ and $n$ is sufficiently large, and yields \eqref{eq:varfnalbeasy}. Though we have not used the fact, it can be verified that this choice of $s$ is the optimal one.   \\

Now, assume that $\beta_n$ is a binomial point process satisfying (the analogue of) \eqref{e:covexpvarpoi}.  The proof of $L^1$-convergence, i.e., \eqref{e:expasymptotpoi} for $\beta_n$ follows as in Step 1 above. The upper bound on the variance in \eqref{e:varasymptotpoi} can be derived as done in Step 2 above with the following modification. Instead of the Poincar\'e inequality, apply the Efron--Stein inequality \eqref{e:ESsimple} to $F(\cdot)=\lambda(C(\cdot,\bde_n))$, and use the fact that  \eqref{e:1diffopbounds} holds for $\lambda(C(\beta_n,\bde_n))$ as well. 
\end{proof}

\subsection{Theorem \ref{t:cltpoi}: central limit theorem}
\label{s:proofclt}

\begin{proof}[Proof of  Theorem \ref{t:cltpoi}] The strategy of the proof is to reduce \eqref{e:normapproxdvpoi} to a central limit theorem for $\lambda(C(\varphi_n;\delta_n))$, where $\varphi_n$ is either $\eta_n$ or $\beta_n$. The latter is then handled using second-order Poincar\'e inequalities (Theorem \ref{t:normapproxpoi} in the case of $\eta_n$, and Theorem \ref{t:normapproxbin} in the case of $\beta_n$).
\smallskip

\noindent{\em Case I. The source set is $\eta_n$.} Let $\eta_n$ be a Poisson process on $bD$ satisfying \eqref{e:covexpvarpoi}. Recall that $\delta_V(n)=\delta_V(P(\eta_n ; \bde_n),D)$, and the Wasserstein distance is as in \eqref{e:defndW}. We will often invoke the following simple observation.  
\be\label{e:wass}
	d_W(X,Y)\leq \sup_{h \in \mathrm{Lip}(1)} \BE|h(X)-h(Y)|\leq \BE|X-Y|.
\ee
By the triangle inequality,  we have that
\bes      \label{e:delvceta}
d_W\Big(\frac{\delta_V(n) - \BE[\delta_V(n)]}{\sqrt{\BV[\delta_V(n)]}}, Z \Big)\leq \text{I}+\text{II}+\text{III},
\ees
where
\beas
\text{I}&:=&d_W\left(\frac{\delta_V(n) - \BE[\delta_V(n)]}{\sqrt{\BV[\delta_V(n)]}}, \frac{\lambda(C(\eta_n ; \bde_n)) - \BE[\lambda(C(\eta_n ; 
\bde_n))]}{\sqrt{\BV[\delta_V(n)]}}\right),\\\text{II}&:=& d_W\left(\frac{\lambda(C(\eta_n ; \bde_n)) - \BE[\lambda(C(\eta_n ; \bde_n))]}{\sqrt{{\BV[\delta_V(n)]}}} , \frac{\lambda(C(\eta_n ; \bde_n)) - \BE[\lambda(C(\eta_n ; \bde_n))]}{\sqrt{\BV[\lambda(C(\eta_n ; \bde_n))]}}\right),\\
\text{III} &:=& d_W\left(\frac{\lambda(C(\eta_n ; \bde_n)) - \BE[\lambda(C(\eta_n ; \bde_n))]}{\sqrt{\BV[\lambda(C(\eta_n ; \bde_n))]}},Z\right).
\eeas

\noindent {\em Part 1. Upper bound for I}. Using\eqref{e:wass}, we have that
\beas 
\text{I}&\leq & \BE\left|
	\frac{\delta_V(n) - \BE[\delta_V(n)]}{\sqrt{\BV[\delta_V(n)]}}
	-\frac{\lambda(C(\eta_n ; \bde_n))-\BE[\lambda(C(\eta_n ; \bde_n))]}
		{\sqrt{\BV[\delta_V(n)]}}\right|\\
		&\leq& 2\lambda(D) \frac{\BP(P(\eta_n ; \bde_n) \not \subset D)}
			{\sqrt{\BV[\delta_V(n)]}}
 \qquad\qquad\qquad\mbox{(by \eqref{e:delvvolp} and \eqref{E:upperlowerbounds})}\\
&\leq& C\frac{n^{-(1+2/d +\epsilon)}}{\sqrt{n^{-(1+2/d)}(\log n)^{-2(1+1/d)}}}
\qquad\qquad\,\, \mbox{(by \eqref{e:covexpvarpoi} and \eqref{e:varasymptotpoi})}\\ 
&=&  C \frac{(\log n)^{(1+1/d)}}{n^{\epsilon + \frac{1}{2} + \frac{1}{d}}}.
\label{eq:bI}
\eeas
\medskip

\noindent {\em Part 2. Upper bound for II}. Denoting $\dfrac{\lambda(C(\eta_n ; \bde_n)) - \BE[\lambda(C(\eta_n ; \bde_n))]}{\sqrt{\BV[\lambda(C(\eta_n ; \bde_n))]}}$ by $X_n$, we have that 
\bea
\text{II}&=&d_W\left(\sqrt{\frac{\BV[\lambda(C(\eta_n ; \bde_n))]}{\BV[\delta_V(n)]}}\,X_n,X_n\right) \notag \\
&\leq & \left |\sqrt{{\frac{\BV[\lambda(C(\eta_n ; \bde_n))]}{\BV[\delta_V(n)]}}} - 1 \right | \BE|X_n|\notag \\
&=& \left | {\frac{\BV[\lambda(C(\eta_n ; \bde_n))]}{\BV[\delta_V(n)]}} - 1 \right | \frac{1}{\left| \sqrt{{\frac{\BV[\lambda(C(\eta_n ; \bde_n))]}{\BV[\delta_V(n)]}}} + 1 \right |}\BE|X_n|\notag \\
&\leq & \left | {\frac{\BV[\lambda(C(\eta_n ; \bde_n))]}{\BV[\delta_V(n)]}} - 1 \right |,\label{eq:inII}
\eea
where, in the last step, we have used the Cauchy--Schwartz inequality to observe that,
\bes
\BE|X_n|=
\BE \left |\frac{\lambda(C(\eta_n ; \bde_n)) - \BE[\lambda(C(\eta_n ; \bde_n))]}{\sqrt{\BV[\lambda(C(\eta_n ; \bde_n))]}} \right | 
\leq \sqrt{ \BE \left[  \frac{\lambda(C(\eta_n ; \bde_n)) - \BE[\lambda(C(\eta_n ; \bde_n))]}{\sqrt{\BV[\lambda(C(\eta_n ; \bde_n))]}} \right ]^2} = 1.
\ees
Using the variance-covariance bound \eqref{eq:varcovbd} as done in \eqref{e:vardvvolc}, we have that     
\bea
 &&\left | {\frac{\BV[\lambda(C(\eta_n ; \bde_n))]}{\BV[\delta_V(n)]}} - 1\right | = \frac{|  \BV[ \lambda(C(\eta_n;\bde_n))] - \BV[\delta_V(n)]|}{\BV[\delta_V(n)]} \notag \\
  &&\qquad\quad\leq  \frac{\lambda(D)^2\BP(P(\eta_n ; \bde_n) \not \subset D)  +  2\lambda(D)\sqrt{\BV[ \delta_V(n)]\BP(P(\eta_n ; \bde_n) \not \subset D)} }{\BV[ \delta_V(n)]}. \label{e:varC/vardV}
  \eea
Combining this with \eqref{eq:inII}, and invoking the bounds in \eqref{e:covexpvarpoi} and \eqref{e:varasymptotpoi}, we have that 
 
 \beas
  \text{II} &\leq & C\left({n^{-(1+2/d +\epsilon)}}{n^{(1+\frac{2}{d})} (\log n)^{2(1+ \frac{1}{d})}} +  \sqrt{{n^{-(1+2/d +\epsilon)}}{n^{(1+\frac{2}{d})} (\log n)^{2(1+ \frac{1}{d})}}} \,\right)\\
  &\leq & C\, \frac{(\log n)^{(1+ \frac{1}{d})}}{n^{\frac{\epsilon}{2}}}.
  \eeas
\medskip

\noindent {\em Part 3. Upper bound for III}. We use Theorem \ref{t:normapproxpoi}, for which must derive bounds for $T_i(\lambda(C(\eta_n ; \bde_n)))$, $i=1,2,3$, as defined in Theorem \ref{t:normapproxpoi}.  Note that in the case of $\eta_n$, the intensity measure is $\mu= n\sigma_f = nf \sigma$.  %

To compute $T_1(\lambda(C(\eta_n ; \bde_n)))$ and $T_2(\lambda(C(\eta_n ; \bde_n)))$,  we need to bound the second-difference operator of $\lambda(C(\eta_n ; \bde_n))$. Note that for $z,w \in bD$,
\bea
\quad &&\Big| D^2_{z,w}(\lambda(C(\eta_n ; \bde_n))) \Big|\notag\\  
\quad &&=   \Big| \lambda(C(\eta_n \cup \{z,w\} ; \bde_n)) -  \lambda(C(\eta_n \cup \{w\} ; \bde_n))-  \lambda(C(\eta_n \cup \{z\} ; \bde_n))  +  \lambda(C(\eta_n ; \bde_n)) \Big| \notag \\
\quad	&&\leq   \lambda(C(z;\bde_n(z)) \cap C(w;\bde_n(w))) 
	\qquad\qquad\qquad\,  \mbox{(by the inclusion-exclusion principle)} \notag \\
\quad	&&\leq C\delta_n^{d+1}\1[C(z;\bde_n(z)) \cap C(w;\bde_n(w)) \neq \emptyset] 
	\qquad  \mbox{(by \eqref{eq:cut})} \notag \\
\quad &&\leq C \delta_n^{d+1}\1[C(z;\|g\|_{\infty}\delta_n) \cap C(w;\|g\|_{\infty}\delta_n) \neq \emptyset]. \label{e:2diffopbounds} 
\eea

Thus, 
\begin{align*}
A(z^1,z^2,z^3):=\, & \BE\left[\big(D^2_{z^1,z^2}\lambda(C(\eta_n ; \bde_n))\big)^2\big( D^2_{z^2,z^3}\lambda(C(\eta_n ; \bde_n))\big)^2\right]  \\
& \leq \delta_n^{4(d+1)}\1\left[z^1\in N(z^2;||g||_\infty\de_n)\right]\1\left[z^3\in N(z^2;||g||_\infty\de_n)\right].  
\end{align*} 
where
	$N(z;\delta):= \{w : C(w;\delta) \cap C(z;\delta) \neq \emptyset\}.$ Also, from \eqref{e:1diffopbounds}, we have that
\bes
B(z^1,z^2):=\BE[(D^1_{z^1}\lambda(C(\eta_n ; \bde_n)))^2(D^1_{z^2}\lambda(C(\eta_n ; \bde_n)))^2] \leq  C \delta_n^{4(d+1)}.
\ees
Therefore, by Proposition~\ref{P:NN1} and the fact that $||g||_\infty<\infty$,
\beas
 T_1(\lambda(C(\eta_n ; \bde_n)))^2 &=&
\int_{(bD)^3}A(z^1,z^2,z^3)^{1/2}B(z^1,z^2)^{1/2}d\mu^3(z^1,z^2,z^3)
\\
&\leq& Cn^3  \delta_n^{4(d+1)} 
\int_{bD} \sigma\left(N(z;||g||_\infty\de_n)\right)^2 \md \sigma_f(z) \\
&\leq&  Cn^3 \delta_n^{6d+4}.
\eeas
In a similar fashion, we have that
\beas
T_2(\lambda(C(\eta_n ; \bde_n)))^2 &=&
\int_{(bD)^3}A(z^1,z^2,z^3)d\mu^3(z^1,z^2,z^3)
\\ &\leq& Cn^3 \delta_n^{4(d+1)} \int \sigma(N(z;\|g\|_{\infty}\delta_n))^2 \md \sigma_f(z) \leq  Cn^3  \delta_n^{6d+4}. 
\eeas
Finally, an upper bound for $T_3(\lambda(C(\eta_n ; \bde_n)))$ follows easily from \eqref{e:1diffopbounds}:
\begin{equation}
\label{e:t3fn}
T_3(\lambda(C(\eta_n ; \bde_n))) \leq Cn \delta_n^{3(d+1)}.
\end{equation}
Now, using that $\de_n=((\log n)/n)^{1/d}$, and noting that the lower bound in \eqref{e:varasymptofnpoi} holds for $\BV[\lambda(C(\eta_n ; \bde_n))]$ because of \eqref{e:varfnvarfna},  we have that

\bea
\label{e:t1fnfinal}
\quad \frac{T_1(\lambda(C(\eta_n ; \bde_n)))}{{{\BV[\lambda(C(\eta_n ; \bde_n))]}}},
\frac{T_2(\lambda(C(\eta_n ; \bde_n)))}{{ {\BV[\lambda(C(\eta_n ; \bde_n))]}}} 
 \leq C \frac{\left(\frac{ (\log n)^{6 + \frac{4}{d}}}{n^{3+ \frac{4}{d}}}\right)^\frac{1}{2}}{n^{-1-\frac{2}{d}}{{(\log n)^{-2-\frac{2}{d}}}}} &\leq&  C \frac{(\log n)^{5+ \frac{4}{d}}}{n^{1/2}},\\
\label{e:t3fnfinal}
\frac{T_3(\lambda(C(\eta_n ; \bde_n)))}{{ {\BV[\lambda(C(\eta_n ; \bde_n))]}}^{3/2}} 
\leq C \frac{\frac{(\log n)^{3+ \frac{3}{d}}}{n^{2+\frac{3}{d}}}}{\left(n^{-1-\frac{2}{d}}{{(\log n)^{-2-\frac{2}{d}}}}\right)^\frac{3}{2}} &\leq& C\frac{(\log n)^{6+6/d}}{n^{1/2}}.
\eea
We now apply \eqref{e:2ndPoincare} from Theorem \ref{t:normapproxpoi} to $F(\cdot)=\lambda(C(\cdot;\bde_n))-\BE(\lambda(C(\cdot;\bde_n)))$, and substitute the bounds \eqref{e:t1fnfinal} and \eqref{e:t3fnfinal} to obtain that
\beas
 \text{III}
&\leq& C \left[{{\frac{(\log n)^{5+\frac{4}{d}}}{n^{1/2}}}} + {{ \frac{(\log n)^{5+\frac{4}{d}}}{n^{1/2}}}} +{{\frac{(\log n)^{6+\frac{6}{d}}}{n^{1/2}}}} \right] \\
 &\leq& C\, \frac{(\log n)^{6+ 6/d}}{n^{1/2}}.
\eeas
Finally, combining the bounds for I, II, and III, we have that
\begin{align} 
d_W\left(\frac{\delta_V(n)-\BE[\delta_V(n)]}{\sqrt{\BV[\delta_V(n)]}}, Z\right) &\leq  C \left(\frac{(\log n)^{(1+1/d)}}{n^{\epsilon + \frac{1}{2} + \frac{1}{d}}}+\frac{(\log n)^{(1+1/d)}}{n^{\frac{\epsilon}{2}}} + \frac{(\log n)^{6+6/d}}{n^{1/2}}\right).
\end{align}
Thus, \eqref{e:normapproxdvpoi} holds for the Poisson process.
\medskip

\noindent{\em Case II. The source set is $\beta_n$.} Let $\beta_n$ be a binomial process on $bD$ satisfying \eqref{e:covexpvarpoi}, and let $\delta_V(n)=\delta_V(P(\beta_n ; \bde_n),D)$. Once again, by the triangle inequality, we may write
\bes
d_W\left(\frac{\delta_V(n) - \BE[\delta_V(n)]}{\sqrt{\BV[\delta_V(n)]}}, Z\right)
\leq \text{I}_B+\text{II}_B+\text{III}_B,
\ees
where
\beas
  \text{I}_B &:=&d_W\left(\frac{\delta_V(n) - \BE[\delta_V(n)]}{\sqrt{\BV[\delta_V(n)]}}, \frac{\lambda(C(\beta_n ; \bde_n)) - \BE[\lambda(C(\beta_n ; \bde_n))]}{\sqrt{\BV[\delta_V(n)]}}\right),\\
   \text{II}_B &:=& d_W\left(\left(\frac{\lambda(C(\beta_n ; \bde_n)) - \BE[\lambda(C(\beta_n ; \bde_n))]}{\sqrt{{\BV[\delta_V(n)]}}} \right), \frac{\lambda(C(\beta_n ; \bde_n)) - \BE[\lambda(C(\beta_n ; \bde_n))]}{\sqrt{\BV[\lambda(C(\beta_n ; \bde_n))]}}\right),\\
 \text{III}_B &:=&d_W\left(\frac{\lambda(C(\beta_n ; \bde_n)) - \BE[\lambda(C(\beta_n ; \bde_n))]}{\sqrt{\BV[\lambda(C(\beta_n ; \bde_n))]}},Z\right). 
\eeas
Owing to the assumed lower bound for $\BV[\delta_V(n)]$, the bound for I$_B$ is the same as that for I, namely, for some $C>0$,
\beas 
\text{I}_B &\leq & C \frac{(\log n)^{(1+1/d)}}{n^{\epsilon + \frac{1}{2} + \frac{1}{d}}}.
\eeas
{Next, since \eqref{e:varC/vardV} holds for $\beta_n$ as well, the assumed lower bound for $\BV[\delta_V(n)]$ yields the same bound for $\text{II}_B$ as that for $\text{II}$, i.e., for some $C>0$,
\bes
\text{II}_B \leq C \frac{(\log n)^{(1+1/d)} }{n^{	\frac{\epsilon}{2}}}.
\ees}

To obtain a bound on $\text{III}_B$, we use Theorem \ref{t:normapproxbin}. First, note that%
 \be\label{e:boundD1}
	|D_j\lambda(C(\bfZ;\bde_n))|  \leq \lambda(C(W^j;\bde_n(W^j))) \leq C \delta_n^{d+1},\quad {Z\in Rec(\bfW)},\ j=1,2.
\ee
From the inclusion-exclusion principle, we have that 
\be\label{e:D12}
D_{1,2}(\lambda(C(\bfW;\bde_n))) \neq 0\ \text{implies that}\ C(W^1;\bde_n(W^1)) \cap C(W^2;\bde_n(W^2)) \neq \emptyset. 
\ee
Thus, since $\bfY\overset{d}{=}\bfW$ and \eqref{e:boundD1} holds, we have that
\bea
T'_1\Big(\lambda(C(\bfW;\bde_n))\Big) 
&\leq& C \delta_n^{4(d+1)} \BP\Big(C(W^1;\|g\|_{\infty}\delta_n) \cap C(W^2;\|g\|_{\infty}\delta_n) \neq \emptyset \Big)\notag \\
&\leq&   C \delta_n^{4(d+1)}\BE[\sigma_f(N(W^1;\|g\|_{\infty}\delta_n))] \leq C \delta_n^{5d+4}, \label{e:t1bin}
\eea
where, in the final step, we have used Proposition~\ref{P:NN1} to estimate $\sigma\left(N(W^1;\|g\|_{\infty}\delta_n)\right)$. Invoking \eqref{e:D12} again, we have the following upper bound for $T'_2\Big(\lambda(C(\bfW;\bde_n))\Big)$:
\bea
&&C\delta_n^{4(d+1)}\sup_{\bfY,\bfZ \in Rec(\bfW)}\!\!\BP\Big(C(Y^1;\bde_n(Y^1))\cap C(Y^2;\bde_n(Y^2)) \neq \emptyset,\notag\\ 
&&\qquad \qquad\qquad\qquad\qquad C(Z^1;\bde_n(Z^1))\cap C(Z^3;\bde_n(Z^3)) \neq \emptyset \Big).\label{e:bd1T2}
\eea
To further bound \eqref{e:bd1T2}, note that by the definition of $\bfY$ and $\bfZ$, $Y^1,Y^2, Z^3$ are independent and, either $Z^1$ is independent of all of them, or $Z^1 = Y^1$. Thus,
\bea
&&  \qquad  T'_2\Big(\lambda(C(\bfW;\bde_n))\Big)\notag\\
&&   \qquad \leq C \delta_n^{4(d+1)}\max \left\lbrace \BP\Big(C(W^1;\bde_n(W^1)) \cap C(W^2;\bde_n(W^2)) \neq \emptyset \Big)^2,   \right.  \notag\\ 
&& \qquad  \left.  \BP\Big(C(W^1;\bde_n(W^1)) \cap C(W^2;\bde_n(W^2)) \neq \emptyset,   C(W^1;\bde_n(W^1)) \cap C(W^3;\bde_n(W^3)) \neq \emptyset \Big)  \right\rbrace \notag\\
&&  \qquad  \leq  C \delta_n^{4(d+1)} \max \left\{{\BE\left[\sigma_f(N(W^1;\|g\|_{\infty}\delta_n))\right]^2},  \BE[\sigma_f(N(W^1;\|g\|_{\infty}\delta_n))^2] \right\} \notag\\
&& \qquad  \leq C \delta_n^{6d+4}. \label{e:t2bin}
\eea
Next, it is immediate from \eqref{e:boundD1} that
 \begin{equation}
 \label{e:t3t4bin}
 T'_3\Big(\lambda(C(\bfW;\bde_n))\Big)  \leq C \delta_n^{4(d+1)}  \quad \text{and}\quad  T'_4\Big(\lambda(C(\bfW;\bde_n))\Big)  \leq C \delta_n^{3(d+1)}. 
 \end{equation}
{Lastly, since \eqref{e:varC/vardV} holds for $\beta_n$ as well, the assumed lower bound for $\BV[\delta_V(P(\beta_n; \bde_n),D)]$ is also a lower bound for $\BV[\lambda(C(\beta_n; \bde_n),D)]$.} Now, we apply \eqref{e:2ndPBinom} in Theorem~\ref{t:normapproxbin} to $F(\cdot)=\lambda(C(\cdot;\bde_n))-\BE(\lambda(C(\cdot;\bde_n)))$, and substitute the bounds \eqref{e:t1bin}, \eqref{e:t2bin} and  \eqref{e:t3t4bin}, and the lower bound for {$\BV[\lambda(C(\beta_n; \bde_n),D)]$} to obtain that
\begin{align*}
\text{III}_B
& 
\leq C\,\frac{\sqrt{n^2 \delta_n^{5d+4}} + \sqrt{n^\frac{3}{2}\delta_n^{6d+4} } + \sqrt{n\delta_n^{4(d+1)}}}{ n^{-(1+2/d)}{ {(\log n)^{-2(1+1/d)}}}} 
+ \frac{Cn \delta_n^{3(d+1)}}{\left(n^{-(1+2/d)}{ {(\log n)^{-2(1+1/d)}}}\right)^\frac{3}{2}}\\
&\leq C\frac{(\log n)^{6 + \frac{6}{d}}}{ \sqrt{n}}.
\end{align*}

Finally, combining the bounds for I$_B$, II$_B$, and III$_B$, we have that
$$  d_W\left(\frac{\delta_V(n) - \BE[\delta_V(n)]}{\sqrt{\BV[\delta_V(n)]}}, Z \right) \leq C \left({{ \frac{(\log n)^{(1+1/d)}}{n^{\epsilon + \frac{1}{2} + \frac{1}{d}}} + \frac{(\log n)^{(1+1/d)}}{n^{\frac{\epsilon}{2}}}+    \frac{(\log n)^{6 + 6/d}}{n^{1/2}}}}\right).$$
Thus, \eqref{e:normapproxdvpoi} holds for $\beta_n$ as well, and the proof is complete.   
\end{proof}

\subsection{Proposition~\ref{lem:rdelta}: conditions for coverage}
\label{s:proofcov}
\begin{proof}[Proof of Proposition \ref{lem:rdelta}] The strategy of the proof is to use Lemma~\ref{L:inc=cov} to convert containment probabilities into converage probabilities. We handle the latter by dividing $bD$ suitably and using a classical balls into bins approach for coverage problems. 
\medskip

 \noindent {\em Proof of $(a)$.} Let $\eta_n$ be a Poisson process on $bD$. We claim that it suffices to show the existence of a $c_D>0$ such that
\begin{equation}
\label{e:noncovcondn}
\lim_{n \to \infty} \BP\Big( S(z;\eps) \setminus S(\eta_n; c_D\delta_n) \neq \emptyset \Big) = 1,\qquad \text{{for some}}\  z \in bD, \eps>0.
\end{equation}
To see this, assume \eqref{e:noncovcondn} and suppose $\min_{z \in bD}g(z) < c_D$. Then, there exist $z_0\in bD$ and $\eps_0>0$ such that $\max_{w \in S(z_0;2\eps_0)}g(w) < c_D$. Moreover, since $||g||_\infty<\infty$ and $\de_n\rightarrow 0$ as $n\rightarrow\infty$, there is an $n_{g,\eps}$ large enough so that if $n\geq n_{g,\eps}$, then $S(z_0;\eps_0)\cap S(w;\bde_n(w))= \emptyset$ whenever $w\notin S(z_0;2\eps_0)$. Now, suppose $S(z_0;\eps_0) \setminus S(\eta_n; c_D\delta_n) \neq \emptyset$ for some $n\geq n_{g,\eps}$. By the monotonicity of cuts, $S(\eta_n \cap S(z_0;2\eps_0) ; \bde_n) \subset S(\eta_n; c_D \delta_n)$. Thus, 
	\bes
		S(z_0;\eps_0) \setminus S(\eta_n; \bde_n)= 
			S(z_0;\eps_0) \setminus S(\eta_n \cap S(z_0;2\eps_0) ; \bde_n) \neq \emptyset.
	\ees
In other words, for any $n\geq n_{g,\eps}$, $S(z_0;\eps_0) \setminus S(\eta_n; c_D\delta_n) \neq \emptyset$ implies that $S(z_0;\eps_0) \setminus S(\eta_n ; \bde_n) \neq \emptyset$, which (by Lemma \ref{L:inc=cov}) further implies that $P(\eta_n ; \bde_n) \not \subset D$. Thus, \eqref{e:noncovcondn} implies $(a)$.

We now establish \eqref{e:noncovcondn}. Let $m\in\mathbb N$ and $\{p^1,\ldots,p^m\}\subset bD$ be as in Lemma \ref{L:Vitali} corresponding to $\delta = \delta_n$.  Note that both $m$ and the $p^j$'s change with $n$, but we do not indicate this in the notation for convenience. Now, given $z \in bD$ and $\eps>0$,  set
$$ \ell_n :=  \ell_{\delta_n,\eps,z} = | \{j:  p^j \in S(z;\epsilon) \} |.$$
Without loss of generality,  we assume that $\{1,...,\ell_n\}=\{j:p^j\in S(z;\eps)\}$. We claim that, in order to prove \eqref{e:noncovcondn}, it suffices to show the existence of a $b=b_D>0$ such that for all $z \in bD$ and $\eps>0$, with high probability there exists a $j\in\{1,...\ell_n\}$ such that $\eta_n\cap S(p^j;b\mathfrak K \de_n)=\emptyset$, i.e., 
\begin{equation}
\label{e:ncovoneempty}
\lim_{n \to \infty}\BP\left(\bigcup_{j=1}^{\ell_n} \left\{\eta_n \cap S(p^j ; b\mathfrak K\delta_n) = \emptyset \right\}\right)  = 1.
\end{equation}
To see this, set $z=z_0$ and $\eps=\eps_0$. Now, suppose $\eta_n\cap S(p^j;b\mathfrak K\de_n)=\emptyset$ for some $j\in\{1,...,\ell_n\}$, where $\ell_n=	\ell_{\de_n,\eps_0,z_0}$. Then, by the quasi-symmetry of the quasimetric $\mathtt d$,  there is a $c_D>0$ depending only on $\mathfrak K$ and $b$ such that $\min\{\mathtt d(p^j,W):W\in\eta_n\}>c_D\delta_n$, i.e., $p^j\notin S(\eta_n;c_D \delta_n)$. Thus, $S(z_0;\eps_0)\setminus S(\eta_n;c_D \delta_n) \neq \emptyset$. In other words, \eqref{e:ncovoneempty} implies \eqref{e:noncovcondn}. 

Finally, it remains to prove \eqref{e:ncovoneempty}. Let $z\in bD$ and $\eps>0$. Let $b\in (0,1)$. Using that the caps $S(p^j ; c\mathfrak K\delta_n)$, $j=1,...,\ell_n$, are pairwise disjoint, the Poisson process is completely independent, $||f||_\infty<\infty$, and \eqref{eq:cap},  we have that for large enough $n$, 
\bea
\BP\left(\bigcap_{j=1}^{\ell_n} \left\{ \eta_n \cap S(p^j ; b\mathfrak K\delta_n) \neq \emptyset\right\}\right) 
&=& \prod_{j=1}^{\ell_n} \BP(\eta_n \cap S(p^j ; b\mathfrak K\delta_n) \neq \emptyset ) \notag \\
& =& \prod_{j=1}^{\ell_n} \left(1 - e^{-n\sigma_f(S(p^j ; b\mathfrak K\delta_n))}\right)  \notag\\
& \leq& \prod_{j=1}^{\ell_n} \left(1 - e^{-nC_{f,D}(b\mathfrak K\delta_n)^d}\right)
\notag\\
&=& \left(1 - n^{-C_{f,D}(b\mathfrak K)^d}\right)^{\ell_n}.  \label{e:probnonemptycell}
\eea
 We assume that $n$ is large enough so that \eqref{e:probnonemptycell} holds and $S(w;\delta_n) \cap S(z;\eps /2)=\emptyset$ whenever $w\notin S(z;\eps )$. Recalling that $bD\subset\cup_{j=1}^mS(p^j;\delta_n)$ and $S(p^j;\delta_n)\cap S(z;\eps)=\emptyset$ for $j=\ell_n+1,...,m$, we thus obtain that $S(z;\eps /2) \subset \cup_{j=1}^{\ell_n} S(p^j;\delta_n)$. Thus, by \eqref{eq:cap}, there is a $C_D>0$ such that $\sigma(S(z;\eps /2)) \leq \ell_n\left(C_D\delta_n^d\right)$, i.e., 
	\bes
		\ell_n\geq C_D^{-1}\de_n^{-d}\sigma(S(z;\eps /2))\geq C_\eps\de_n^{-d}
		=C_\eps\left(\frac{n}{\log n}\right),
	\ees
where $C_{\eps } := C_D^{-1}\min\{\sigma(S(z;\eps /2)):{z \in bD}\}$. Thus, for sufficiently large $n=n_\eps$, we have that $\ell_n\geq n^\beta$ for some $\beta<1$. Now, choosing $b<1$ so that $C_{f,D}b^d\mathfrak K^d < \beta$,  we have that the right-hand side of \eqref{e:probnonemptycell} is bounded above by $(1 - n^{-C_{f,D}b^d\mathfrak K^d})^{n^{\beta}}$, which converges to $0$ as $n\rightarrow\infty$. This proves \eqref{e:ncovoneempty} and, therefore, (a) for $\eta_n$.
\medskip

\noindent {\em Proof of $(b)$.} Let $\eta_n$ be a Poisson process on $bD$. Let $\{p^1,...,p^m\}$ as in the proof of $(a)$ above. Recall that $m$ depends on $n$, so we denote it by $m_n$ in this proof. We claim that it suffices to prove that for any $\alpha>0$, there is a $b\geq 1$ so that
\begin{equation}
\label{e:covnoempty}
\BP\left( \bigcup_{j=1}^{m_n} \{ \eta_n \cap S(p^j ; b\delta_n) = \emptyset \}\right) =\mathcal O(n^{-\alpha})\ \text{as}\ n\rightarrow\infty.
\end{equation}
To see this, suppose  for some $b\geq 1$, $\eta_n \cap S(p^j ; b\delta_n) \neq \emptyset$ for all $j=1,...,m_n$. Then, by the quasi-symmetry and the quasi-triangle inequality satisfied by $\mathtt d$,  we obtain that for for some $C_D>0$, $S(p^j; \delta_n) \subset S(\eta_n ; C_D \delta_n)$ for all $j=1,...,m_n$. Now, if $\min_{z \in bD}g(z) \geq C_D$, then for all $j$,  $S(p^j;\delta_n) \subset S(\eta_n ;  \bde_n)$.  Thus,  from Lemma \ref{L:Vitali}(a) and \ref{L:inc=cov},   we have that $P(\eta_n ; \bde_n) \subset D$. Thus, \eqref{e:covnoempty} implies $(b)$. 

It now remains to prove \eqref{e:covnoempty}.
So,  using the union bound and similar arguments as in the proof of $(a)$ above,  we have that
\beas
\BP\left(\bigcup_{j=1}^{m_n} \{ \eta_n \cap S(p^j ; b\delta_n) = \emptyset\}\right)
& \leq& \sum_{j=1}^{m_n} \BP\left(\eta_n \cap S(p^j ; b\delta_n) = \emptyset\right)\\
&=& \sum_{j=1}^{m_n} e^{-n\sigma_f(S(p^j ; b\delta_n))} \leq m_n\left( n^{-C_{f,D}b^d}\right).
\eeas
Now, since $S(p^j ; \mathfrak K\delta_n)$'s are mutually disjoint, we have from \eqref{eq:cap} a $C>0$ such that $m_nC\mathfrak K^d\delta_n^d \leq \sum_{j=1}^{m_n}\sigma\left(S(p^j ; \mathfrak K\delta_n)\right)\leq \sigma(bD)$. Modifying $C$ as needed, we get  $m_n \leq C\delta_n^{-d} = Cn/(\log n) \leq Cn$. Thus, 
\bes
 \BP\left(\bigcup_{j=1}^{m_n} \{ \eta_n \cap S(p^j ; b\delta_n)=\emptyset\}\right) \leq C n^{1 - C_{f,D}b^d}.
\ees
Now, \eqref{e:covnoempty} follows by choosing $b$ suitably large, thus completing the proof of $(b)$ for $\eta_n$. 
\medskip

\noindent {\em Proof of $(c)$.} The proofs of $(a)$ and $(b)$ for a binomial process $\beta_n$ are the same as those for $\eta_n$ --- except in the computation of the probabilities of $\cap_{j=1}^{\ell_n} \{ \beta_n \cap S(p^j ; c\delta_n) \neq \emptyset \}$ and $\cup_{j=1}^{m_n} \{ \beta_n \cap S(p^j; c\delta_n)= \emptyset \}$.  

For the proof of $(a)$ for $\beta_n$, we couple $\beta_n$ with a suitable Poisson process to bound $\BP(\cap_{j=1}^{\ell_n} \{ \beta_n \cap S(p^j ; c\delta_n) \neq \emptyset \} )$.  Let $N'_n$ be a Poisson random variable with mean $(1+a)n$ for some $a > 0$. Let $\eta'_n := \{W^1,\ldots,W^{N'_n}\}$ be a Poisson process with intensity measure $(1+a)nf$.  Recall that $W^i$'s are i.i.d.  with density $f$ and $\beta_n = \{W^1,\ldots,W^n\}$. Observe that if $N'_n \geq n$, then $\beta_n \subset \eta'_n$. Therefore,
\begin{align*}
\BP\left(\bigcap_{j=1}^{\ell_n} \{ \beta_n \cap S(p^j ; c\delta_n) \neq \emptyset \} \right) 
& \leq \BP\left(\bigcap_{j=1}^{\ell_n} \{ \beta_n \cap S(p^j ; c\delta_n) \neq \emptyset \} \cap \{N'_n \geq n\} \right) + \BP(N'_n < n) \\
& \leq  \BP\left(\bigcap_{j=1}^{\ell_n} \{ \eta'_n \cap S(p^j ; c\delta_n) \neq \emptyset \}\right) + \BP(|N'_n - (1+a)n| \geq an) .
\end{align*}
On the right-hand side above,  the first term converges to $0$ as $n\rightarrow\infty$ due to the bound in \eqref{e:probnonemptycell} and the subsequent argument. The second term  converges to $0$ as $n\rightarrow\infty$ by Chebyshev's inequality.  Thus, we obtain the analogue of \eqref{e:ncovoneempty} for $\beta_n$. The rest of the proof of $(a)$ for $\beta_n$ can be repeated verbatim. 

To prove $(b)$ for $\beta_n$, we use the obvious inequality $1-x \leq e^{-x}$, $x>0$, to obtain that 
$$ \BP(\beta_n \cap S(p^j ; b\delta_n) = \emptyset) = \left(1 - \sigma_f(S(p^j ; b\delta_n))\right)^n  \leq e^{ -n \sigma_f(S(p^j; b\delta_n))}$$
for all $j=1,...,m_n$. The rest of the proof remains the same as in the case of $\eta_n$. This concludes the proof of $(c)$.
\end{proof}
\appendix
\section{Background on Poisson and Binomial processes}
\label{s:probtool}
We now collect some notions and results on Poisson and binomial processes that are used in the article; see \cite{last2017lectures,peccati2016stochastic} for more details.

{Let $D \subset \Cd$ be} as defined above, and ${\mathcal B} := \mathcal{B}(bD)$ be the Borel $\sigma$-field on
$bD$.   Let $\bN$ be the set of finite counting measures on $bD$,
and ${\mathcal N}$ be the evaluation $\sigma$-field i.e.,  the $\sigma$-field of subsets of $\bN$ generated by
the sets of the form $\{ \varphi \in \bN : \varphi(B)=k\}$, with $B \in {\mathcal
  B}$, $k \in \N \cup \{0\}$.
\begin{definition}
Let $(\Omega, {\mathcal F}, {\mathbb P})$ be a probability space. A {\em point
process} $\eta$ on $bD$ is a random element of $(\bN, {\mathcal N})$,
i.e. $\eta: \Omega \rightarrow \bN$ is measurable. The intensity measure of a point process $\eta$ is the measure $\lambda$ defined as $\lambda(B) = \E[\eta(B)],  B \in {\mathcal B}$.
\end{definition}
Let $\lambda$ be a finite measure. A point process $\eta$ is said to be {\em a Poisson process on $bD$ with intensity measure $\lambda$} if
\begin{enumerate}
\item[(a)] for all $ B \in \mathcal{B}$, the random variable $\eta(B)$ has a Poisson distribution with mean $\lambda(B)$; 
\item[(b)] for all $m \geq 1$ and all pairwise disjoint subsets $B_1,B_2, \ldots, B_m \in \mathcal{B}$,  the random variables $\eta(B_1), \ldots, \eta(B_m)$ are independent ({\em complete independence}).
\end{enumerate}
An elementary way to construct a Poissson process with a finite intensity measure $\lambda$ on $bD$ is as follows.  Let $W^n,  n \geq 1$, be i.i.d.  random variables with distribution $\frac{\lambda(\cdot)}{\lambda(bD)}$ and $N$ be an independent Poisson random variable with mean $\lambda(bD)$.  Then $\eta = \{W^1,\ldots,W^N\}$ is a Poisson process with intensity measure $\lambda$.  We also represent $\eta$ as a counting measure, i.e.,   $\eta = \sum_{i=1}^N \delta_{W^i}.$  We always work with this explicit construction of the Poisson process for convenience. 	

{We recall the so-called difference operators \cite[Section 18.1]{last2017lectures}.  Let $z \in bD,  \phi \in \bN$ and $F : \bN \to \R$ be a measurable function.  The first and second-order difference operators are defined as}
\beas
D_zF(\phi) &:=& F(\phi \cup \{z\}) - F(\phi), \\
D^2_{z,w}F(\phi) &:=& D_z(D_wF(\phi)) = D_w(D_zF(\phi))\\
 &= &F(\phi \cup \{z,w\}) - F(\phi \cup \{w\}) - F(\phi \cup \{z\}) + F(\phi).
\eeas

Let $F : \bN \to \R$ be a measurable and square-integrable function with respect to $\eta$, i.e.,  $\BE[F(\eta)^2]  < \infty$.  From \cite[Exercise 18.8 and Theorem 18.7]{last2017lectures},  we have the following bounds on variance 
\begin{equation}
\label{e:variancebounds}
\int ( \BE[D_zF(\eta)])^2 \md \lambda(z) \leq \BV(F(\eta)) \leq \int \BE[(D_zF(\eta))^2] \md \lambda(z). 
\end{equation}
{The upper bound is the well-known Poincar\'e inequality.   We now state a normal approximation result in the Wasserstein metric \eqref{e:defndW} for $F(\eta)$.}
\begin{theorem}(Second-order Poincar\'e inequality,  \cite[Theorem 1.1]{last2016normal})
\label{t:normapproxpoi}
Let $F : \bN \to \R$ be a measurable function with $\BE[F(\eta)] = 0$ and $\BE[F(\eta)^2]  < \infty$.  Define
\begin{align*}
T_1(F) &:= \left[ \int \left( \BE[(D^2_{z^1,z^2}F)^2(D^2_{z^2,z^3}F)^2] \right)^{1/2}\left( \BE[(D_{z^1}F)^2(D_{z^2}F)^2] \right)^{1/2} \md \lambda^3(z^1,z^2,z^3) \right]^{1/2}, \\
T_2(F) &:=  \left[ \int \BE[(D^2_{z^1,z^2}F)^2(D^2_{z^1,z^3}F)^2]  \md \lambda^3(z^1,z^2,z^3)   \right]^{1/2}, \\
T_3(F) &:=  \int \BE[|D_zF|^3] \md \lambda(z),
 \end{align*}
where $\eta$ has been omitted in the argument of $F$ for convenience. Then
\be\label{e:2ndPoincare}
 d_W\left(\frac{F(\eta)}{\BV[F(\eta)]^{1/2}}, Z\right) \leq  \frac{2T_1(F) + T_2(F)}{\BV[F(\eta)]}  + \frac{T_3(F)}{\BV[F(\eta)]^{3/2}}.
\ee
\end{theorem}
We now state similar results for the Binomial process $\beta_n$.  Recall that $\beta_n = \{W^1,\ldots,W^n\}$ with $n$ i.i.d.  points sampled according to the distribution $f\sigma$ where $f$ is a probability density.  Let $F : \bN \to \R$ be a measurable function.  From the Efron--Stein inequality (\cite[Remark 2.5]{lachieze2017new}),  one can derive the following standard variance upper bound for $F(\beta_n)$ satisfying $\BE[F(\beta_n)^2] < \infty$. 
\begin{equation}
\label{e:ESsimple}
\BV(F(\beta_n)) \leq n \int \BE[(D_zF(\beta_{n-1}))^2] f(z) \md \sigma(z).
\end{equation}
We now state a normal approximation theorem for functionals of binomial process, after introducing the necessary notation.  The theorem is obtained by combining the bounds in \cite[Theorem 2.2]{chatterjee2008new} and \cite[Theorem 5.1]{lachieze2017new} (see also \cite[Remark 5.4]{lachieze2017new}).  Let $\bfW = (W^1,\ldots,W^n)$ and let $\bfW',  \bfW''$ be independent copies of $\bfW$.   We say that $\bfZ = (Z^1,\ldots,Z^n)$ is {\em a recombination of $\bfW$}  if ${Z^i} \in \{W^i,  (W')^i,  (W'')^i\}$ for each $1 \leq i \leq n$. We denote the set of possible recombinations of $\bfW$ as $Rec(\bfW)$.  {Let $\bfW^{(i)} = ( \ldots,W^{i-1},W^{i+1},\ldots), 1 \leq i \leq n$, be the vector obtained by dropping the $i$th co-ordinate from $\bfW$. Similarly, let $\bfW^{(i,j)} = (\bfW^{(i)})^{(j)},  1 \leq i \neq j \leq n$, be the vector obtained by dropping the $i$th and $j$th co-ordinates from $\bfW$. We abuse notation and set $F(\mathbf V) = F(\{V^1,\ldots,V^m\})$ for any vector $\mathbf V=(V^1,...,V^m)$.}  In other words,  we view vectors as point sets when applying $F$. For $1 \leq i \neq j  \leq n$,  define
\begin{align*}
D_iF(\bfW) &:= F(\bfW) - F(\bfW^{(i)}), \\
D_{i,j}F(\bfW) =  D_{j,i}F(\bfW) & := F(\bfW) - F(\bfW^{(i)}) - F(\bfW^{(j)}) + F(\bfW^{(i,j)}).
\end{align*}  
\begin{theorem}(\cite[Lemma 2.3]{thale2018random})
\label{t:normapproxbin}
Let $F : \bN \to \R$ be a measurable function with $\BE[F(\beta_n)] = 0$ and $\BE[F(\beta_n)^2]  < \infty$.  Define
\begin{align*}
T'_1(F) &:=\sup\left\{\BE\left[\1[D_{1,2}F(\bfY) \neq 0]D_1f(\bfZ)^2D_2f(\bfZ')^2\right]:\bfY,\bfZ,\bfZ' \in Rec(\bfW)\right\},  \\
T'_2(F) &:=   \sup\left\{\BE\left[\1[D_{1,2}F(\bfY) \neq 0, D_{1,3}F(\bfY') \neq 0]D_2f(\bfZ)^2D_2f(\bfZ')^2\right]:\bfY,\bfZ,\bfZ' \in Rec(\bfW)\right\}, \\
T'_3(F) &:= \BE[|D_1F(\bfW)|^4], \\
T'_4(F) &:=  \BE[|D_1F(\bfW)|^3].
 \end{align*}
Then, there is a finite constant $C>0$ such that  
\be\label{e:2ndPBinom} d_W\left(\frac{F(\beta_n)}{\BV[F(\beta_n)]^{1/2}}, Z\right) \leq C\frac{\sqrt{n^2T'_1(F)} + \sqrt{n^{3/2}T'_2(F)} + \sqrt{nT'_3(F)} }{\BV(F(\beta_n))} + \frac{nT'_4(F)}{\BV(F(\beta_n))^{3/2}}.
\ee
\end{theorem}

While we have  only used bounds in terms of the first and second-order difference operators,  we note that higher-order difference operators may be defined recursively and are important ingredients in the stochastic analysis of Poisson and binomial processes; see  \cite{last2017lectures,peccati2016stochastic,lachieze2017new}. 

\section{The M{\" o}bius geometry of strongly $\C$-convex domains}\label{SS:curvature}

In Section~\ref{S:GeomEst}, we often use the unitary invariance of strong $\C$-convexity and induced polyhedra to simplify certain computations. We give here the proofs of Lemmas~\ref{L:deffn} and \ref{L:Model}, which are key to such simplifications. {In fact, we prove a slight generalization of Lemma~\ref{L:deffn} to demonstrate why the class of polyhedra considered in this paper are invariant under the group of complex affine transformations, but not the full group of linear fractional transformations.} We also elaborate on the curvature function $\nu_D$ that appears as a constant in the geometric estimates. Most of these results are implicitly present in the literature, but we gather them here for easy reference.  

The transformation law in Part $(b)$ of Lemma~\ref{L:deffn} is a special case of a more general transformation law, for which we need the following definition.
 
\begin{definition} A {\em linear fractional transformation} (or {\em M{\" o}bius transformation}) on $\Cd$ is a map of the form 
		\be\label{E:LFT}
		\M_\beta:z\mapsto\left(\frac{b_{10} +b_{11}z_1+\cdots b_{1d}z_d}{b_{00}+b_{01}z_1+\cdots+ b_{0d}z_d},...,\frac{b_{d0}+b_{d1}z_1+\cdots+ b_{dd}z_d}{b_{00}+b_{01}z_1+\cdots+ b_{0d}z_d}\right),
\ee
where $\beta=(b_{jk})_{0\leq j,k\leq d}\in \GL(d+1;\C)$. The function $(z_1,...,z_d)\mapsto b_{00}+b_{01}z_1+\cdots +b_{0d}z_d$ is denoted by $\beta_0$, and $\ell_\beta=\{z\in\Cd:\beta_0(z)=0\}$ is the set of indeterminacy of $\M_\beta$. 
\end{definition}

Any affine transformation on $\Cd$ is a linear fractional transformation. Lemma~\ref{L:deffn} is thus a corollary of the following result, which is implicitly present in \cite[Proposition 4.1]{LaSt} and \cite[Theorem 3]{Bo05}.

\begin{lemma} Let $D$ and $\rho$ be as in Definition~\ref{D:geom}. 
\begin{itemize}
\item [$(a)$] For any defining function $\wt\rho$ of $D$,  $L_{\wt\rho}(z,w)=L_{\rho}(z,w)$ for all $(z,w)\in\Cd\times bD$.
\item [$(b)$] Suppose $\beta\in \GL(d+1;\C)$ such that $\overline D\cap \ell_\beta=\emptyset$. Then, for $\rho^*=\rho\circ\M_\beta^{-1}$, 
	\be\label{eq:MobiusonL}
		L_{\rho}(z,w)=\frac{||\bdy\rho^*( \M_\beta(w))||}{||\bdy\rho(w)||}
		\frac{\beta_0(z)}{\beta_0(w)}
		L_{\rho^*}(\M_\beta(z),\M_\beta(w)),\quad z\in \Cd\setminus \ell_\beta, w\in bD. 
	\ee
\end{itemize}  
\end{lemma}
\begin{proof} To prove $(a)$, note that there exists a positive $\cont^1$ function $h$ defined on some neighborhood $U$ of $b D$, such that, $\wt\rho=h\rho$. Thus, since $\rho$ vanishes on $bD$, $\bdy\wt\rho(w)=h(w)\bdy\rho(w)$ for any $w\in bD$, and the claim follows. 

For $(b)$, first observe that {$\M_\beta(w)=(\M^1_\beta(w),...,\M^d_\beta(w))=({\beta_0(w)})^{-1} \Big({\boldsymbol b+w\cdot\wt\beta^\tr}\Big)$, and 
	\bes
		(D_\C\M_\beta)(w):=\left(\partl{\M^j_\beta(w)}{z_k}\right)_{1\leq j,k\leq d}=\frac{1}{\beta_0(w)}\wt \beta
			-\frac{1}{\beta_0(w)^2}\big(\boldsymbol b^\tr+\wt \beta\cdot
w^\tr\big)\cdot (b_{01},...,b_{0d}),
	\ees
where $\boldsymbol b=(b_{10},...,b_{d0})$ and $\wt \beta$ is the principal minor $(b_{jk})_{1\leq j,k\leq d}$. Now, setting $w^*=\M_\beta(w)$ and denoting the product of two matrices $A$ and $B$ by $A\cdot B$, we get that}
	\beas
		||\bdy\rho(w)|| L_{\rho}(z,w)&=&{\bdy(\rho^*\circ \M_\beta)(w)\cdot (w-z)^\tr}\\
		&=&{\bdy\rho^*(\M_\beta(w))\cdot  (D_\C\M_\beta)(w)\cdot (w-z)^\tr}\\
		&=&{\bdy\rho^*(w^*)}\cdot
		 \left( \frac{1}{\beta_0(w)}\wt \beta-\frac{1}{\beta_0(w)^2}
			\big(\boldsymbol b^\tr+\wt \beta\cdot w^\tr\big)\cdot 
				(b_{01},...,b_{0d})\right)\cdot(w-z)^\tr\\
		&=&{{\bdy\rho^*(w^*)}\cdot\left(
		 \frac{1}{\beta_0(w)}\wt \beta \cdot(w^\tr-z^\tr)-\frac{\beta_0(w)-\beta_0(z)}{\beta_0(w)^2}
			\big(\boldsymbol b^\tr+\wt \beta\cdot w^\tr\big)\right)}\\
		&=&{\bdy\rho^*(w^*)}\cdot
		 \left(-\frac{1}{\beta_0(w)}(\boldsymbol b^\tr
				+\wt\beta \cdot z^\tr)+\frac{\beta_0(z)}{\beta_0(w)^2}
					\big(\boldsymbol b^\tr+\wt \beta \cdot w^\tr\big)\right)\\
	&=&{||\bdy\rho^*(w^*)||}\frac{\beta_0(z)}{\beta_0(w)}
	\left<\frac{\bdy\rho^*(w^*)}{||\bdy\rho^*(w^*)||},
		\M_\beta(w)-\M_\beta(z)\right>,
	\eeas
which is precisely \eqref{eq:MobiusonL}. 
\end{proof}

Next, as an analogue of the Gaussian curvature for hypersurfaces in $\rl^d$, we consider the determinant of the (self-adjoint) linear operator $\pi\circ S:\mathcal H_wM\rightarrow \mathcal H_wM$, where $S:\mathcal T_wM\rightarrow \mathcal T_wM$ is the shape operator, and $\pi:\mathcal T_wM\rightarrow  \mathcal H_wM$ is the orthogonal projection with respect to the standard Euclidean inner product on $\mathcal T_wM$ (viewed as a subspace of $\rl^{2d}$).   A more computationally convenient definition of this second-order quantity is as follows. 

\begin{definition}\label{D:curvature} Let $D\subset\Cd$ and $\rho$ be as above. For any $w\in bD$, we define the {\em complex-restricted curvature} of $bD$ at $w$ as 
	\bes
		\nu_D(w)=(-1)^{d+1}\frac{\det \mathcal Q[\rho](w)}{||\nabla \rho(w)||^{2d+2}},
	\ees
where $\mathcal Q[\rho]$ is the $(2d+2)\times(2d+2)$-matrix function
\bes
\mathcal Q[\rho]=
\left(\begin{array}{c | c|c|c}
		0 & 0 & \rho_k & 0\\
		\hline
		0 & 0 & 0 & \rho_{\, \overline k}\\
		\hline
		\rho_j& 0 & \rho_{j,k} & \rho_{j,\overline k}\\
		\hline 0 & \rho_{\,\overline j} & 	
		\rho_{\, \overline j, k} & \rho_{\, \overline j,\overline k} 
 			\end{array}\right)_{1\leq j,k\leq d},
\ees
where $\rho_j=\smpartl{\rho}{z_j}$, $\rho_{j,\overline k}=\smsecpartl{\rho}{z_j}{\zbar_k}$, for $1\leq j,k\leq d$.
\end{definition}
 
It is easy to check that $\nu_D$ is independent of the choice of defining function, a positive function for strongly $\C$-convex domains, and invariant under unitary transformations; see \cite{Gu21}. We now prove Lemma~\ref{L:Model}, which employs a standard trick in the local study of smooth hypersurfaces; see \cite[Lemma 4.1]{We77}, and the comments preceding Remark 8 in \cite{Ba16}.


\noindent{\em Proof of Lemma~\ref{L:Model}.}  Recall that $A\cdot B$ denotes the product of two matrices $A$ and $B$. First, we choose an $\eps_D>0$ so that $B_w(\eps_D)\cap bD$ is a graph over the tangent space at $w$ for all $w\in bD$. Now, after a unitary transformation, we may assume that $w=0$, $\mathcal T_0bD = \{z\in\Cd:\ima z_d=0\}$ and $\hat\eta(0)=(0,...,0,-i)$. Then, $B_0(\eps_D)\cap D=\{z\in B_0(\eps_D):\rho(z)<0\}$, where
 \bes
		\rho(z)=-y_d+\sum_{j,k=1}^{d-1} a_{j,k}z_j\zbar_k+\rea\sum_{j,k=1}^{d-1} b_{j,k}z_jz_k+\ima\sum_{j=1}^{d-1}c_jz_jx_d+ex_d^2+O||(z',x_d)||^3,
	\ees
for some $a_{j,k}, b_{j,k}, c_j\in\C$, $j,k=1,...,d-1$, $e\in\rl$, such that
	\begin{itemize}
\item $A=(a_{j,k})$ is Hermitian,
\item $B=(b_{j,k})$ is symmetric,
\item $|\rea(z '\cdot B\cdot z'^{\, \tr})|<\zbar '\cdot A\cdot z'^{\, \tr}$ for all nonzero $z'\in\C^{d-1}$. 
\end{itemize}

Now, let $U$ be a unitary matrix in $\GL(d-1;\C)$ such that  $U^*\cdot A\cdot 	U$ is diagonal, and the entries of the symmetric matrix $U^\tr\cdot B\cdot U$ are real. Then, after applying the unitary transformation $U_1=\left(\begin{smallmatrix}U^{-1} &  \boldsymbol {0}_{d-1\times 1}    \\ 0& 1 \end{smallmatrix}\right)$, we obtain that $B_0(\eps_D)\cap U_1(\Om)=\{z\in B_0(\eps_D):\rho_1(z)<0\}$ for
	\bes
		\rho_1(z)=
		-y_d+\sum_{j=1}^{d-1} \alpha_j|z_j|^2+\rea\sum_{j,k=1}^{d-1} \wt b_{j,k}z_jz_k+\ima\sum_{j=1}^{d-1}\wt c_jz_ju_d+eu_d^2+O||(z',x_d)||^3,
	\ees
where $\alpha_1,...,\alpha_{d-1}>0$ are the eigenvalues of $A$, $\wt B=(\wt b_{j,k})$ is a real symmetric matrix, and $\wt c=(\wt c_1,...,\wt c_{d-1})=(c_1,...,c_{d-1})\cdot ~U$. Finally, let $R$ be a real orthogonal matrix that diagonalizes $\wt B$. Letting  $U_w=\left(\begin{smallmatrix}R^{-1} &  \boldsymbol {0}_{d-1\times 1}    \\ 0& 1 \end{smallmatrix}\right)\cdot U_1$, we get that $B_0(\eps_D)\cap U_w(\Om)=\{\rho_2(z)<0\}$ for 
	\bes
		\rho_2(z)=
		-y_d+\sum_{j,k=1}^{d-1} (\alpha_j|z_j|^2+\beta_j(\rea z_j^2))+\ima\sum_{j=1}^{d-1}\gamma_jz_jx_d+ex_d^2+O||(z',x_d)||^3,
	\ees
where $\beta_1,...,\beta_{d-1}$ are the eigenvalues of $\wt B$, and $(\gam_1,...,\gam_{d-1})=\wt c\cdot R$. Since unitary transformations preserve strong $\C$-convexity, we have that $0\leq \beta_j<\alpha_j$ for all $j=1,...,d-1$. 

Now, by the unitary invariance of $\nu_D$, $\nu_D(w)=\nu_{B_0(\eps_D)\cap U_w(D)}(0)$, and a direct computation shows that the latter is equal to $\frac{1}{16}\prod_{j=1}^{d-1}(\alpha_j^2-\beta_j^2)$.
\qed

\section{Area and volumes estimates for model domains}\label{SS:Models} In the proof of Theorem~\ref{T:CutsCapsGeneral}, the first-order terms in the estimates come from the model domains
	\bes
D_{\balp,\bbet}:=\left\{(z',x_d+iy_d)\in\Cd:y_d>\sum_{j=1}^{d-1}\alpha_j|z_j|^2+	\beta_j\rea(z_j^2)\right\},
	\ees
where $\balp=(\alpha_1,...\alpha_{d-1})$ and $\bbet=(\beta_1,...,\beta_{d-1})$ satisfy $0\leq \beta_j<\alpha_j$, $j=1,...,d-1$. The computations for these domains is carried out here. Recall that the dependence of cuts, tubular cuts, caps and visibility regions on the underlying domain is indicated via a superscript, as in $C^D(w;\de)$, $S^D(w;\de)$, etc., and the projection $\pi_{2d-1}(C)$ of $C\subset\Cd$ onto $\C^{d-1} \times \R$ is denoted by $\wt C$. 

\begin{theorem}\label{T:modelcomp}
Let $D_{\balp,\bbet}$ be as above, and $\de>0$. Then, there exists a $C_{\balp,\bbet}>0$ such that
	\bea
		&&\lambda\big(C(0;\de)\big)=
		 \dfrac{h_{d+1}\,\kappa_{2d-2}}{d\,\sqrt{v_{\balp,\bbet}}}\,\de^{d+1},
\label{E:modelcut}\\
		&&\left|\sigma\big(S(0;\de)\big)- \dfrac{h_{d}\kappa_{2d-2}}{\sqrt{v_{\balp,\bbet}}}\,\de^{d}\right|\leq C_{\alpha,\bbet}\:\de^{d+1}, \label{E:modelcap}
	\eea
where $v_{\balp,\bbet}=\prod\limits_{j=1}^{d-1}(\alpha^2_j-\beta_j^2)$, and $\kappa_d$ and $h_d$ are as in Conjecture~\ref{conj}. 

\end{theorem}

\begin{proof}First, we assume that $\balp=\boldsymbol 1:=(1,...,1)$, and denote $D_{\balp,\bbet}$ by $D_\bbet$. We claim that
\bea
\lambda\left(C(0;\de)\right)&=&
\dfrac{h_{d+1}\,\kappa_{2d-2}}{d\,\sqrt{v_{\boldsymbol 1,\bbet}}}\,\de^{d+1},\\
\lambda_{2d-1}\left(\wt C(0;\de)\right)&=&
\dfrac{h_{d}\kappa_{2d-2}}{\sqrt{v_{\boldsymbol 1,\bbet}}}\,\de^{d},\label{E:projcut}
\eea
for any $\de>0$. 
To prove this, we set $f_\bbet(z',x_d)=\sum_{j=1}^{d-1}(|z_j|^2+\beta_j\rea z_j^2)$, and observe that 
\beas
	C(0;\de)&=&\left\{(z',x_d+iy_d)\in\Cd: (z',x_d)\in\wt C(0;\de),\ 
		f_\bbet(z',x_d) < y_d < \sqrt{\de^2-x_d^2} \right\},\\		
		\wt C(0;\de)&=&\left\{(z',x_d)\in\C^{d-1}\times\rl:
			f_\bbet(z',x_d)^2
				+x_d^2<\de^2\right\}.
\eeas
Using spherical coordinates on each slice $\left\{\sum_{j=1}^{d-1}(|z_j|^2+\beta_j\rea z_j^2)=r^2\right\}$, $r>0$, we have that 
	\beas
		\lambda\left(C(0;\de)\right)&=&\int_{\wt C(0;\de)}\sqrt{\de^2-x_d^2}
			-\sum_{j=1}^{d-1}(|z_j|^2+\beta_j\rea z_j^2)\, d\lambda_{2d-1}(z',x_d)\\
		&=&
\frac{(2d-2)\kappa_{2d-2}}{\sqrt{1-\bbet^2}}\int\limits_{-\de}^{\de}
		\int\limits_{0}^{(\de^2-x^2_d)^{\frac{1}{4}}}
			\left(\sqrt{\de^2-x_d^2}-r^2\right) r^{2d-3}\,
		dr\,dx_d =\dfrac{h_{d+1}\,\kappa_{2d-2}}{d\sqrt{v_{\boldsymbol 1,\bbet}}}\,\de^{d+1},\\
		\lambda_{2d-1} \left(\wt C(0;\de)\right)&=&
		\int_{\wt C(0;\de)} d\lambda_{2d-1}(z',x_d)
			=\frac{(2d-2)\kappa_{2d-2}}{\sqrt{1-\bbet^2}}\int\limits_{-\de}^{\de}
		\int\limits_{0}^{(\de^2-x^2_d)^{\frac{1}{4}}} r^{2d-3}\, dr\,dx_d
			= \frac{h_d\kappa_{2d-2}}{\sqrt{v_{\boldsymbol 1,\bbet}}}\de^{d}.
	\eeas
Thus, we have \eqref{E:modelcut} in the special case of $D_\bbet$.

To prove \eqref{E:modelcap} for $D_\bbet$, we observe that $\wt S(0;\de)=\wt C(0;\de)$, and for $(z',x_d)\in\wt C(0;\de)$, 
	\be\label{eq:grad_est}
		||\nabla f_\bbet(z',x_d)||^2=4\sum_{j=1}^{d-1}|z_j+\beta_j\zbar_j|^2
				\leq 2\sum_{j=1}^{d-1}(|z_j|^2+\beta_j\rea \zbar_j^2)
					<2\de.
	\ee
Thus, since $S(0;\de)$ is the graph of $f_\bbet$ over $\wt C(0;\de)$, we have that
\beas
		\sigma\big(S(0;\de)\big)=\int_{\wt C(0;\de)}
	\left(1+||\nabla f_\bbet(z',x_d)||^2\right)^{\frac{1}{2}}
		\, d\lambda_{2d-1}(z',x_d),
\eeas
and, therefore, 
\bes
\left|\sigma\left(S(0;\de)\right)-\lambda_{2d-1}\left(\wt C(0;\de)\right)\right|\leq \de\lambda_{2d-1}\left(\wt C(0;\de)\right)\ees
for all $\de>0$. Equation \eqref{E:modelcap} now follows from \eqref{E:projcut} in the special case of $D_\bbet$. 

For the general case, consider the map $A:(z_1,...,z_d)\mapsto (\sqrt{\alpha_1}z_1,...,\sqrt{\alpha_{d-1}}z_{d-1},z_d)$, which biholomorphically maps $D_{\balp,\bbet}$ onto $D_{\bbet/\balp}$, where
$\bbet/\balp=(\beta_1/\alpha_1,...,\beta_{d-1}/\alpha_{d-1})$. Now, invoking the transformation law for cuts and caps (see Lemma~\ref{L:deffn}), and noting that $A$ and $\left(A|_{\C^{d-1}\times\rl}\right)$ are diagonal transformations, we obtain that for all $\de>0$,
	\beas	
	&&\lambda\big(C^{D_{\balp,\bbet}}(0;\de)\big)
	=\left(\prod_{j=1}^{d-1}\alpha_j^{-1}\right)\lambda\big(C^{D_{\bbet/\balp}}(0;\de)\big),\\
	&& \left|\sigma(S^{D_{\balp,\bbet}}(0;\de))-\left(\prod_{j=1}^{d-1}\alpha_j^{-1}\right)
	\lambda_{2d-1}\big(S^{D_{\bbet/\balp}}(0;\de)\big)\right|\leq
\de\left(\prod_{j=1}^{d-1}\alpha_j^{-1}\right)\lambda_{2d-1}\big(S^{D_{\bbet/\balp}}(0;\de)\big). 
	\eeas
\end{proof}

\begin{theorem}\label{T:modelcomp2} Let $D_{\balp,\bbet}$ be as above, $\de>0$, and $t\in[0,1)$.Then, there is a $C_{\balp,\bbet}>0$ such that for $\de<\frac{1}{2\max_j\alpha_j}$,
	\beas
	 \frac{\kappa_{2d-2}}{\sqrt{v_{\balp,\bbet}}}h_d(t)\de^d\leq \sigma(G_{t\de}(0;\de))\leq 
		 \frac{\kappa_{2d-2}}{\sqrt{v_{\balp,\bbet}}}h_d(t)\de^d+C_{\balp,\bbet}\de^{d+1},
	\eeas
where $h_d(t)=2\int_0^{\sin^{-1}\sqrt{1-t^2}}(\cos\theta-t)^{d-1}\cos\theta\, d\theta$. 
\end{theorem}
\begin{proof} Since $G_{t\de}(0;\de)\subset bD$, it is the graph of $f_{\balp,\bbet}$ over  $\wt G_{t\de}(0;\de)$, which is given by 
	\beas
\left\{(z',x_d)\in \C^{d-1}\times\rl:(t\de+q_1(z'))^2+\left(x_d+q_2(z')\right)^2<\de^2(1+4q_3(z'))\right\},
	\eeas
where $q_1(z')~=~\sum_{j=1}^{d-1}\left(\alpha_j|z_j|^2+\beta_j\rea\zbar_j^2\right)$, $q_2(z')=2\sum_{j=1}^{d-1}\beta_j\ima z_j^2$, and $q_3(z')= \sum_{j=1}^{d-1}|\alpha_jz_j+\beta_j\zbar_j|^2$. Note that $q_1,q_2,q_3\geq 0$. 

Now, setting $A=\max_j \alpha_j$ and recalling that $\beta_j<\alpha_j$, $j=1,...,d-1$, we have that $q_3(w')\leq 2Aq_1(w')$. Thus, if $z\in G_{t\de}(0;\de)$,
	\bes
		q_3(z')^2\leq 4A^2\de^2(1+4q_3(z')),
	\ees
i.e., $q_3(z')\leq 2A\de+16A^2\de^2$. Thus, for $\de<1/(2A)$, $K_{t\de,\de}	\subseteq\wt G_{t\de}(0;\de)
			\subseteq K_{t\de,\de(1+8A\de)}$,
where,
	\bes
		K_{\eps,\eta}=\left\{(z',x_d)\in\C^{d-1}\times\rl:
		(\eps+q_1(z'))^2+\left(x_d+q_2(z')\right)^2<\eta^2\right\},\quad 0<\eps<\eta.
	\ees
Consequently, for $\de<1/(2A)$, 
	\be\label{eq:areabounds}
		\int\limits_{K_{t\de,\de}}d\lambda_{2d-1}(z',x_d)\leq 
	\int\limits_{\wt G_{t\de}(0;\de)}\sqrt{1+4q_3(w')}\,d\lambda_{2d-1}(z',x_d)
		\leq\int\limits_{K_{t\de,\de+8A\de^2}} (1+8A\de)d\lambda_{2d-1}(z',x_d),
	\ee
where the term in the center is $\sigma\big(G_{t\de}(0;\de)\big)$. It now suffices to compute $\lambda_{2d-1}(K_{\eps,\eta}\big)$. Now, setting 
	$\xi_j=\sqrt{\alpha_j+\beta_j}\, \rea z_j+i\sqrt{\alpha_j-\beta_j}\, \ima z_j$, $1\leq j\leq d-1$, and $\tau=x_d+2\sum_{j=1}^{d-1}\beta_j\ima w_j^2$,
we obtain that 
\beas
	\lambda_{2d-1}(K_{\eps,\eta})&=&\frac{1}{v_{\balp,\bbet}}\lambda_{2d-1}
		\left\{(\xi,\tau)\in\C^{d-1}\times\rl:(\eps+||\xi||^2)^2+\tau^2<\eta^2
			\right\}\\
	&=&\frac{2\kappa_{2d-2}}{\sqrt{v_{\balp,\bbet}}}\int\limits_{0}^{\sqrt{\eta^2-\eps^2}}
			(\sqrt{\eta^2-\tau^2}-\eps)^{d-1}\,d\tau
	=\frac{\kappa_{2d-2}}{\sqrt{v_{\balp,\bbet}}}\,\eta^{d} h_d\left(\frac{\eps}{\eta}\right),
	\eeas
where $h_d(t)=2\int_0^{\sin^{-1}\sqrt{1-t^2}}(\cos\theta-t)^{d-1}\cos\theta \,d\theta$. 

A few comments about $h_d$ are in order. Note that $h_1(t)=2\sqrt{1-t^2}$, $t\in[0,1]$. When $d\geq 2$, $h_d$ is a decreasing Lipschitz function on $[0,1]$ with $h_d(1)=0$. This follows from observing that $h'_d=-(d-1)h_{d-1}(t)$, which is negative and bounded (in absolute value) on $(0,1)$. Thus, for any $t\in(0,1)$, there is an $r\in(t(1+8A\de)^{-1},t)$ such that
	\bes
		h_d\left(\frac{t\de}{\de+8A\de^2}\right)
		=h_d(t)-h_d'(r)\frac{8At\de}{1+8A\de}
		\leq h_d(t)+(d-1)h_{d-1}(0)\frac{8At\de}{1+8A\de}.
	\ees
Combining this with \eqref{eq:areabounds} gives that
	\beas
		\frac{\kappa_{2d-2}}{\sqrt{v_{\balp,\bbet}}} h_d(t) \de^{d}\leq
			 \sigma\big(G_{t\de}(0;\de)\big)&\leq&			
	\frac{\kappa_{2d-2}}{\sqrt{v_{\balp,\bbet}}} \de^{d}(1+8A\de)^d\left(h_d(t)+(d-1)h_{d-1}(0)\frac{8At\de}{1+8A\de}\right),
		\eeas
which yields our claim. 
\end{proof}
\section*{Acknowledgements}
SA's research was partially supported by CPDA from the Indian Statistical Institute and Infosys Chair Professor position at the Chennai Mathematical Institute, Chennai.  PG was partially supported by the Infosys Young Investigator Award. DY's research was partially supported by SERB-MATRICS Grant MTR/2020/000470 and CPDA from the Indian Statistical Institute.   DY is thankful to Matthias Reitzner for comments on the literature of random polytopes and to Srikanth Iyer for discussions related to Conjecture \ref{conj}.  DY would also like to thank Joe Yukich for very helpful comments on a first draft of the article.  

\bibliographystyle{plainnat}
\bibliography{random_poly_rc-spy}

\end{document}